\documentclass[12pt]{amsart}
\usepackage{graphicx,color}
\usepackage{colordvi}
\usepackage{amssymb, mathrsfs, amsfonts, amsmath}
\usepackage{latexsym}
\usepackage{amsthm,amsfonts}
\usepackage[all]{xy}
\usepackage{enumitem}

\newtheorem{theorem}{Theorem}[section]
\newtheorem{lemma}[theorem]{Lemma}
\newtheorem{corollary}[theorem]{Corollary}
\newtheorem{proposition}[theorem]{Proposition}
\newtheorem{definition}[theorem]{Definition}
\newtheorem{example}[theorem]{Example}
\newtheorem{remark}[theorem]{Remark}
\newtheorem{claim}[theorem]{Claim}

\newcommand{\R}{\mathbb R}
\newcommand{\Q}{\mathbb Q}

\begin{document}

\title[ Ambient Lipschitz Geometry of Normally Embedded Surface Germs
 ]{Ambient Lipschitz Geometry of Normally Embedded Surface Germs}

\author[]{Lev Birbrair}
\address{Departamento de Matem\'atica, Universidade Federal do Cear\'a
(UFC), Campus do Pici, Bloco 914, Cep. 60455-760. Fortaleza-Ce,
Brasil}
\address{Department of Mathematics, Jagiellonian University, Profesora Stanisława Łojasiewicza 6, 30-348, Kraków, Poland}
\email{lev.birbrair@gmail.com}

\author[]{Davi Lopes Medeiros}
\address{Departamento de Matem\'atica, Universidade Federal do Cear\'a
	(UFC), Campus do Pici, Bloco 914, Cep. 60455-760. Fortaleza-Ce,
	Brasil}
\email{profdavilopes@gmail.com}

\thanks{L. Birbrair: Research supported under CNPq 302056/2022-0 grant and by U1U/W16/NO/01.01 grant by Jagiellonean University}
\thanks{D. L. Medeiros: Research supported by the Coordenação de Aperfeiçoamento de Pessoal de Nível Superior – Brasil (CAPES) – Finance Code 001, and by the Serrapilheira Institute (grant number Serra– R-2110-39576)}

\date{\today}

%\thanks{The authors were partially supported by CNPq-Brazil}
\begin{abstract}
We study the ambient Lipschitz geometry of semialgebraic surfaces. It was discovered in \cite{BBG} that ambient Lipschitz geometry is different from the outer Lipschitz geometry.  We show that two surface germs in $\R^3$, Lipschitz normally embedded and with isolated singularity, are ambient bi-Lipschitz equivalent if, and only if, they are outer bi-Lipschitz equivalent and ambient topologically equivalent.
\end{abstract}

\maketitle

\tableofcontents

\section{Introduction}\label{sec:intro}

Lipschitz geometry of singularities is an intensively developing part of singularity theory. The classification of singular sets up to bi-Lipschitz equivalence is an interesting research topic, since it is closely related to topological, differential and analytical equivalences.

There are two natural metrics on a connected set $X\subset\R^n$: the \emph{inner metric}, where the distance between two points in $X$ is the length of the shortest path in $X$ connecting these points, and the \emph{outer metric}, where the distance between two points of $X$ is their euclidean distance in $\R^n$. A set $X$ is called \emph{Lipschitz normally embedded} (see \cite{birbrair2000normal}) if its inner and outer metrics are equivalent.

There are three natural equivalence relations associated with those metrics. Two sets $X \subset \R^m$ and $Y \subset \R^n$ are \emph{inner} (resp., \emph{outer}) \emph{bi-Lipschitz equivalent} if there is an inner (resp., outer) bi-Lipschitz homeomorphism $h:X\to Y$. The sets $X$ and $Y$ are \emph{ambient bi-Lipschitz equivalent} if the homeomorphism $h:X\to Y$ can be extended to a bi-Lipschitz homeomorphism $H: \R^m \to \R^n$ on the ambient space. Ambient equivalence is stronger than  outer equivalence, and  outer equivalence is stronger than  inner equivalence.  The finiteness theorems of Mostowski \cite{Mostowski} and Valette \cite{valette2005Lip} show that there are finitely many ambient bi-Lipschitz equivalence classes in any definable family.

The most recent contribution to Lipschitz geometry is the so-called metric knot theory.   Birbrair and Gabrielov \cite{BG} studied the relationship between ambient and outer bi-Lipschitz equivalence. This issue is relevant when two semialgebraic set germs have the same ambient topology. Surprisingly, they found infinitely many  semialgebraic surface singularity germs in $\R^3$, with the same ambient topology and outer metric structure, but which are not ambient bi-Lipschitz equivalent. In $\R^4$ the problem is even richer and closely related to the classical knot theory. Birbrair, Gabrielov and Brandenbursky \cite{BBG} found another result, relating surface singularity germs in  $\R^4$ and knots: For any knot $K$ one can construct a surface singularity germ $X_K$ such that\\
i) the link of $X_K$ at the origin is a trivial knot;\\
ii) all germs $X_K$ are outer bi-Lipschitz equivalent;\\
iii) $X_{K}$ and $X_{L}$ are ambient bi-Lipschitz equivalent only if the knots $K$ and $L$ are isotopic.

On the other hand,   Birbrair, Fernandes and Jelonek \cite{extention} proved that if two semialgebraic sets $X_1$ and $X_2$ in $\mathbb{R}^{2k+1}$ with codimension at least $k+1$ are outer bi-Lipschitz equivalent, then they are ambient bi-Lipschitz equivalent.

The main result of this paper is the following. If $X\subset \R^3$ is a germ of a real semialgebraic LNE surface, with isolated singularity, then the ambient Lipschitz geometry of $X$ is completely determined by the topology and outer Lipschitz geometry. In other words, two germs $X, Y \subset \R^3$ are ambient Lipschitz equivalent if they are ambient topologically equivalent and outer Lipschitz equivalent. Notice that the majority of the well known invariants in Lipschitz geometry are created in order to show that two given surfaces are NOT bi-Lipschitz equivalent (inner, outer or ambient). However, to establish whether two given semialgebraic germs are bi-Lipschitz equivalent is, in general, a much more difficult problem. For this purpose one must construct a ``vocabulary'' in order to classify such sets. Birbrair and Gabrielov  \cite{BG}  conjectured that Lipschitz normally embedded semialgebraic surfaces germs are ambient bi-Lipschitz equivalent if and only if they are outer bi-Lipschitz equivalent. The present paper proves this conjecture for germs of semialgebraic surfaces in $\R^3$ with isolated singularity. The proof is essentially constructive, as we build an ambient bi-Lipschitz map mainly in an effective way.

The paper is structured as follows. Section \ref{preliminaries} is devoted to the basic definitions and the known facts in Lipschitz geometry of surfaces. In Section 2.1 we define the notions of Lipschitz normal embedding (LNE), H\"older triangles, curvilinear triangles and horns. In Section 2.2 we recall the theorem of topological local conical structure. We also introduce Proposition \ref{colagem-bi-lip-inner}, which allows one to obtain an inner bi-Lipschitz map by gluing other bi-Lipschitz maps. These results are introduced in order to make the exposition self-contained.

In Section 3 we develop some technical lemmas in order to construct ambient bi-Lipschitz maps. Usually the link of a semialgebraic set at a point is an intersection
with a small sphere. However, in Section 3.1 we consider the plane link, which is the intersection of the set with a plane,
close to the singular point. We prove that for any semialgebraic surface there exists an ambient bi-Lipschitz homeomorphism, sending a germ of a surface into a cone with  vertex at this point and sending the usual links into their respective plane links.

In Section 3.2 we start to construct a ``vocabulary'' of bi-Lipschitz maps. We define a so-called translation map along an arc and a function dilatation map. The translation and dilatation maps move a germ $(X,x_0)$ into a cone with the same vertex $x_0$, leaving the  cone invariant. The translation map along an arc moves each plane link of $X$ by the vector corresponding to the intersection of the arc with this link. The dilatation map with respect to a function $f$ dilates the plane link of $X$ at distance $t$ from $x_0$ by the value $f(x)$. These maps are examples of semialgebraic ambient bi-Lipschitz maps and they will be used in the proof of the general result; Figure \ref{fig32} shows how those maps move the germ.

The idea of the proof of the main result is the following. One can divide the germ of the surface into a family of so-called synchronized H\"older triangles. For each of those triangles one can construct a semialgebraic ambient bi-Lipschitz map, defined in $\R^3$, such that the image of the triangle becomes linear on each plane link (we call them linear triangles). Then we obtain a germ, constructed as a union of linear triangles. The important part of the paper is a simplification of such germs. The main goal is to obtain a surface such that its plane link is either a segment or a union of three segments. Then one can prove that the germ of the surface is ambient bi-Lipschitz equivalent to a H\"older triangle in the first case, and to a horn in the second case.

This idea is realized in Sections 4--11. Section 4 is devoted to establishing the fundamental definitions and properties of synchronized triangles. A synchronized triangle is a H\"older triangle, defined as the graph of a semialgebraic function over a plane, passing through the singular point. This semialgebraic function can be considered as a family of semialgebraic functions, defined on a family of segments. A synchronized triangle is called $M$-bounded if all the derivatives of all the functions from the functions family are bounded by a constant $M$.

In Section 5 we prove that for any surface germ one can construct a so-called convex decomposition. It decomposes the surface into synchronized triangles, each generated by a family of convex functions.  Section 6 is devoted to the so-called ambient isotopy in curvilinear rectangles, which is a way to continually deform synchronized triangles, preserving the ambient Lipschitz structure. The technique used in this paper is closely related to \cite{birbrair2007K}.

In Section 7.1 we define linear triangles and their fundamental elements. Then in Section 7.2 we study the so-called kneadable triangles. Kneadable triangles are H\"older triangles where one can construct an ambient bi-Lipschitz isotopy between the identity map and the linear triangle, determined by the boundary arcs of the triangle (Figure \ref{fig33} explains this construction). In Section 8 we prove that every LNE surface germ can be decomposed into finitely many kneadable triangles. The LNE condition is essential in this step -- otherwise one cannot produce an ambient bi-Lipschitz map, transforming the surface into a finite union of linear triangles.

In Section 9 we study polygonal surfaces, i.e.~LNE surfaces such that the plane links are polygons. Such surfaces can be considered as unions of linear H\"older triangles, corresponding to the edges of the sections. Each such edge vanishes at the singular point with a rate. It was proved in \cite{birbrair1999} that the union of two H\"older triangles with different rates is inner bi-Lipschitz equivalent to a H\"older triangle with minimal vanishing order. We prove that if the union of the triangles is Lipschitz normally embedded, then the statement remains true for the ambient Lipschitz equivalence and for linear triangles. That is why the number of linear triangles with non-minimal order can be reduced. This way we prove that such surfaces are ambient bi-Lipschitz equivalent to surfaces where all the vanishing rates of all the edges are equal.

A polygonal surface may have a link homeomorphic to a segment, or  a link homeomorphic to a circle. Notice that we may suppose that all the triangles have the same order. In both cases the goal in Section 10 is to reduce the number of such triangles. We show that one can achieve this reduction by using an ambient bi-Lipschitz map restricted to a suitable region obtained from a triangulation of the link. Finally, we show that if the link is homeomorphic to a segment then the germ is ambient bi-Lipschitz equivalent to a H\"older triangle, and if the link is homeomorphic to a circle, then the corresponding reduction procedure makes the link  triangular, which is ambient bi-Lipschitz equivalent to a horn (or to a cone if the exponent is~1).

Finally, in Section 11 we prove the main result. For this purpose, we define an extended canonical tree, a combinatorial object that provides information about the inner Lipschitz structure of each connected component of the link, and the ambient topology of the surface. Then we prove that such extended canonical trees describe completely the ambient Lipschitz geometry of LNE surface germs with isolated singularity. In the end, we exhibit an example of two LNE surface germs in $\R^3$, with non-isolated singularity, that are outer bi-Lipschitz equivalent and topologically equivalent, but are not ambient bi-Lipschitz equivalent, showing that the main result is the best possible in $\R^3$.

We would like to thank Alexandre Fernandes, Edson Sampaio, \mbox{Maciej} Denkowski and Rodrigo Mendes for their interest in this work and for valuable conversations and suggestions.

\newpage

\section*{Notations List}\label{sec:notations}

\begin{itemize}
	\item $\|(x_1,...,x_n)\|=(x_1^2+...+x_n^2)^{\frac{1}{2}}$.
	
	\item $\{\gamma_u(t)\}_{0\le u \le 1}$: arc coordinate system of a synchronized triangle (see Definition \ref{arc-coord}).
	
	\item $\{\gamma_{u,v}(t)\}_{(u,v) \in [0,1]^2}$: arc coordinate system of a curvilinear rectangle (see Definition \ref{arc-coord-curv-rectangle}).
	
	\item $\overline{\gamma_1 \gamma_2}$: Linear triangle determined by $\gamma_1$ and $\gamma_2$ (see Definition \ref{linear-triangle}).
	
	\item $\overrightarrow{\gamma_1 \gamma_2}(t) = \frac{\gamma_2(t) - \gamma_1 (t)}{\| \gamma_2(t) - \gamma_1 (t) \|}$ (see Definition \ref{linear-triangle}).
	
	\item $\overrightarrow{\gamma_1 \gamma_2}(t) = \frac{\gamma_2(t) - \gamma_1 (t)}{\| \gamma_2(t) - \gamma_1 (t) \|}$ (see Definition \ref{linear-triangle}).
	
	\item $\angle \gamma_1 \gamma_2 \gamma_3 (t)$: angle formed by $\overrightarrow{\gamma_2 \gamma_1} (t)$ and $\overrightarrow{\gamma_2 \gamma_3}(t)$ (see Definition \ref{linear-triangle}). 
	
	\item $\angle \gamma_1 \gamma_2 \gamma_3$: angle formed by $\overrightarrow{\gamma_2 \gamma_1}$ and $\overrightarrow{\gamma_2 \gamma_3}$ (see Definition \ref{linear-triangle}).
	
	\item $(\gamma_1 \dots \gamma_n)$: (open or closed) polygonal surface germ (see Definition \ref{poly-surface-1}).
	
	\item $\mathbb{B}_r^{n}(p)=\{x\in \R^n \, ; \, \|x-p\|<r\}$.
	
	\item $C_{a}^{n+1}$, $C_{a}^{n+1}(t), C_{a}^{n+1}[t]$: see Definitions \ref{cone} and \ref{link-plano}.
	
	\item $d$, $d_{X}$: inner and outer metric on $X$ (see Definition \ref{Outer-metric}).
	
	\item $int(X)$, $\partial X$: set of interior points and set of boundary points of $X$, respectively.	
	
	\item $J_{\varphi}$: Jacobian of the map $\varphi: \R^m \to \R^n$.
	
	\item $\mathbb{S}_r^{n-1}(p)=\{x\in \R^n; \|x-p\|=r\}$.
	
	\item $\mathbb{S}_r^{n-1}=\mathbb{S}_r^{n-1}(0)$; $\mathbb{S}^{n-1}=\mathbb{S}_1^{n-1}(0)$.
	
	\item $sing(X)$: set of singular points of $X$ (see Remark \ref{Rem:smooth}).
	
	\item $T(\gamma_{1},\gamma_{2})$: H\"older triangle with boundary arcs $\gamma_1, \gamma_2$ (see Definition \ref{Holder-triangle}).
	
	\item $tord$ and $tord_X$: tangency order on outer and inner metric (see Definition \ref{tord}).
	
	\item $U_{a}^{n+1}$, $U_{a}^{n+1}(t)$: see Definition \ref{cone}.	
	
	\item $V(X)$: Valette link of $X$ (see Definition \ref{arc}).
	
	\item $X_t = X \cap \mathbb{S}_{t}^{n-1}$.
	
	\item $X(t)=X\cap C_a^{n+1}(t)$ and $X[t]=X\cap C_a^{n+1}[t]$: see Definition \ref{link-plano}.
	
	\item Let $f,g\colon (0,\varepsilon)\to [0,+\infty)$ be functions. We write $f\lesssim g$ if there is a constant $f(t)\leq C g(t)$ for all $t\in (0,\varepsilon)$. We write $f\approx g$ if $f\lesssim g$ and $g\lesssim f$. We write $f\ll g$ if $\lim\limits_{t\to 0^+} \frac{f(t)}{g(t)}=0$.
	
	\item Ambient bi-Lipschitz isotopy/isotopic: See Definitions \ref{amb-isotopy} and \ref{amb-isotopy-equiv}.	
	
	\item Arc: see Definition \ref{arc}.
	
	\item Bi-Lipschitz equivalent: see Definition \ref{outer-lip}.
	
	\item Bi-Lipschitz map: see Definitions \ref{Lip} and \ref{Inner-metric}.
	
	\item Canonical tree: see Definition \ref{canonical-tree}.	
	
	\item Curvilinear rectangle/region: see Definition \ref{curvilinear-rectangle}.

	\item Curvilinear triangle: see Definition \ref{curvilinear}.

	\item Extended canonical tree: see Definition \ref{extended-canonical}.
	
	\item Generating functions family: see Definition \ref{synch-triangle-def}.	
	
	\item H\"older triangle, Horn: see Definitions \ref{Holder-triangle} and \ref{horn}.
	
	\item Kneading envelope: see Definiton \ref{kneading-envelope} and Remark \ref{kneading-envelope-remark}.	
	
	\item Kneadable triangle in a supporting envelope: see Definitions \ref{kneadable} and \ref{supporting-envelope}.
	
	\item LNE, Lipschitz Normally Embedded: see Definition \ref{LNE}.
	
	\item Separating family of cones: see Definition \ref{separating-cones}.
	
	\item Synchronized/Convex Decomposition: see Propositions \ref {synch-decomp} and \ref{convex-decomp}.
	
	\item Synchronized triangle: see Definition \ref{synch-triangle-def}.
	
	\item Synchronized triangles aligned on boundary arcs: see Definition \ref{curvilinear-rectangle}.
\end{itemize}

\section{Preliminaries}\label{preliminaries}

\subsection{Basic Definitions}\label{BasicDefinition}

\begin{definition}\label{Lip}\normalfont
	Given two metric spaces $(X_1, d_1)$ and $(X_2, d_2)$, we say that a map $\varphi: X_1 \to X_2$ is \textit{bi-Lipschitz for the metrics $d_1$ and $d_2$} (or bi-Lipschitz, for short) if there is a real number $C \ge 1$ such that
	\begin{equation*}
		\frac{1}{C} \cdot d_1(p, q) \leq d_2(\varphi(p), \varphi(q)) \leq C \cdot d_1(p, q), \quad \forall p, q \in X_1.
	\end{equation*}
	For each $C \ge 1$ that satisfies such condition, we say that $\varphi$ is $C$-bi-Lipschitz. Furthermore, we say that the metric spaces $(X_1, d_1)$ and $(X_2, d_2)$ are \textit{bi-Lipschitz equivalent} if there is a bi-Lipschitz map $\varphi: X_1 \to X_2$.
\end{definition}

\begin{definition}\label{Outer-metric}\normalfont
	Given $X \subset \mathbb{R}^{n}$, we define the following two metrics in $X$:
	\begin{itemize}
		\item \textit{outer metric of $X$}: we define the outer metric of $X$ as the distance $d: X \times X \to \mathbb{R}_{+}$, given by $d(x, y) = \|x - y\|$, for all $x, y \in X$;
		\item \textit{inner metric of $X$}: if $X$ is path-connected, we define the inner metric of $X$ as the distance $d_{X}: X \times X \to \mathbb{R}_{+}$, given by $d_X(x, y) = \inf{\{ l(\alpha) \}}$, for all $x, y \in X$. The infimum is taken over all rectifiable paths $\alpha \subset X$ from $x$ to $y$, and $l(\alpha)$ is the length of $\alpha$ (if such a rectifiable path $\alpha$ does not exist, we define $d_X(x, y) = \infty$).
	\end{itemize}
\end{definition}

\begin{definition}\label{Inner-metric}\normalfont
	Given two sets $X_1 \subset \mathbb{R}^{m}$, $X_2 \subset \mathbb{R}^{n}$, we define the following bi-Lipschitz maps from $X_1$ to $X_2$:
	\begin{itemize}
		\item \textit{outer bi-Lipschitz map}: an outer bi-Lipschitz map between $X_1$ and $X_2$ is a map that is bi-Lipschitz for the outer metrics of $X_1$ and $X_2$;
		\item \textit{inner bi-Lipschitz map}: an inner bi-Lipschitz map between $X_1$ and $X_2$ (assuming $X_1$ and $X_2$ are path-connected) is a map that is bi-Lipschitz for the inner metrics of $X_1$ and $X_2$.
		\item \textit{ambient bi-Lipschitz map} an ambient bi-Lipschitz map between $X_1$ and $X_2$ is a bi-Lipschitz map $\varphi : \mathbb{R}^{m} \to \mathbb{R}^{n}$ with respect to the outer metric, such that $\varphi (X_1) = X_2$. Notice that if such an ambient bi-Lipschitz map exists, then $m=n$.
	\end{itemize}
\end{definition}

\begin{definition}\label{outer-lip}\normalfont
	Given two sets $X_1 \subset \mathbb{R}^m$, $X_2 \subset \mathbb{R}^n$, we say that $X_1$ and $X_2$ are \textit{outer bi-Lipschitz equivalent} (resp. inner, ambient) if there is an outer bi-Lipschitz (resp. inner, ambient) map between $X_1$ and $X_2$. Given $p\in X_1$, $q\in X_2$, we say that two germs $(X_1,p)$, $(X_2,q)$ are \textit{outer bi-Lipschitz equivalent} (resp. inner, ambient) if there are neighborhoods $U$ of $p$ and $V$ of $q$ and an outer bi-Lipschitz (resp. inner, ambient) map between $X_1\cap U$ and $X_2 \cap V$.
\end{definition}

\begin{remark}\label{amb-out-inn}\normalfont
	If $X_1$ and $X_2$ are ambient bi-Lipschitz equivalent, then $X_1$ and $X_2$ are outer bi-Lipschitz equivalent, and if $X_1$ and $X_2$ are outer bi-Lipschitz equivalent, then $X_1$ and $X_2$ are inner bi-Lipschitz equivalent. However, the converses are not generally true (for counterexamples, see \cite{BBG}).
\end{remark}

\begin{definition}\label{LNE}\normalfont
	We say that a path-connected set $X \subset \mathbb{R}^{n}$ is \textit{Lipschitz normally embedded} (or LNE, for short) if the outer metric and inner metric are bi-Lipschitz equivalent i.e. there is a constant $C \ge 1$ such that
	\begin{equation*}
		d_X(x,y) \leq C \cdot d(x,y); \quad \forall \; x,y \in X.
	\end{equation*}
	In this case, we say that $X$ is $C$-LNE. Given $p\in X$, we say that $X$ is \textit{Lipschitz normally embedded at $p$} (or LNE at $p$) if there exists a neighborhood $U$ of $p$ such that $X \cap U$ is LNE, or equivalently, the germ $(X,p)$ is LNE. If $X\cap U$ is $C$-LNE for some $C\ge 1$, we say that $(X,p)$ is $C$-LNE.
\end{definition}

\begin{definition}\label{arc}\normalfont
	An \textit{arc in $\mathbb{R}^{n}$} with initial point $p$ is a germ at the origin of a semialgebraic map $\gamma : [0, t_0) \to \mathbb{R}^{n}$, for some $t_0 >0$, such that $\gamma (0) = p$. Every arc with an initial point at the origin will be simply denoted as an arc. Given a germ at the origin of a set $X$, the set of all arcs $\gamma \subset X$ is called the \textit{Valette link of $X$} and is denoted by $V(X)$ (See \cite{Valette-Link}).
\end{definition}

\begin{remark}\normalfont
	Usually, we identify an arc $\gamma$ with its image $\gamma(t) \in \mathbb{R}^{n}$ obtained by intersecting $\gamma$ with a sphere centered at $p$ with radius $t$ (or with a plane $\{x_{n+1}=t\}$), for $t$ sufficiently small. This intersection is unique, by Theorem \ref{conical}.	
\end{remark}

\begin{definition}\label{tord}\normalfont
	Given a set $X$ and two arcs $\gamma_1, \gamma_2 \subset V(X)$, we define the \textit{order of tangency of $\gamma_1, \gamma_2$ in the outer metric}, denoted as $tord(\gamma_1, \gamma_2)$, as the exponent $\beta \in \mathbb{Q}$ such that there exists a constant $c>0$ satisfying
	$$\| \gamma_1 (t) - \gamma_2 (t) \| = ct^{\beta}+o(t^\beta)$$
	Notice that, if $\gamma_1 \neq \gamma_2$, by the Newton-Puiseux theorem, such constants $c$ and $\beta$ exist. We also define $tord(\gamma,\gamma) =\infty$ for every curve $\gamma \in V(X)$. Similarly, we define the \textit{order of tangency of $\gamma_1, \gamma_2$ in the inner metric} and denote it as $tord_X(\gamma_1, \gamma_2)$.
\end{definition}

\begin{remark}\label{Rem:tord}\normalfont
	These order of tangency values are rational numbers (or elements of the field of exponents in an o-minimal structure) satisfying $1 \le tord(\gamma_1,\gamma_2) \le tord_X(\gamma_1, \gamma_2)$ for all $\gamma_1, \gamma_2 \in V(X)$. Moreover, $X$ is LNE if, and only if, $tord(\gamma_1,\gamma_2) \le tord_X(\gamma_1, \gamma_2)$ for all $\gamma_1, \gamma_2 \in V(X)$ (see \cite{comRodrigo}).
\end{remark}

\begin{definition}\label{curvilinear}\normalfont
	A set $A \subset \mathbb{R}^n$ is a \textit{curvilinear triangle with vertices $a_1, a_2, a_3$} if the following holds:
	\begin{itemize}
		\item $A$ is a 2-dimensional topological submanifold with boundary.
		\item The boundary $\partial A$ of $A$ is the union of the points $a_1, a_2, a_3$ and the smooth curves connecting these points, which are called the edges of the triangle.
		\item The interior $int(A)$ of $A$ is smooth.
	\end{itemize}
	We can choose one of its vertices and denote it as the \textit{main vertex of the triangle}. If $a_1$ is the main vertex of $A$, we denote \textit{the boundary arcs of} $A$ as the germs, at $a_1$, of each of the edges connecting $a_1$ to $a_2$ and $a_1$ to $a_3$.
\end{definition}

\begin{remark}\label{Rem:smooth}\normalfont
	Given an $n$-dimensional set $X$, the set of smooth points of $X$ consists of all points $p\in X$ such that the tangent cone to $X$ at $p$ is isomorphic to $\mathbb{R}^n$. The remaining points of $X$ will be called singular, and we denote the set of such points as $sing(X)$.
\end{remark}

\begin{definition}\label{Holder-triangle}\normalfont
	Given a rational number $\alpha \ge 1$, we define the \textit{$\alpha$-standard H\"older triangle} as the set
	$$T_{\alpha} = \{ (x,y) \in \mathbb{R}^{2} \mid 0\le x \le 1 ; 0\le y \le x^{\alpha} \}.$$
	
	The curves $l_0 := \{ (x,0) \in \mathbb{R}^{2} : 0\le x \le 1\}$ and $l_1 := \{ (x,x^{\alpha}) \in \mathbb{R}^{2} : 0\le x \le 1 \}$ are defined as the \textit{boundary arcs of $T_{\alpha}$}. We say that a set $X\subset \mathbb{R}^{n}$ is an \textit{$\alpha$-H\"older triangle with the main vertex at $a \in X$} if there exists an inner bi-Lipschitz map $\varphi : (T_\alpha,0) \to (X,a)$. Here, $T_\alpha, X$ are seen as metric spaces with the induced metric from $\mathbb{R}^2,\mathbb{R}^n$, respectively. The sets $\gamma_0=\varphi (l_0)$, $\gamma_1=\varphi (l_1)$ are defined as the \textit{boundary arcs of $X$} and we also denote the H\"older triangle $X$ as $T(\gamma_0,\gamma_1)$.
\end{definition}

\begin{remark}\label{embedding-remark}\normalfont
	For each integer $n\ge 2$, if we consider the embedding $f: \mathbb{R}^2 \to \mathbb{R}^n$; $f(x,y)=(y,0,\dots,0,x); \; \forall (x,y) \in \mathbb{R}^2$, then for each rational number $\alpha \ge 1$, we define the $\alpha$-standard H\"older triangle embedded in $\mathbb{R}^n$ as the set $f(T_\alpha)$. It is easy to see that the $\alpha$-standard H\"older triangle embedded in $\mathbb{R}^n$ is an $\alpha$-H\"older triangle with the main vertex at the origin.
\end{remark}

\begin{definition}\label{horn}\normalfont
	Given a rational number $\beta\ge 1$, we define the \textit{$\beta$-standard horn} as the set
	$$H_{\beta} = \{ (x,y,t) \in \mathbb{R}^{3} \mid 0\le t \le 1 \; ; \; x^2 + y^2 = t^{2\beta} \}.$$
	
	We say that a set $X\subset \mathbb{R}^{n}$ is a \textit{$\beta$-horn with the main vertex at $a \in X$} if there exists an inner bi-Lipschitz map $\varphi : (H_\beta,0) \to (X,a)$. Here, $H_\beta$ and $X$ are seen as metric spaces with the induced metric from $\mathbb{R}^2$ and $\mathbb{R}^n$, respectively.
\end{definition}

\subsection{Preliminary Results}

\begin{theorem}[Local Conical Structure] \label{conical}
	Let $X\subset \mathbb{R}^n$ be a semialgebraic set and $p\in X$ be a non-isolated point of $X$. For each $\varepsilon>0$, let $p*(\mathbb{S}_{\varepsilon}^{n-1}(p)\cap X)$ be the cone with vertice $p$ and base $\mathbb{S}_{\varepsilon}^{n-1}(p)\cap X$, that is,
	$$p*(\mathbb{S}_{\varepsilon}^{n-1}(p)\cap X)=\{ \lambda p+ (1-\lambda)x \mid \lambda \in [0,1] \;;\; x \in \mathbb{S}_{\varepsilon}^{n-1}(p) \cap X \}.$$
	
	Then, there is $\varepsilon>0$ small enough and a semialgebraic homeomorphism $h:\overline{\mathbb{B}_{\varepsilon}^{n}(p)}\cap X \to p*(\mathbb{S}_{\varepsilon}^{n-1}(p)\cap X)$ such that $\| h(x) - p \|=\|x-p\|$, for all $x \in \overline{\mathbb{B}_{\varepsilon}^{n}(p)}\cap X$, and $h|_{\mathbb{S}_{\varepsilon}^{n-1}(p) \cap X} = id_{\mathbb{S}_{\varepsilon}^{n-1}(p) \cap X}$.
\end{theorem}

\begin{proof}
	See \cite{coste-livro-semialg}, Theorem 4.4.
\end{proof}

\begin{proposition}\label{bi-lip-extende-fecho}
	Let $(X_1,d_1)$ and $(X_2,d_2)$ be metric spaces, and let $U_1 \subseteq X_1, U_2 \subseteq X_2$ be open sets such that $\overline{U_1}=X_1,\overline{U_2} =X_2$. If $\varphi: X_1 \to X_2$ is a homeomorphism, such that $\varphi(U_1)=U_2$ and $\varphi|_{U_1} : U_1 \to U_2$ is a $C$-bi-Lipschitz map, then $\varphi$ is a $C$-bi-Lipschitz.
\end{proposition}

\begin{proof}
	Given $p,q \in X_1$, there are sequences $\{p_n\}_{n\in \mathbb{Z}_{\ge 1}}, \{q_n\}_{n\in \mathbb{Z}_{\ge 1}} \subseteq U_1$ such that $p_n \to p$ e $q_n \to q$. Since $\varphi$ is homeomorphism, for every $\varepsilon>0$, there is $N_0 \in \mathbb{Z}_{\ge 1}$ such that, for all $n > N_0$ integer, we have
	$$d_1(p,p_n),d_1(q,q_n),d_2(\varphi(p),\varphi(p_n)),d_2(\varphi(q),\varphi(q_n))<\varepsilon.$$
	
	Since $\varphi|_{U_1}$ is $C$-bi-Lipschitz, we have
	$$d_2(\varphi(p),\varphi(q)) \le d_2(\varphi(p),\varphi(p_n)) + d_2(\varphi(p_n),\varphi(q_n))+d_2(\varphi(q_n),\varphi(q))<$$
	$$< \varepsilon + C \cdot d_1(p_n,q_n)+\varepsilon \le 2\varepsilon+ C \cdot (d_1(p_n,p)+d_1(p,q)+d_1(q,q_n))<$$
	$$<2\varepsilon +C \cdot (\varepsilon+d_1(p,q)+\varepsilon) =2(C+1)\varepsilon+C \cdot d_1(p,q).$$
	
	Since this holds for every $\varepsilon>0$, we have $d_2(\varphi(p),\varphi(q)) \le C\cdot d_1(p,q)$. Similarly, $d_1(p,q) \le C \cdot d_2(\varphi(p),\varphi(q))$. Therefore, $\varphi$ is $C$-bi-Lipschitz.
\end{proof}

\begin{proposition}\label{extensão-bi-lip-invariante-bola}
	Let $U,V \subseteq \mathbb{R}^n$ be open and not empty, and let $\psi: \overline{U} \to \overline{V}$, $\varphi: \overline{V} \to \overline{V}$ be outer bi-Lipschitz maps. If $\varphi(p) = p, \, \forall \, p \in \partial V$, then the map $\Phi: \mathbb{R}^n \to \mathbb{R}^n$, given by
	$$\Phi(p) =
	\begin{cases}
		p, & p \notin U \\
		\psi^{-1} \circ \varphi \circ \psi (p), & p \in \overline{U}
	\end{cases}$$
	is an outer bi-Lipschitz map.
\end{proposition}

\begin{proof}
	Suppose that $\varphi$, $\psi$ are $C$-bi-Lipschitz, for some $C\ge 1$. Then, the map $\psi^{-1} \circ \varphi \circ \psi : \overline{U} \to \overline{U}$ is $C^{3}$-bi-Lipschitz. Since $\psi$ is homeomorphism, for any $p\in \partial U$, $\psi(p) \in \partial V$. Then, $\psi^{-1} \circ \varphi \circ \psi (p) = \psi^{-1} \circ \psi (p)=p$, showing that $\Phi$ is well defined.
	
	Now we will prove that $\Phi$ is $C^3$-bi-Lipschitz. The result is immediate if $p,q \in U$ or $p,q \notin U$. If $p \in U, q \notin U$, let $r \in \partial U$ be a point in the line segment connecting $p$ and $q$. Then we have
	$$\| \Phi(p)-\Phi(q) \| \le \| \Phi(p)-\Phi(r) \| + \| \Phi(r)-\Phi(q) \| = \| \Phi(p)-\Phi(r) \| + \| r-q \| \le $$
	$$\le C^3 \cdot \| p-r \| + \| r-q \| \le C^3 \cdot \left( \| p-r \| + \| r-q \| \right) = C^3 \cdot \| p-q \|.$$
	Similarly, $\| p - q \| \le C^3 \cdot \| \Phi(p) - \Phi(q) \|$. Therefore, $\Phi$ is an outer bi-Lipschitz map.
\end{proof}

\begin{proposition} \label{colagem-bi-lip-inner}
	Let $X_1,X_2 \subseteq \mathbb{R}^n$, $Y_1,Y_2 \subseteq \mathbb{R}^m$ be closed arcwise connected sets such that $X_1\cap X_2, Y_1 \cap Y_2 \ne \emptyset$. Suppose $\varphi_1 : X_1 \to Y_1$ and $\varphi_2 : X_2 \to Y_2$ are inner bi-Lipschitz maps satisfying $\varphi_1(p)=\varphi_2(p)$, for every $p\in X_1 \cap X_2$, and $\varphi_1^{-1}(q)=\varphi_2^{-1}(q)$, for every $q\in Y_1 \cap Y_2$. Then, if $X=X_1\cup X_2$ and $Y=Y_1 \cup Y_2$, the map $\varphi: X \to Y$ given by $\varphi(p)=\varphi_i(p)$, if $p\in X_i$ ($i=1,2$), is an inner bi-Lipschitz map.
\end{proposition}

\begin{proof}
	Since $\varphi_1(p)=\varphi_2(p), \, \forall\,  p \in X_1 \cap X_2$, $\varphi$ is well defined. Since $\varphi_1, \varphi_2$ are homeomorphisms, by the Gluing Lemma for homeomorphisms we have that $\varphi$ is also a homeomorphism. As $X_1, X_2,Y_1, Y_2$ are arcwise connected and $X_1\cap X_2, Y_1 \cap Y_2 \ne \emptyset$, we conclude that $X$ and $Y$ are also arcwise connected.
	
	For $k=1,2$, since $X_k \subseteq X$ and $Y_k \subseteq Y$, we have $d_{X_k}(a,b) \ge d_X (a,b), \forall \, a,b \in X_k$ and $d_{Y_k}(a,b) \ge d_Y (a,b), \forall \, a,b\in Y_k$. Suppose that each $\varphi_k$ is inner $C$-bi-Lipschitz. We will prove that if $p, q\in X$ and $d_X(p,q) < \infty$, then
	$$\dfrac{1}{2C} \cdot d_X(p,q) \le d_Y(\varphi(p),\varphi(q)) \le 2C \cdot d_X(p,q), \quad \forall \, p,q\in X.$$
	
	Given $\varepsilon>0$, consider a rectifiable curve $\gamma : [0,1] \to X$ of length $\ell$ such that $\gamma(0)=p,\gamma(1)=q$ and $\ell < d_X(p,q) + \varepsilon$. We have that $\gamma \cap X_1$ and $\gamma \cap X_2$ are, each one, the union of curves and points, such that the sum of their lengths, in $X_1$ and $X_2$, respectively, is at most the length of $\gamma$ in $X$, which is $\ell$.
	
	\begin{figure}[!h]
		\centering
		\includegraphics[width=12cm]{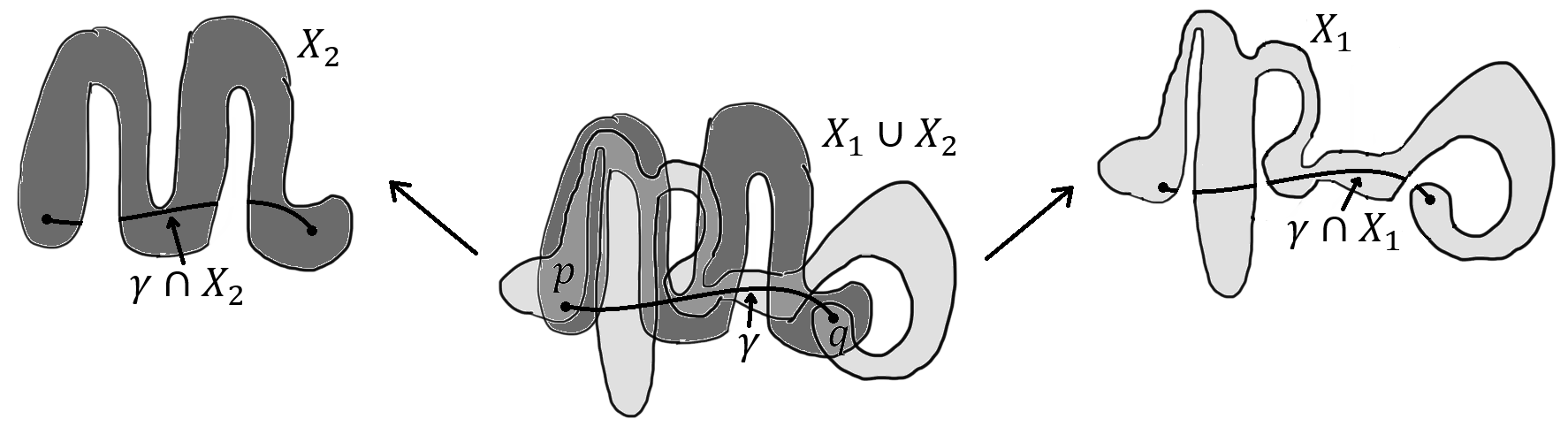}
		\label{fig31}
		\caption{Proof of Proposition \ref{colagem-bi-lip-inner}.}
	\end{figure}
	
	Since $\varphi(\gamma \cap X_k) =\varphi_k(\gamma \cap X_k)$ and $\varphi_k$ is inner $C$-bi-Lipschitz, the sum of the lengths, in $Y_k$, of the curves in $\varphi(\gamma \cap X_k)$ is at most $C\cdot \ell$. Then, the length, in $Y$, of $\varphi(\gamma) = \varphi_1(\gamma \cap X_1) \cup \varphi_2(\gamma \cap X_2)$ is at most $2C\cdot \ell < 2C(d_X(p,q)+\varepsilon)$. Therefore, $d_Y(\varphi(p),\varphi(q)) <2C(d_X(p,q)+\varepsilon)$ and since this holds for every $\varepsilon>0$, we have $d_Y(\varphi(p),\varphi(q)) \le 2C\cdot d_X(p,q)$. Similarly, $d_X(p,q) \le 2C \cdot d_Y(\varphi(p),\varphi(q))$, finishing the proof. Notice also that this implies $d_X(p,q)=\infty$ if, and only if, $d_Y(\varphi(X),\varphi(Y))=\infty$. We conclude that $\varphi$ is a $2C$-inner bi-Lipschitz map.
\end{proof}

\begin{definition}\label{Def:bounded-away}\normalfont
	Let $m,n \in \mathbb{Z}_{\ge 1}$ and let $A_{m,n}$ be a real matrix. We say that $A$ is \textit{bounded away from 0 and infinity in $X$} if there is $M>0$ such that $\frac{1}{M} < \| A \| < M$. If $\{A(p)\}_{p\in X}$ is a family of matrices, we say that $\{A(p)\}_{p\in X}$ is \textit{bounded away from 0 and infinity} if there is $M>0$ such that, for every $p \in X$, $\frac{1}{M} < \| A(p) \| < M$.
\end{definition}

\begin{remark}\label{norm-remark}
	In Definition \ref{Def:bounded-away}, $\left\| \cdot \right\|$ denotes a fixed matrix norm. Since all such norms are equivalent, the definition is independent of the chosen norm. Unless otherwise specified in this paper, we use the maximum norm on real matrices.
\end{remark}

\begin{proposition}\label{jacobiano-lip}
	Let $X,Y \subseteq \mathbb{R}^{n}$ be connected and open sets, and let $\varphi : \overline{X} \to \overline{Y}$ be a homeomorphism such that $\varphi_0 := \varphi|_{X} : X \to Y$ is a diffeomorphism. If the jacobians $\{J_{\varphi_0}(p)\}_{p \in X }$ and $\{J_{{\varphi_0}^{-1}}(p)\}_{p \in X}$ are bounded away from 0 and infinity in $X$, then $\varphi$ is an inner bi-Lipschitz map.
\end{proposition}

\begin{proof}
	Since $\{J_{\varphi_0}(p)\}_{p \in X }$ and $\{J_{{\varphi_0}^{-1}}(p)\}_{p \in Y}$ are bounded away from 0 and infinity in $X$, there is $C>0$ such that $$\frac{1}{C} < \left\| J_{\varphi_0}(p) \right\|, \left\| J_{{\varphi_0}}^{-1}(\varphi_0(p)) \right\| < C \, , \, \forall \, p \in X.$$
	Given $a,b \in X$ and $\varepsilon>0$, consider a rectifiable curve $\gamma : [0,1] \to X$, with $\gamma(0)=a$, $\gamma(1)=b$, and length $\ell(\gamma) < d_X(a,b)+\varepsilon$. We have
	$$d_Y(\varphi_0(a),\varphi_0 (b)) \le \ell(\varphi_0 (\gamma)) = \int_{0}^{1} \| (\varphi_0 \circ \gamma)^{\prime}(t) \| \,dt = \int_{0}^{1} \| J_{\varphi_0}(\gamma(t)) \cdot \gamma^{\prime}(t) \| \,dt \le$$
	$$\le \int_{0}^{1} \| J_{\varphi_0}(\gamma(t)) \| \cdot \| \gamma^{\prime}(t) \| \,dt \le \int_{0}^{1} C \cdot \| \gamma^{\prime}(t) \| \,dt =C \int_{0}^{1} \| \gamma^{\prime}(t) \| \,dt =$$
	$$ = C\cdot \ell(\gamma) < C(d_X(a,b)+\varepsilon) \therefore  d_Y(\varphi_0(a),\varphi_0 (b)) < C\cdot d_X(a,b) + C\varepsilon.$$
	
	As the above inequality holds for all $\varepsilon >0$, we have $d_Y(\varphi_0(a),\varphi_0 (b)) \le C\cdot d_X(a,b)$, for all $a,b \in X$. Similarly, we have $d_X(a,b) \le C \cdot d_Y(\varphi_0(a),\varphi_0 (b))$, proving that $\varphi_0$ is a $C$-bi-Lipschitz map. Since $\varphi$ is a homeomorphism, $\varphi$ is a inner $C$-bi-Lipschitz map by Proposition \ref{bi-lip-extende-fecho}.
\end{proof}

\begin{proposition} \label{im-LNE-is-LNE}
	Let $X \subseteq \mathbb{R}^n$ be a LNE set and let $\Phi: \mathbb{R}^n \to \mathbb{R}^n$ be an outer bi-Lipschitz map. Then, $\Phi(X) \subseteq \mathbb{R}^n$ is a LNE set.
\end{proposition}

\begin{proof}
	Let $C>1$ such that $X$ is $C$-LNE and $\Phi$ is outer $C$-bi-Lipschitz. For each $\varepsilon>0$ and $a=\Phi(p), b = \Phi(q)$ in $X$, let $\gamma \subset X$ be a rectifiable curve connecting $p$ to $q$, with length $\ell<d_X(p,q)+\varepsilon$. since $X$ is $C$-LNE, we have $\ell \le C \cdot \| p-q\| +\varepsilon$, and since $\Phi$ is outer $C$-bi-Lipschitz, the length of $\Phi(\gamma)$ is at most $C\cdot \ell$. Therefore,
	$$d_{\Phi(X)}(a,b) \le C\cdot \ell \le C^2 \cdot \| p-q\|+C\cdot\varepsilon \le C^3 \cdot \| a-b\|+C\cdot\varepsilon.$$
	Since this inequality holds for all $\varepsilon>0$, we conclude that $\Phi(X)$ is $C^3$-LNE.
\end{proof}

\begin{proposition} \label{grafico-LNE}
Let $x_1,x_2 \in \R$, with $x_1 < x_2$, and let $f,g: [x_1,x_2] \to \mathbb{R}$ be piecewise smooth functions satisfying $g(x) \le f(x)$, for all $x \in[x_1,x_2]$. If there is $M>0$ such that $|f^{\prime}(x)|,|g^{\prime}(x)| < M$, for all $x$ where $f,g$ are differentiable, then $X$ is LNE, where 
$$X = \{(x,y) \in \mathbb{R}^2 \mid x_1 \le x \le x_2 \; ; \; g(x) \le y \le f(x) \}.$$
\end{proposition}

\begin{proof}
We will prove first the following auxiliary lemma.

\begin{lemma}\label{abertura-limitada}
Let $P,Q,R \in \mathbb{R}^n$ be distinct points such that $\cos{(\angle PQR )} \le 1 - \frac{2}{T^2}$, for some $T>1$. Then, $d(P,Q) + d(Q,R) \le T\cdot d(P,R).$
\end{lemma}

\begin{proof}
Let $x=d(P,Q)$ and $y=d(Q,R)$. By the law of cosines in $PQR$, we have
$$d(P,R)^2 = x^2+y^2 - 2xy \cdot cos{(\angle ABC )} \ge x^2 + y^2 - 2xy \left(1-\frac{2}{T^2} \right) = $$
$$=\frac{1}{T^2} \left( x+y \right)^2 +\left( 1 - \frac{1}{T^2} \right) (x-y)^2 \ge \frac{1}{T^2} \left( x+y \right)^2 = \left( \frac{d(P,Q)+d(Q,R)}{T} \right)^2.$$
Therefore, $d(P,Q) + d(Q,R) \le T\cdot d(P,R).$
\end{proof}

Back to the proof of Proposition \ref{grafico-LNE}, let $A, B \in X$, $A=(x_A,y_A), B=(x_B,y_B)$. If $x_A=x_B$, then $d(A,B) = d_X(A,B)$. If $x_A \ne x_B$, suppose without loss of generality $x_A < x_B$ and take $\lambda \in [0,1]$ satisfying $y_A = \lambda g(x_A) + (1-\lambda )f(x_A)$. Let
$$\gamma: [x_A,x_B] \to \mathbb{R} \; ; \; \gamma(x)=\lambda g(x)+ (1-\lambda) f(x) \; , \; \forall x \in [x_A,x_B],$$
and let $C=(x_B,\gamma(x_B))$. For each $x \in [x_A,x_B]$ where $f,g$ are differential in $x$, we have
$$|\gamma^{\prime}(x)|= |g^{\prime}(x) + (1-\lambda) f^{\prime}(x)| \le \lambda |g^{\prime}(x)| + (1-\lambda) |f^{\prime}(x)| < \lambda M + (1-\lambda) M=M$$

The length of $\gamma$ is $\ell(\gamma) = \int \sqrt{1+\gamma^{\prime}(x)^2}dx$, where the integral is taken over all non-singular points of $f$ and $g$ in $[x_A,x_B]$. Since $\gamma^{\prime}(x)^2 \le M^2$, we have
$$\ell(\gamma) \le \int_{x_A}^{x_B} \sqrt{1+M^2} dx =  \sqrt{1+M^2}\cdot (x_B - x_A) \le \sqrt{1+M^2} \cdot d(A,C). $$

If $C \ne B$, let $\alpha= \angle ACB$. Since $|\gamma^{\prime}(x)| \le M$, we have
$$|\cot{\alpha}| = \dfrac{|\gamma(x_B)-y_{A}|}{x_B-x_{A}} \le M \Rightarrow |\cos{\alpha}| = \sqrt{\dfrac{1}{\dfrac{1}{\cot^2{\alpha}}+1}} \le \sqrt{\dfrac{1}{\dfrac{1}{M^2}+1}} = 1 - \dfrac{2}{T^2},$$
for some $T > 1$. If $N = T \cdot\sqrt{1+M^2}$, by Lemma \ref{abertura-limitada}, we have
$$d_X(A,B) \le d_X(A,C)+d_X(C,B) = d_X(A,C)+\| C - B \| \le \ell(\gamma)+ \| C - B \| \le$$
$$\le \sqrt{1+M^2} \cdot \| A - C\| + \|C-B\| \le \sqrt{1+M^2}\cdot(\|A - C\|+\|C-B\|)$$
$$\le \sqrt{1+M^2}\cdot T \cdot \|A - B \| \therefore d_X(A,B) \le N \cdot d(A,B).$$

Such a conclusion also holds for $B=C$, because
$$d_X(A,B) \le \ell(x) \le \sqrt{1+M^2}\cdot d(A,C) = \sqrt{1+M^2}\cdot d(A,B) \le N \cdot d(A,B).$$

Therefore $X$ is $N$-LNE, and the result follows.
\end{proof}

\begin{theorem}\label{Edson-Rodrigo}
Let $X \subset \mathbb{R}^{n}$ be a closed semialgebraic set, such that $0 \in X$ and the link of $X$ is connected. Then, $X$ is LNE at 0 if and only if there is a constant $C\ge1$ such that $X_t$ is $C$-LNE, for all $t>0$ small enough.
\end{theorem}

\begin{proof}
See \cite{LLNE-LNE}, Theorem 3.1.
\end{proof}

\section{Technical Lemmas in Ambient Lipschitz Geometry}\label{technical}

\subsection{Reduction to Plane Links}

In this section, we show that the main problem of the paper can be reduced to the analysis on surfaces within a specific type of cone, whose sections orthogonal to an axis are closed balls that can be interpreted as its plane link. Such a reduction will be important for the achievement of calculations of ambient bi-Lipschitz maps in later sections.

\begin{definition}\label{cone}\normalfont
Given $n \in \mathbb{Z}_{\ge 1}$ and $a, R>0$, define
$$C_{a}^{n+1}=\{ (x_1, \dots, x_n, t) \in \mathbb{R}^{n+1} \mid t\ge 0; \; x_1^2 + \dots + x_n^2 \le (at)^2 \}.$$
\end{definition}

Let us define the following sets.

$$ -C_{a}^{n+1}=\{ - p \mid p \in C_a^{n+1} \}\, ; \, C_{a}^{n+1}(R) = C_{a}^{n+1} \cap \{t=R\} ;  $$
$$U_{a}^{n+1} = \mathbb{R}^{n+1} \setminus -C_{a}^{n+1}\, ; \, U_{a}^{n+1}(R) :=U_{a}^{n+1} \cap \mathbb{S}^{n}(0,R).$$

\begin{figure}[h]
\centering
\includegraphics[width=13cm]{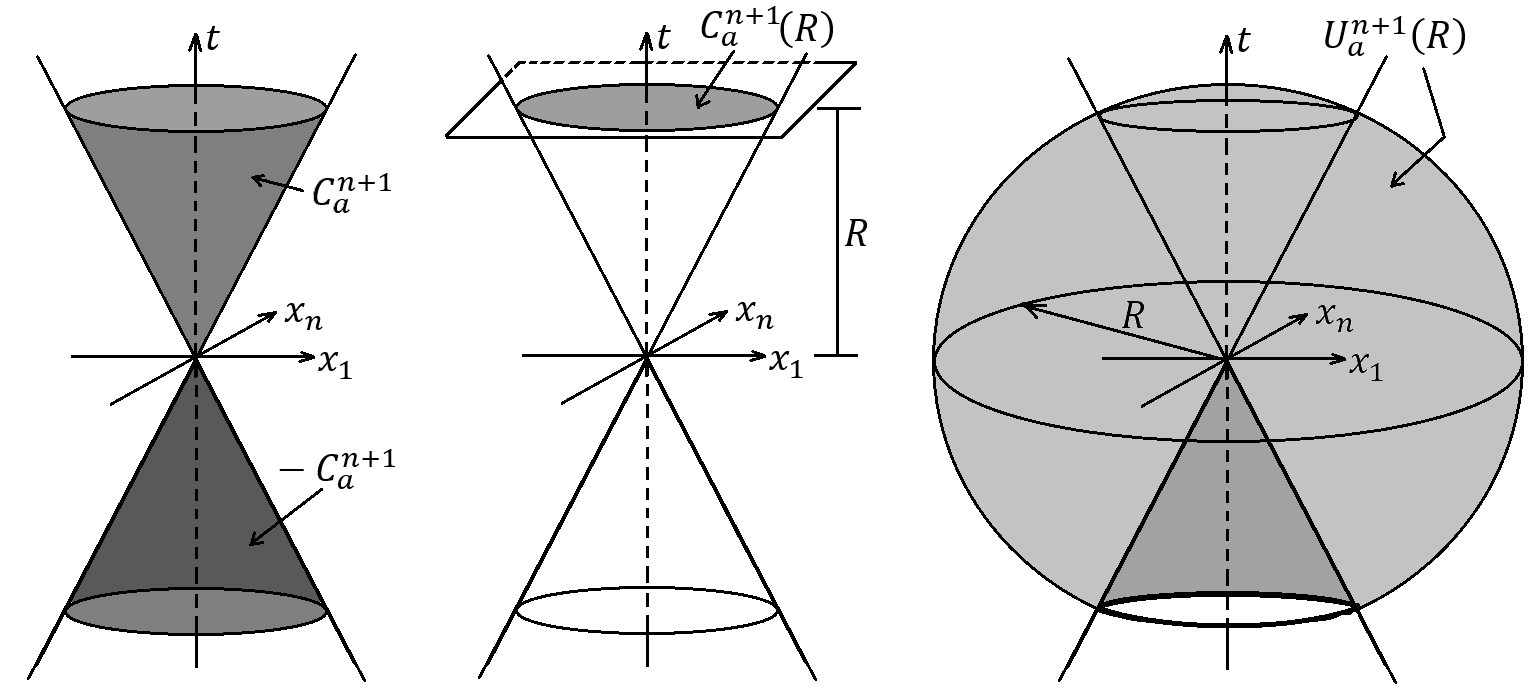}
\label{fig2}
\caption{Graphic representations of $C_{a}^{n+1}$, $-C_{a}^{n+1}$, $C_{a}^{n+1}(R)$ and $ U_a^{n+1}(R)$.}
\end{figure}

We know that, for each $R>0$, the stereographic projection
$$\psi_{R}: \mathbb{S}^n(0,R) - \{ (0,\dots,0,-R) \} \to \mathbb{R}^{n}\times \{ R \}$$
$$(x_1,\dots,x_n,x_{n+1}) \mapsto \left( \lambda x_1, \dots, \lambda x_n, R \right); \; \lambda=\dfrac{2R}{x_{n+1}+R}$$

is a diffeomorphism, whose inverse is
$$\psi_{R}^{-1}: \mathbb{R}^{n}\times \{R \} \to \mathbb{S}^n(0,R) - \{ (0,\dots,0,-R) \}$$
$$ (x_1,\dots,x_n, R) \mapsto \left( \lambda^{\prime}x_1, \dots, \lambda^{\prime}x_n, (2\lambda^{\prime} -1) R \right) \; ; \lambda^{\prime} = \dfrac{4R^2}{x_1^2+\dots+x_n^2+4R^2}.$$

Furthermore, for every $a>2$ large enough, there is $a^{\prime}>0$ sufficiently small such that $\psi_R|_{\overline{U_{a^{\prime}}^{n+1 }(R)}}$ is a diffeomorphism over its image $C_a^{n+1}(R)$, and that $a^{\prime}$ depends only on $a$ (more precisely, $a^{ \prime}=\frac{4a}{a^2+4}$).

\begin{proposition} \label{reduction}
The map $\psi : \overline{U_{a^{\prime}}^{n+1}} \to C_a^{n+1}$ given by $\psi (x_1,\dots ,x_n, R) = \psi_R (x_1,\dots ,x_n,R)$, for all $(x_1,\dots ,x_n,R) \in \overline{U_{a^{\prime}}^{n+1}}$ , is an outer bi-Lipschitz map.
\end{proposition}

\begin{proof}
The set $C_{a}^{n+1}$ is obviously path connected and LNE, so it suffices to check if $\psi^{-1}$ satisfies the conditions of Proposition \ref{jacobiano-lip}. The map $\psi|_{U_{a^{\prime}}^{n+1}}$ is a diffeomorphism whose inverse is given by
$$\psi^{-1}|_{int(C_a^{n+1})}(x_1,\dots,x_n, R) = \psi_R^{-1} (x_1,\dots,x_n, R ) \; ; \; \forall (x_1,\dots,x_n, R) \in int(C_a^{n+1}).$$

Since $C_a^{n+1}(1)$ is compact, $\psi_1^{-1}|_{C_a^{n+1}(1)}$ is bounded away from 0 and infinity. Since $\psi^{-1}$ sends the cone of $C_a^{n+1}(1)$ to the cone of $\overline{U_{a^{\prime}}^{n+1}(1)}$, with $\psi^{-1}(C_a^{n+1}(R))=\overline{U_{a^{\prime}}^{n+1}(R)}$, for all $R>0$, $\psi^{-1}$ is bounded away from 0 and infinity.
\end{proof}

\begin{proposition} \label{suficiencia-em-cone}
Assume that Theorem \ref{main} holds for every LNE surface germ $(X,0) \in C_a^{3}$. Thus, the Theorem \ref{main} is valid for every surface germ LNE $(X,0) \subset \mathbb{R}^3$.
\end{proposition}

\begin{proof}
Given $X \subset \mathbb{R}^3$ LNE, there exists a vector $u \notin S = \{ \gamma^{\prime}(0) \mid \gamma \in V(X) \}$ . Taking a rotation of axes, if necessary, we can assume that $u=(0,0,-p)$, for some $p>0$. Being $S$ closed, there exists $a^{\prime}>0$ small enough such that $u \notin S, \forall u \in \overline{U_{a^{\prime}}^{3}}$. Hence, $\psi: (\overline{U_{a^{\prime}}^{3}},0) \to (C_a^{3},0)$ is an outer bi-Lipschitz map and so $(\psi(X),0)$ is a germ of a LNE surface, by Proposition \ref{im-LNE-is-LNE}. 

If $\varphi_1,\dots,\varphi_m : (C_a^{3},0) \to (C_a^{3},0)$ are sequences of outer bi-Lipschitz maps such that $\varphi_i(p)=p $, for all $p\in \partial C_a^{3}$ and $i=1,\dots,m$, and if $\varphi := \varphi_m \circ \dots \circ \varphi_1$ is such that $ \varphi(\psi(X),0)= \psi(T_{\alpha},0)$ or $\psi(H_{\beta})$, then taking the map $\Phi$ of Proposition \ref{extensão-bi-lip-invariante-bola}, we have, by the hypothesis of the Proposition, that $\Phi(X,0)=(T_{\alpha},0)$ or $(H_{\beta},0)$.
\end{proof}

\begin{figure}[h]
\centering
\includegraphics[width=15cm]{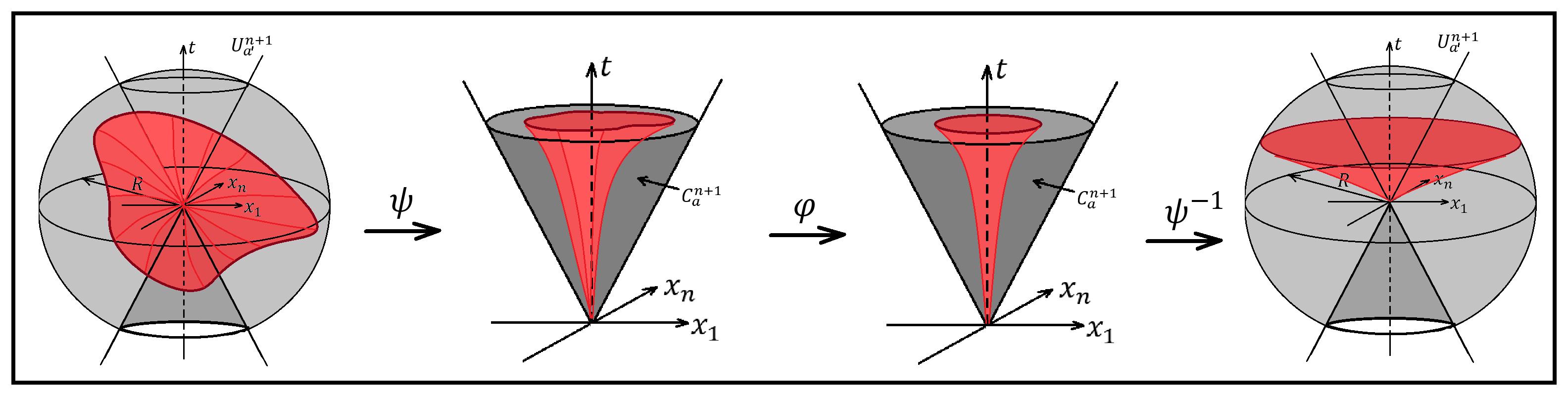}
\label{fig31}
\caption{Proof of Proposition \ref{suficiencia-em-cone}.}
\end{figure}

Motivated by Propositions \ref{reduction} and \ref{suficiencia-em-cone}, we have the following definitions:

\begin{definition} \label{link-plano}\normalfont
For each $X \subset C_a^{n+1}$ and each $t>0$, define the sets:
$$C_a^{n+1}[t] = \{ (x_1, \dots, x_n,t) \in C_a^{n+1} \mid 0<t < t \}\cup \{ 0 \} $$
$$C_a^{n+1}(t) = \{ (x_1, \dots, x_n, t) \in C_a^{n+1} \}$$
$$X[t]=X\cap C_a^{n+1}[t]$$
$$X(t)=X\cap C_a^{n+1}(t)$$

The set $X(t)$ is the $t$-plane link of $X$.
\end{definition}

\begin{definition} \label{equiv-amb-cone}\normalfont
Given the germs of sets $(X,0),(Y,0) \subset (C_a^{n+1},0)$, we say that $(X,0)$ and $(Y,0)$ are ambient bi-Lipschitz equivalents in $(C_a^{n+1},0)$ if there is $\varepsilon>0$ small enough and a outer bi-Lipschitz map $\varphi: (C_a^{n+1},0) \to (C_a^{n+1},0)$ such that
\begin{enumerate}
\item $\varphi(p) = p$, for all $p \in \partial C_a^{n+1} \cap C_a^{n+1}[\varepsilon]$;
\item $\varphi(X[\varepsilon])=Y[\varepsilon].$
\end{enumerate}
\end{definition}

\subsection{Translation and Dilatation along Arcs}

Once the reduction has been made, we provide some lemmas that allow the translation and dilatation of the plane links of the studied surfaces.

\begin{definition}\label{arc-translation}\normalfont
 Let $a_0>0$, $X \subset C_{a_0}^{n+1}$ be a surface and $(\gamma,0) \subset (C_{a_0},0)$, $\gamma(t) = (y_1(t),\dots,y_n(t),t)$ an arc. The \emph{arc translation of $X$ along $\gamma$} is the set, defined as
$$X^{\gamma} =\{(x_1+y_1(t),\dots,x_n+y_n(t),t) \mid (x_1,\dots,x_n,t) \in X\}.$$
\end{definition}

\begin{lemma}[Arc Translation Lemma] \label{translation}
There exists $a\ge a_0$ such that $(X,0),(X^{\gamma},0)$ are ambient bi-Lipschitz equivalents in $(C_a^{n+1},0)$.
\end{lemma}

\begin{proof}
Since $(\gamma,0) \subset (C_{a_0},0)$ is an arc, there is $\gamma^{\prime}(0)=(y_1^{\prime}(0),\dots , y_n^{\prime}(0),1) \in C_a^{n+1}$. For simplicity, we will write $y_i=y_i(t)$ and $y_i^{\prime} = y_i^{\prime}(t)$. We have that $M>0$ and $\varepsilon>0$ exist small enough such that $\left| \frac{y_i}{t} \right|, \left| y_i^{\prime} \right| < M$, for $i=1,\dots,n$ and all $0<t < \varepsilon$. For each $a>\max \{ 1,a_0 \}$. Let $b=\sqrt{a}$. Consider the sets, defined as follows:

$$ D^{\prime} = \{ (x_1,\dots,x_n,t) \in C_a^{n+1}[\varepsilon] \mid (x_1-y_1)^2+\dots+(x_n-y_n )^2 \le (bt)^2 \};$$
$$ D = \{ (x_1,\dots,x_n,t) \in C_a^{n+1}[\varepsilon] \mid x_1^2+\dots+x_n^2 \le (bt)^2 \}; $$
$$ \Gamma_{\lambda}^{\prime} = \{ (x_1,\dots,x_n,t) \in C_a^{n+1}[\varepsilon] \mid (x_1-\lambda y_1)^2 +\dots+(x_n- \lambda y_n)^2 = (\lambda bt + (1-\lambda)at)^2 \} \, (0\le \lambda \le 1);$$
$$ \Gamma_{\lambda} = \{ (x_1,\dots,x_n,t) \in C_a^{n+1}[\varepsilon] \mid x_1^2+\dots+x_n^2 = (\lambda bt + (1-\lambda)at)^2 \} \, (0\le \lambda \le 1).$$

Evidently $(\gamma,0) \subset (D^{\prime},0)$, $\Gamma_{1}^{\prime} = \partial D^{\prime}$, $\Gamma_{1} =\partial D$ and $\Gamma_{0}^{\prime}=\Gamma_{0}=\partial C_a^{n+1}$. We also have that $C_a^{n+1}(t) = D(t) \cup \bigcup_{0\le \lambda \le 1} \Gamma_{\lambda}(t) = D^{\prime}( t) \cup \bigcup_{0\le \lambda \le 1} \Gamma_{\lambda}^{\prime}(t)$.

\begin{claim} \label{homotetias}
There is $a>0$ large enough such that
\begin{enumerate}
\item $X(t) \subset D(t)$ and $X^{\gamma}(t) \subset D^{\prime}(t)$, for each $0<t\ll 1$;
\item the $n$-spheres $D(t),D^{\prime}(t)$ are inside $C_a^{n+1}(t)$, for each $0<t\ll 1$;
\item for $0\le \lambda_1 < \lambda_2 \le 1$, we have $\Gamma_{\lambda_1}(t) \cap \Gamma_{\lambda_2}(t) = \emptyset$ and $\Gamma_{\lambda_1 }^{\prime} \cap \Gamma_{\lambda_2}^{\prime} = \emptyset$.
\end{enumerate}
\end{claim}

\begin{proof}
Since $X \subset C_{a_0}$ and we can make $\varepsilon$ small enough so that $\| \gamma(t) \|$ is bounded, for $0<t<\varepsilon$, we obtain that there exist $a$  big enough so that the $n$-balls $D(t)$ and $D^{\prime }(t)$, both of radius $bt>a_0 t$, contain $X(t)$ and $X^{\gamma}(t)$, respectively. This implies condition (1). We have $ \| (y_1,\dots,y_n) \| < (a-b)t$, since $b, a-b \to \infty$ when $a \to \infty$. Since the common center of $D(t),C_{a}^{n+1}(t)$ is $\gamma_0(t)=(0,\dots,0,t)$, we obtain that $D( t) \subset int(C_a^{n+1}(t))$, and since $\|\gamma(t) - \gamma_0(t)\| = \| (y_1,\dots,y_n) \| < at-bt$, we have $D^{\prime}(t) \subset int(C_a^{n+1}(t)$. The condition (2) is proved.

To prove (3), for every $0\le \lambda \le 1$ and $t>0$ small enough, let $\gamma_{\lambda}(t)= \lambda \gamma(t)+(1-\lambda)\gamma_0(t)$ and $R_{\lambda} = \lambda b + (1-\lambda)a$. Also, let $\tilde{\gamma}(t)=(\frac{a}{a-b}y_1,\dots,\frac{a}{a-b}y_n,t)$. A direct calculation shows that, for every $0\le \lambda_1 < \lambda_2 \le 1$,

$$\gamma_{\lambda_1}(t) - \tilde{\gamma}(t) = \frac{R_{\lambda_1}t}{R_{\lambda_2}t} (\gamma_{\lambda_2}(t)-\tilde{\gamma}(t)).$$

This implies that the homotety $H$ wih center $\tilde{\gamma}(t)$ and ratio $\frac{R_{\lambda_1}}{R_{\lambda_2}}$ sends $\Gamma_{\lambda_2}^{\prime}(t)$ to $\Gamma_{\lambda_1}^{\prime}(t)$. Suppose $p \in \Gamma_{\lambda_1 }^{\prime} \cap \Gamma_{\lambda_2}^{\prime}$. Then, $p \ne \tilde{\gamma}(t)$, because $\tilde{\gamma}(t) \in int(D^{\prime}(t))$ and $int(D^{\prime}(t))$ in on the interior of the balls with boundary $\Gamma_{\lambda_1}^{\prime}(t)$, $\Gamma_{\lambda_2}^{\prime}(t)$. Since $\tilde{\gamma}(t)$, $p$ and $H(p)$ are collinear, we have $p=H(p)$, which implies $\frac{R_{\lambda_1}}{R_{\lambda_2}}=1$, that is $\lambda_1 = \lambda_2$, a contradiction. Therefore, $\Gamma_{\lambda_1}^{\prime} \cap \Gamma_{\lambda_2}^{\prime} = \emptyset$. An analogous argument taking $ \gamma_0(t)=\gamma(t)=\tilde{\gamma}(t)$ shows that $\Gamma_{\lambda_1}(t) \cap \Gamma_{\lambda_2}(t) = \emptyset$.

\begin{figure}[h]
\centering
\includegraphics[width=10cm]{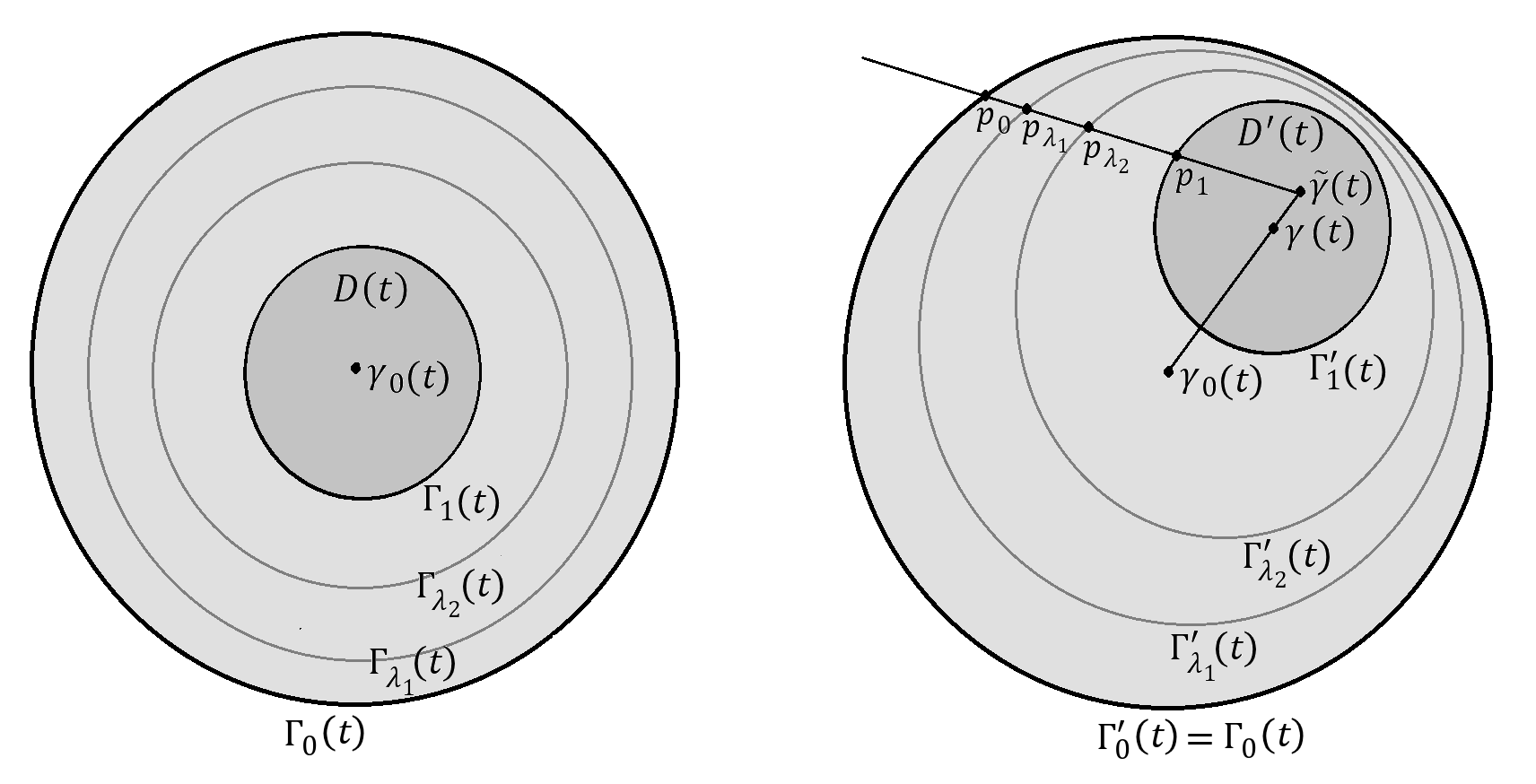}
\label{fig3}
\caption{Proof of Claim \ref{homotetias}.}
\end{figure}

\end{proof}

Now define $\varphi : (C_a^{n+1},0) \to (C_a^{n+1},0)$ as
$$\varphi(x_1,\dots, x_n,t)=
\begin{cases}
       (x_1-y_1,\dots,x_n-y_n,t), & (x_1,\dots,x_n,t) \in D(t) \\
       (x_1-\lambda y_1,\dots,x_n-\lambda y_n,t), & (x_1,\dots,x_n,t) \in \Gamma_{\lambda}(t)
   \end{cases}.
$$

Notice that for $\lambda = 1$, we have $(x_1-\lambda y_1,\dots,x_n-\lambda y_n,t)=(x_1-y_1,\dots,x_n-y_n,t)$, implying that $ \varphi$ is continuous. Also, $\varphi(D(t))=D^{\prime}(t)$, $\varphi(X')=X$, $\varphi(\Gamma_{\lambda}(t))=\Gamma_{\lambda}^{\prime}(t)$ and $\varphi$ have an inverse $\varphi^{-1}: (C_a^{n+1},0) \to (C_a^{n+1 },0)$, given by
$$\varphi^{-1}(x_1,\dots, x_n,t)=
\begin{cases}
       (x_1+y_1,\dots,x_n+y_n,t), & (x_1,\dots,x_n,t) \in D^{\prime}(t) \\
       (x_1+\lambda y_1,\dots,x_n+\lambda y_n,t), & (x_1,\dots,x_n,t) \in \Gamma_{\lambda}^{\prime}(t)
   \end{cases}.
$$

This shows that $\varphi$ is a diffeomorphism by taking the constraint on each of the open sets $D$ and $U=\bigcup_{0 < \lambda < 1} \Gamma_{\lambda}$. As $\overline{D},\overline{U}$ are path connected and have intersection equal to $\Gamma_1$, by Proposition \ref{colagem-bi-lip-inner}, it suffices to prove that $\varphi$ is inner bi-Lipschitz , since $(C_a^{n+1},0)$ is LNE (as it is convex). 

For that, let's check if $\varphi|_{(D,0)}$, $\varphi|_{(U,0)}$ satisfies the conditions of Theorem \ref{jacobiano-lip}. Given $p=(x_1, \dots, x_n, t) \in D[\varepsilon]$ calculating the Jacobian $J_{\varphi}(p)$, we have

$$
J_{\varphi} (p) = \left(
\begin{array}{ccccc}
1 & 0 & \dots & 0 & -y_1^{\prime} \\
0 & 1 & \dots & 0 & -y_2^{\prime} \\
\vdots & \vdots & \ddots & \vdots & \vdots \\
0 & 0 & \dots & 1 & -y_n^{\prime} \\
0 & 0 & \dots & 0 & 1
\end{array} \right)
$$

Since $|y_i^{\prime}|<M$, for $i=1,\dots,n$, we have $\varphi $ is bounded away from 0 and infinity in $(D,0)$, and similarly the same holds for $\varphi^{-1}$. Therefore, $\varphi|_{(D,0)}$ is a bi-Lipschitz map. Now, given $p=(x_1, \dots, x_n, t) \in U[\varepsilon]$, the $(n+1)$th row of $J_{\varphi}(p)$ is $( 0,\dots,0,1)$, showing that $\| J_{\varphi}(p) \| \ge 1$. Let us then check whether the other entries of the Jacobian are bounded. Indeed, for $1\le i \le n$, when deriving the equation $(x_1-\lambda y_1)^2+\dots+(x_n- \lambda y_n)^2 = (\lambda bt + ( 1-\lambda)at)^2$, we get:

$$\sum_{k=1,k\ne i}^{n} 2(x_k-\lambda y_k) \left( -\frac{\partial \lambda}{\partial x_i}y_k \right) + 2( x_i-\lambda y_i) \left(1 -\frac{\partial \lambda}{\partial x_i}y_i \right) = 2(\lambda bt+(1-\lambda)at)(bt-at)\frac{ \partial \lambda}{\partial x_i} \Rightarrow$$
$$\Rightarrow \left( t\frac{\partial \lambda}{\partial x_i} \right)\left( (a-b)(\lambda b + (1-\lambda)a) - \sum_{k=1 ,}^{n} \left( \frac{x_k-\lambda y_k}{t}\right) \left( \frac{y_k}{t} \right) \right) = \frac{\lambda y_i - x_i }{t}$$

As $\left| (a-b)(\lambda b + (1-\lambda)a) - \sum_{k=1,}^{n} \left( \frac{x_k-\lambda y_k}{t}\right) \left( \frac{y_k}{t} \right) \right| \ge (a-b)b-n(a+1.M)M > 0$ for $a$ big enough, and $\left| \frac{\lambda y_i - x_i}{t} \right| \le 1.M+a$, so $\left( t\frac{\partial \lambda}{\partial x_i} \right)$ is bounded. Similarly, $\left( t\frac{\partial \lambda}{\partial t} \right)$ is bounded. Therefore, for $1\le i,j \le n$,

$$ \left| \frac{\partial \varphi_{i}}{\partial x_j} \right| = \left| \frac{\partial x_i - \lambda y_i}{\partial x_j} \right| = \left| \frac{\partial x_i}{\partial x_j} - \frac{\partial \lambda}{\partial x_j} y_i \right| \le \left| \frac{\partial x_i}{\partial x_j} \right| + \left| t \frac{\partial \lambda}{\partial x_j} \right| . \left| \frac{y_i}{t} \right| \le 1 + \left| t \frac{\partial \lambda}{\partial x_j} \right| M$$
$$ \left| \frac{\partial \varphi_{i}}{\partial t} \right| = \left| \frac{\partial x_i - \lambda y_i}{\partial t} \right| = \left| \frac{\partial \lambda}{\partial t} y_i + \lambda y_i^{\prime} \right| \le \left| t \frac{\partial \lambda}{\partial t} \right|.\left| \frac{y_i}{t}\right| + \lambda . \left| y_i^{\prime} \right| \le \left| \frac{\partial \lambda}{\partial t} \right|.M + 1.M, $$

which shows that $\varphi$ is bounded away from 0 and infinity in $(U,0)$. Similarly, $\varphi^{-1}$ is bounded away from 0 and infinity in $(U^{\prime},0)$. Thus, $\varphi|_{(U,0)}$ is a bi-Lipschitz map.
\end{proof}

\begin{definition}\label{Function dilatation}\normalfont
Let $a_0>0$, $X \subset C_{a_0}^{n+1}$ be a surface and $f: (\mathbb{R},0) \to \mathbb{R}$ be a semialgebraic function germ whose Puiseux series is $f (t)=c_0+o(1)$; $c_0 >0$. The \emph{dilatation of $X$ with respect to $f$} is the set defined as
$$f \cdot X=\{(f(t)x_1 (t),\dots,f(t)x_n,t) \mid (x_1,\dots,x_n,t) \in X\}.$$
\end{definition}

\begin{lemma}[Function Dilatation Lemma] \label{dilatation}
There is $a\ge a_0$ such that $(X,0),(f \cdot X,0)$ are ambient bi-Lipschitz equivalent in $(C_a^{n+1},0)$.
\end{lemma}

\begin{proof}
Notice initially that $\displaystyle \lim_{t \to 0} f(t) = c_0$ and $\displaystyle \lim_{t \to 0} tf^{\prime}(t)=\displaystyle \lim_{t \to 0} \sum_{q\ge 0} q c_q t^q =0$. Hence, suppose that $\varepsilon>0$ is small enough such that $|t f^{\prime} (t)| < \frac{c_0}{2} < f(t) < \frac{3c_0}{2}$, for $0<t<\varepsilon$. For each $a>\max \{ 1,a_0 \}$, let $b=\sqrt{a}$.
Let us define the sets:
$$ D = \{ (x_1,\dots,x_n,t) \in C_a^{n+1}[\varepsilon] \mid x_1^2+\dots+x_n^2 \le (bt)^2 \};$$
$$ D^{\prime} = \{ (x_1,\dots,x_n,t) \in C_a^{n+1}[\varepsilon] \mid x_1^2+\dots+x_n^2 \le (f (t)bt)^2 \};$$
$$ \Gamma_{\lambda} = \{ (x_1,\dots,x_n,t) \in C_a^{n+1}[\varepsilon] \mid x_1^2+\dots+x_n^2 = (\lambda bt + (1-\lambda)at)^2 \} \, (0\le \lambda \le 1);$$
$$ \Gamma_{\lambda}^{\prime} = \{ (x_1,\dots,x_n,t) \in C_a^{n+1}[\varepsilon] \mid x_1^2+\dots+x_n^ 2 = (\lambda f(t)bt + (1-\lambda)at)^2 \} \, (0\le \lambda \le 1).$$

Similarly to Claim \ref{homotetias}, we can prove the following:

\begin{claim}
There is $a>0$ large enough such that
\begin{enumerate}
\item $X(t) \subset D(t)$ and $f \cdot X(t) \subset D^{\prime}(t)$, for each $0<t\ll 1$;
\item the $n$-spheres $D(t),D^{\prime}(t)$ are inside $C_a^{n+1}(t)$, for each $0<t\ll 1$;
\item for $0\le \lambda_1 < \lambda_2 \le 1$, we have $\Gamma_{\lambda_1}(t) \cap \Gamma_{\lambda_2}(t) = \emptyset$ and $\Gamma_{\lambda_1 }^{\prime} \cap \Gamma_{\lambda_2}^{\prime} = \emptyset$.
\end{enumerate}
\end{claim}

Now define $\varphi : (C_a^{n+1},0) \to (C_a^{n+1},0)$ as
$$\varphi(x_1,\dots, x_n,t)=
\begin{cases}
       (f(t)x_1,\dots,f(t)x_n,t), & (x_1,\dots,x_n,t) \in D(t) \\
       \left( \left( \frac{\lambda f(t)b + (1-\lambda)a}{\lambda b + (1-\lambda) a}\right)x_1,\dots,\left( \frac{\lambda f(t)b + (1-\lambda)a}{\lambda b + (1-\lambda) a}\right)x_n,t \right), & (x_1,\dots,x_n, t) \in \Gamma_{\lambda}(t)
   \end{cases}.
$$

Notice that for $\lambda = 1$,
$$(f(t)x_1,\dots,f(t)x_n,t) = \left( \left( \frac{\lambda f(t)b + (1-\lambda)a}{\lambda b + (1-\lambda) a}\right)x_1,\dots,\left( \frac{\lambda f(t)b + (1-\lambda)a}{\lambda b + (1-\lambda) a}\right)x_n,t \right),$$
implying that $\varphi$ is continuous. Also, $\varphi(D(t))=D(t)^{\prime}$, $\varphi(X')=X$, $\varphi(\Gamma_{\lambda}(t))=\Gamma_{\lambda}^{\prime}(t)$ and $\varphi$ has a continuous inverse $\varphi^{-1}: (C_a^{n+1},0) \to (C_a^{n+ 1},0)$, given by
$$\varphi^{-1}(x_1,\dots, x_n,t)=
\begin{cases}
       (\frac{1}{f(t)}x_1,\frac{1}{f(t)}x_n,t), & (x_1,\dots,x_n,t) \in D^{\prime}( t) \\
       \left( \left( \frac{\lambda b + (1-\lambda) a}{\lambda f(t)b + (1-\lambda)a}\right)x_1,\dots,\left( \frac{\lambda b + (1-\lambda) a}{\lambda f(t)b + (1-\lambda)a}\right)x_n,t \right), & (x_1,\dots,x_n, t) \in \Gamma_{\lambda}^{\prime}(t)
   \end{cases}.
$$

This shows that $\varphi$ is a diffeomorphism by taking the constraint on each of the open sets $D$ and $U=\bigcup_{0 < \lambda < 1} \Gamma_{\lambda}$. As $\overline{D},\overline{U}$ are path connected and have intersection equal to $\Gamma_1$, by Proposition \ref{colagem-bi-lip-inner}, it suffices to prove that $\varphi$ is inner bi-Lipschitz , since $(C_a^{n+1},0)$ is LNE (as it is convex). For that, let's check if $\varphi|_{(D,0)}$, $\varphi|_{(U,0)}$ satisfies the conditions of Theorem \ref{jacobiano-lip}. Given $p=(x_1, \dots, x_n, t) \in D[\varepsilon]$ calculating the Jacobian $J_{\varphi}(p)$, we have

$$
J_{\varphi} (p) = \left(
\begin{array}{ccccc}
f(t) & 0 & \dots & 0 & x_1 f^{\prime}(t) \\
0 & f(t) & \dots & 0 & x_2 f^{\prime}(t) \\
\vdots & \vdots & \ddots & \vdots & \vdots \\
0 & 0 & \dots & f(t) & x_n f^{\prime}(t) \\
0 & 0 & \dots & 0 & 1
\end{array} \right)
$$

Hence, $\det{J_{\varphi} (p)}= f(t)^n \in \left( \left( \frac{c_0}{2} \right)^n, \left( \frac{ 3c_0}{2} \right)^n \right)$ and as $|f(t)|<\frac{3c_0}{2}$, $|x_i f^{\prime}(t)| =\left| \frac{x_i}{t} \right|.|t f^{\prime}(t)| <a.\frac{c_0}{2}$, for $i=1,\dots,n$, we have that $\varphi$ is bounded away from 0 and infinity in $(D,0)$. Therefore, $\varphi|_{(D,0)}$ is a bi-Lipschitz map. Now, given $p=(x_1, \dots, x_n, t) \in U[\varepsilon]$, the $(n+1)$th row of $J_{\varphi}(p)$ is $( 0,\dots,0,1)$, showing that $\| J_{\varphi}(p) \| \ge 1$. Let us then check whether the other entries of the Jacobian are bounded. Indeed, for $1\le i \le n$, when deriving the equation $x_1^2+\dots+x_n^2 = (\lambda bt + (1-\lambda)at)^2$, we get
$$2x_i = 2(\lambda bt+(1-\lambda )at)(bt-at)\frac{\partial \lambda}{\partial x_i} \Rightarrow \left( t\frac{\partial \lambda}{ \partial x_i} \right)\left( (a-b)(\lambda b + (1-\lambda)a) \right)= \frac{x_i}{t}.$$

Since $\left| (a-b)(\lambda b + (1-\lambda)a) \right| \ge (a-b)b> 0$ and $\left| \frac{x_i}{t} \right| \le a$, $\left( t\frac{\partial \lambda}{\partial x_i} \right)$ is bounded. Similarly, $\left( t\frac{\partial \lambda}{\partial t} \right)$ is also bounded. Denoting $k=\frac{\lambda b + (1-\lambda) a}{\lambda f(t)b + (1-\lambda)a}$, we have $\frac{f(b) t}{a} \le k \le \frac{a}{b}$ and
$$t\frac{\partial k}{\partial x_i} = \frac{\left( b-a \right) \left( t \frac{\partial \lambda}{\partial x_i} \right)(\lambda f (t)b + (1-\lambda)a) - (\lambda b + (1-\lambda)a) \left( t \frac{\partial \lambda}{\partial x_i} \right)(f( t)b-a)}{(\lambda f(t)b + (1-\lambda)a)^2};$$
$$t\frac{\partial k}{\partial t} = \frac{\left( b-a \right) \left( t \frac{\partial \lambda}{\partial t} \right)(\lambda f (t)b + (1-\lambda)a) - (\lambda b + (1-\lambda)a)\left( \left( t \frac{\partial \lambda}{\partial t} \right) (f(t)b-a) + b \lambda(tf^{\prime}(t)) \right)}{(\lambda f(t)b + (1-\lambda)a)^2}.$$

Thus, such expressions are bounded, because $\lambda f(t)b + (1-\lambda)a \ge f(t)b$ and all terms in the numerator are bounded. Then, for $1\le i,j \le n$,
$$ \left| \frac{\partial \varphi_{i}}{\partial x_j} \right| = \left| \frac{\partial (kx_i)}{\partial x_j} \right| = \left| k \frac{\partial x_i}{\partial x_j} + \frac{\partial k}{\partial x_j} x_i \right| \le k + \left| t \frac{\partial k}{\partial x_j} \right| . \left| \frac{x_i}{t} \right| \le \frac{a}{b} + \left| t \frac{\partial \lambda}{\partial x_j} \right| a$$
$$ \left| \frac{\partial \varphi_{i}}{\partial t} \right| = \left| \frac{\partial (kx_i)}{\partial t} \right| =\left| t\frac{\partial k}{\partial t} \right| . \left| \frac{ x_i}{t} \right| \le \left|t \frac{\partial k}{\partial t} \right|a,$$

which shows that $\varphi$ is bounded away from 0 and infinity in $(U,0)$. Similarly, $\varphi^{-1}$ is bounded away from 0 and infinity in $(U^{\prime},0)$. Therefore, $\varphi|_{(U,0)}$ is a bi-Lipschitz map.
\end{proof}

\begin{corollary}\label{ajuste-cornetas}
Let $g: (\mathbb{R},0) \to \mathbb{R}$ be a semialgebraic function germ whose Puiseux series is $g(t)=c_{\beta}t^{\beta}+o (t^{\beta})$; $c_\beta >0$, for some rational $\beta \ge 1$. So, the germ at $0$ of
$$X = \{ (x_1, x_2, t) \in C_a^{3} \mid x_1^2 + x_2^2 = g(t)^2 \}$$
is ambient bi-Lipschitz equivalent, in $C_a^3$, to the standard $\beta$-horn germ.
\end{corollary}

\begin{proof}
Let $\psi$ be the map of Proposition \ref{reduction}. The germ of $\psi(H_{\beta})$ is $(X^{\prime},0)$, where
$$X^{\prime} = \{ (x_1, x_2, t) \in C_a^{3} \mid x_1^2 + x_2^2 = h(t)^2 \},$$

for some function germ $h: (\mathbb{R},0) \to \mathbb{R}$ whose Puiseux series is $h(t)=d_{\beta}t^{\beta}+o( t^{\beta})$, $d_\beta >0$. By the Function dilatation Lemma, applied to $f(t)=\frac{h(t)}{g(t)}$, we have that $(X,0)$ is ambient bi-Lipschitz equivalent to $(X^{ \prime},0)$ in $C_a^{3}$. From Proposition \ref{suficiencia-em-cone}, $(X,0)$ is ambient bi-Lipschitz equivalent to $(H_{\beta},0)$.
\end{proof}

\begin{figure}[h]
\centering
\includegraphics[width=12cm]{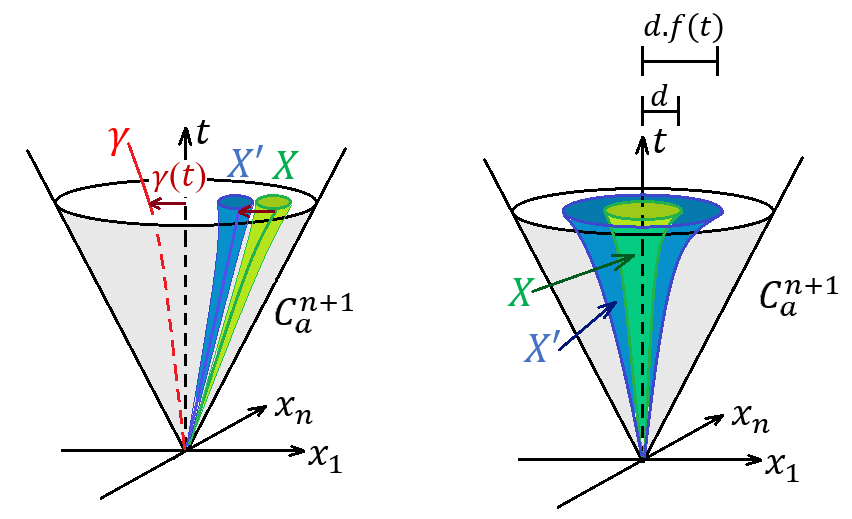}
\label{fig32}
\caption{Graphic representations of Lemma \ref{translation} (left) and Lemma \ref{dilatation} (right).}
\end{figure}

\section{Synchronized Triangles and Curvilinear Rectangles}

\subsection{Definitions and Examples}

We define the concepts of synchronized triangles, curvilinear rectangles and regions delimited by triangles and we will also see some examples of such objects, as well as important classifications, such as the arc coordinate system. The main idea is to select a suitable subset of the Valette link of such triangles that shall be useful for construction of ambient bi-Lipschitz maps.

\begin{definition}\normalfont \label{synch-triangle-def}
Let $a>0$, let $(\gamma_0,0),(\gamma_1,0) \subseteq (C_{a}^{3},0)$ be two distinct arc and let $ (T,0) \subseteq (C_{a}^{3},0)$ a H\"older triangle germ whose boundary arcs are $(\gamma_0,0),(\gamma_1,0)$. We say that $(T,0)$ is a \textit{synchronized triangle germ} if $\gamma_0(0)=\gamma_1(0)=(0,0,0)$ and for all $t>0$ small enough
\begin{itemize}
\item $\gamma_0 = \gamma_0(t) = (x_0(t),y_0(t),t)$, $\gamma_1 = \gamma_1(t) = (x_1(t),y_1(t),t)$ and $x_0(t)<x_1(t)$;
\item $\pi_z (T\cap\{z= t\})$ is the graph of a semialgebraic function $f_t : [x_0(t),x_1(t)] \to \mathbb{R}$ satisfying $f_t (x_0 (t)) = y_0(t)$ and $f_t (x_1 (t)) = y_1(t)$.
\end{itemize}
Here, $\pi_z : \mathbb{R}^3 \to \mathbb{R}^2$; $\pi_z(x,y,z)=(x,y)$ is the $z$-projection. The set of functions $\{ f_t \}_{0<t\ll 1}$ is the $\textit{generating functions family}$ of $(T,0)$.

\begin{figure}[h]
\centering
\includegraphics[width=11cm]{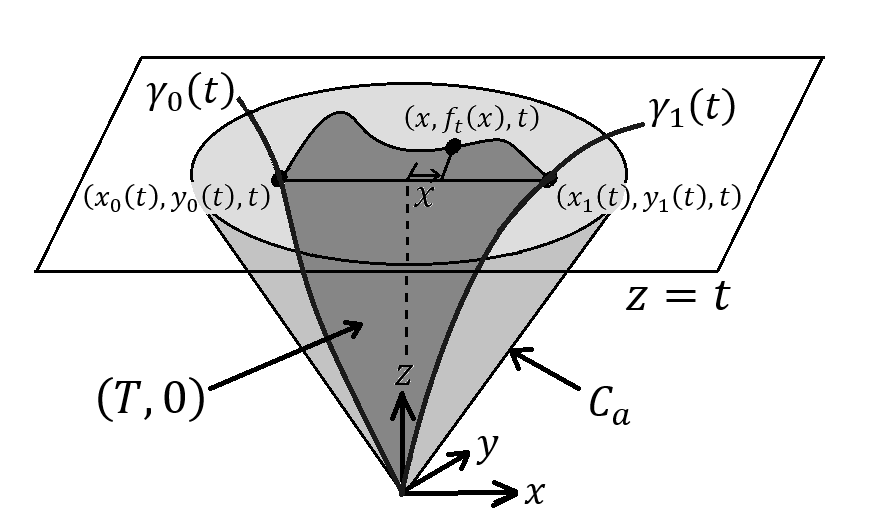}
\label{fig4}
\caption{Synchronized triangle germ.}
\end{figure}
\end{definition}

\begin{remark} \label{decomp-synch-differential}
If $(T,0)$ is a synchronized triangle and $\{f_t \}_{0 < t < \varepsilon}$ is its generating functions family, we have that $(sing(T),0)$ is a union of finitely many arcs $(\delta_i,0)$, with $\delta_i (t) = (r_i(t),s_i(t),t)$ , for each $i$. We can also assume, reordering the indices if necessary, that $x_0(t) = r_1(t) < \dots < r_n(t) = x_1(t)$ and that each $f_t$ is differentiable in $\bigcup_{i= 1}^{n} (r_i(t),r_{i+1}(t))$. That is why $f_t$ is piecewise smooth, for all $t>0$ small enough.
\end{remark}

\begin{definition} \label{arc-coord}\normalfont
For every $t>0$ small enough and for every $u\in [0,1]$, let
$$x_u (t)= u \cdot x_1 (t) + (1- u) \cdot x_0 (t) \in [x_0(t), x_1(t)].$$

For every point $P \in (T,0)$, there are unique $t>0$ and $u \in [0,1]$ such that $P=(x_u (t), f_t (x_u (t)),t)$. Thus, if we let for each $u \in [0,1]$ the arcs $\gamma_{u}$ in $C_a^{3}$ as $\gamma_u(0)=0$ and $\gamma_{u} (t) = (x_u (t), f_t(x_u(t)),t)$, for $0<t\ll 1$, then the family of arcs $\gamma_{u} \subset V(T)$ have union equal to $(T,0)$. We define $\{\gamma_{u}\}_{0\le u\le 1}$ as the \textit{arc coordinate system of the synchronized triangle $(T,0)$}. Since $(\gamma_{u},0)$ is semialgebraic for each $0\le u \le 1$, the limit $\displaystyle \lim_{t \to 0 ^{+}} (\gamma_{u}^{\prime} (t))$ exists and is a vector, which is denoted by $\gamma_{u}^{\prime}(0)$.
\end{definition}

\begin{definition}\normalfont \label{synch-C1}
Given $M>0$, we say that a synchronized triangle $(T,0)$, with generating functions family $\{f_t\}$, is $\textit{M-bounded}$ if for every $t>0$ small enough and for all $a\in [x_0(t),x_1(t)]$, the lateral derivatives of $f_t$ in $a$ are bounded by $M$, that is
$$M > \left| \dfrac { \partial f_t}{\partial x_{+}}(a) \right| = \lim_{h \to 0^+} \dfrac { f_t (a+h) - f_t (a)}{h} \, ; \, x_0(t) \le a < x_1(t);$$
$$M > \left| \dfrac { \partial f_t}{\partial x_{-}}(a) \right| = \lim_{h \to 0^-} \dfrac { f_t (a+h) - f_t (a)}{h} \, ; \, x_0(t) < a \le x_1(t).$$

We also say that a synchronized triangle $(T,0)$ is $\textit{of class $C^1$}$ (or $C^1$, for short) if $(T,0)$ is a curvilinear triangle (see Definition \ref{curvilinear}). Finally, we say that a synchronized triangle $(T,0)$, which is $C^1$, is $\textit{convex}$ if every function of its generating functions family is convex.
\end{definition}

\begin{example}\label{ex-1-synch}
Consider the set
$$T= \{ (x,t^2-x^2,t) \mid 0 < t \le 1; -t \le x \le t \} \subset C_{\sqrt{2}}.$$

Notice that $(T,0)$ is a H\"older triangle with main vertex at the origin, and whose boundary arcs are the germs $(\gamma_0,0),(\gamma_1,0)$, where $\gamma_0( t)= (-t,0,t); \gamma_1(t) = (t,0,t)$, for $0< t \le 1$. $(T,0)$ is also a synchronized triangle germ, where for every $t \in [0,1)$,
$$ f_t : [-t,t] \to \mathbb{R}\, ; \, f_t(x) = t^2 - x^2$$
is generating functions family of $(T,0)$. We have
$$\frac{\partial f_t}{\partial x}(a) = -2a \in [-2t,2t] \subseteq [-2,2]\, ; \, \forall \, 0<t\le 1 \, , \, -t < a < t,$$
which shows that $(T,0)$ is $2$-bounded and is $C^1$.
\end{example}

\begin{example}\label{ex-2-synch}
Consider the set in $C_{\sqrt{2}}$
$$T= \left\{ \left( x,\frac{x}{t}+t,t\right) \mid 0 < t < 1; -t^2 \le x \le 0 \right\} \cup \left\{ \left( x,-\frac{x}{t}+t,t \right) \mid 0 < t < 1; 0 \le x \le t^2 \right\}.$$

Notice that $(T,0)$ is a H\"older triangle with main vertex at the origin, and whose boundary arcs are the germs $(\gamma_0,0),(\gamma_1,0)$, where $\gamma_0( t)= (-t^2,0,t); \gamma_1(t) = (t^2,0,t)$, for $0< t < 1$. $(T,0)$ is also a synchronized triangle, where, for every $t \in (0,1)$,
$$ f_t : [-t,t] \to \mathbb{R}; \; f_t(x) = \begin{cases}
       \frac{x}{t}+t, & x \in [-t^2,0] \\
       -\frac{x}{t}+t, & x \in [0,t^2]
   \end{cases}$$
is the generating functions family of $(T,0)$. Notice also that $(T,0)$ is synchronized but not of class $C^1$. Since
$$ \left| \frac{\partial f_t}{\partial x}(a) \right| = \dfrac{1}{t}; \forall \, 0< t < 1, a \in (-t^2,0)\cup(0,t^2),$$
we have that $(T,0)$ is not $M$-bounded, for all $M > 0$, because $\lim_{t \to 0^{+}} \frac{1} {t} = +\infty$.
\end{example}

\begin{definition}\normalfont \label{curvilinear-rectangle}
Let $a>0$ and let $(T_1,0), (T_2,0) \subset (C_a^{3},0)$ be synchronized triangle germs. We say that \textit{ $(T_1,0), (T_2,0)$ are aligned on boundary arcs} if there are four distinct curve germs $(\gamma_0,0),(\gamma_1,0),(\rho_0 ,0),(\rho_1,0) \subseteq (C_{a}^{3},0)$, such that, for all $t>0$ small enough,
\begin{itemize}
\item $\gamma_0 = \gamma_0(t) = (x_0(t),y_0(t),t), \gamma_1 = \gamma_1(t) = (x_1(t),y_1(t),t)$;
\item $\rho_0 = \rho_0(t) = (x_0(t),w_0(t),t), \gamma_1 = \gamma_1(t) = (x_1(t),w_1(t),t)$;
\item $(\gamma_0,0),(\gamma_1,0)$ are the boundary arcs of $(T_1,0)$ and $(\rho_0, 0),(\rho_1,0)$ are the boundary arcs of $(T_2,0)$, with $x_0(t)<x_1(t)$.
\end{itemize}

Suppose further that $\{f_t\}, \{g_t\}$ are the generating function families of $(T_1,0)$ and $(T_2,0)$, respectively, and for all $t >0$ small enough,
\begin{itemize}
\item $f_t(x_0)=y_0(t)$, $f_t(x_1)=y_1(t)$, $g_t(x_0)=w_0(t)$ and $g_t(x_1)=w_1(t)$;
\item $f_t(x) \ge g_t(x)$, for all $x_0(t) \le x \le x_1(t)$.
\end{itemize}

We define the \textit{curvilinear rectangle bounded by $(T_1,0),(T_2,0)$} as the germ of the following semialgebraic set:
$$R = \{ (x,y,t)\in C_a^{3} \; \mid \; x_0(t) \le x \le x_1 (t) \; ; \; g_t(x) \le y \le f_t(x) \}$$

If $(\gamma_0,0)=(\rho_0,0)$ and $(\gamma_1, 0) = (\rho_1,0)$, the curvilinear rectangle bounded by $(T_1,0),(T_2,0 )$ is also called the \textit{region bounded by $(T_1,0),(T_2,0)$}.

\begin{figure}[h]
\centering
\includegraphics[width=12cm]{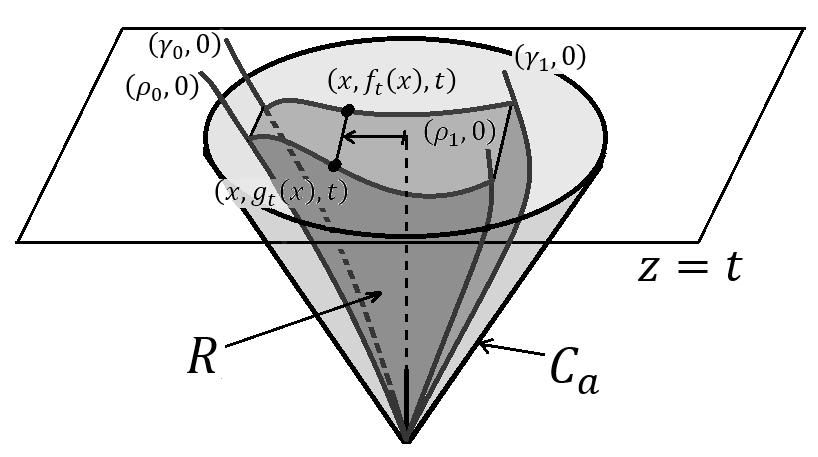}
\label{fig5}
\caption{Curvilinear rectangle bounded by $(T_1,0)$, $(T_2,0)$.}
\end{figure}
\end{definition}

\begin{definition}\normalfont \label{arc-coord-curv-rectangle}
For each $t>0$ and for each $(u,v)\in [0,1]^2$, let
$$x_u (t)= u \cdot x_1 (t) + (1- u) \cdot x_0 (t) \in [x_0(t), x_1(t)];$$
$$\sigma_{u,v} (t)= v \cdot f_t (x_u(t)) + (1- v) \cdot g_t (x_u(t)) \in [g_t (x_u(t) ), f_t (x_u(t))].$$

For every point $P \in (R,0)$, there are $t>0$, $(u,v) \in [0,1]$ such that $P=(x_u (t ), \sigma_{u,v} (t),t)$. Furthermore, for each $P\in (R,0)$, $u$ and $t$ are unique and if $f_t(x_u(t)) > g_t(x_u(t))$, $v$ will also be unique (if $f_t(x_u(t)) = g_t(x_u(t))$, we have $P=(x_u (t), \sigma_{u,v} (t),t)$ for every $v \in [0,1]$). Thus, if we define, for each $(u,v) \in [0,1]^2$, the semialgebraic arc $\gamma_{u,v}$ in $(C_a^{3},0)$ as $\gamma_{u,v}(0)=0$ and $\gamma_{u,v} (t) = (x_u (t), \sigma_{u,v} (t),t) \; ; \;t>0$, then the family of arcs $\gamma_{u,v} \subset V(R)$ have union equal to $(R,0)$. We define $\{\gamma_{u,v}\}_{(u,v)\in [0,1]^2}$ as the \textit{arc coordinate system of the curvilinear rectangle bounded by $ ( T_1,0),(T_2,0)$}. Since $(\gamma_{u,v},0)$ is semialgebraic for each $0\le u,v \le 1$, the limit $\displaystyle \lim_{t \to 0^{+}} (\gamma_{u,v}^{\prime} (t))$ exists and is a vector, which is denoted by $\gamma_{u,v}^{\prime }(0)$.
\end{definition}

\subsection{Basic Properties}

We show now some basic structural properties of synchronized triangles and sets delimited by triangles aligned on boundary arcs. Consider all sets as in the definitions of this section.

\begin{proposition}\label{curv-rect-is-LNE}
Let $(T_1,0), (T_2,0) \subset (C_a^3,0)$ be germs of $C^1$, $M$-bounded synchronized triangles, aligned on the boundary arcs. Then the curvilinear rectangle bounded by $(T_1,0)$ and $(T_2,0)$ is the germ of a semialgebraic LNE set.
\end{proposition}

\begin{proof}
Since each $f_t$, $g_t$ have derivative bounded by $M$, by Proposition \ref{grafico-LNE}, the link of the curvilinear rectangle bounded by $(T_1,0)$ and $(T_2,0)$ is $C$-LNE, for some uniform constant $C>0$. By Theorem \ref{Edson-Rodrigo}, $(R,0)$ is LNE.
\end{proof}

\begin{corollary}
Let $(T,0) \subset (C_a^3,0)$ be a $C^1$, $M$-bounded synchronized triangle germ, and let $\{ \gamma_u \}_{0\le u\le 1}$ be the arc coordinate system of $(T,0)$. If $(T,0)$ is an $\alpha$-H\"older triangle, for some $\alpha \in \mathbb{Q}_{\ge 1}$, then for every $0\le u_1 < u_2 \le 1$ we have
$$tord(\gamma_{u_1},\gamma_{u_2})=tord_{T}(\gamma_{u_1},\gamma_{u_2})=\alpha.$$
\end{corollary}

\begin{proposition}\label{synch-derivative-along-t}
Let $a,M>0$ and let $(T,0) \subset (C_a^3,0)$ be a $C^1$, $M$-bounded synchronized triangle germ. If $\{ f_t \}$ is the generating functions family of $(T,0)$, then there is $\varepsilon>0$ small enough and $N>1$ large enough such that, for all $0<t < \varepsilon$, $x \in (x_0(t),x_1(t))$ and $u \in [0,1]$,
$$\left| \dfrac{\partial}{\partial t}(f_t)(x) \right| \; , \; \left| \gamma_{u}^{\prime}(t) \right| < N.$$
\end{proposition}

\begin{proof}
The following result can be seen as a straightforward consequence of \cite{Kurdyka1992OnAS}. Here, we will give a different proof, using the curve selection lemma. Since $(T,0)$ is $C^1$, we can take $\varepsilon_0 >0$ such that $\dfrac{\partial}{\partial t}(f_t)(x)$ is well defined, for every $0<t < 2\varepsilon_0$ and every $x_0(t) < x < x_1(t)$. For each $n \in \mathbb{Z}_{\ge 1}$, let
$$A_n = \left\{ (t,x) \in (0,2\varepsilon_0) \times \mathbb{R} \mid x_0(t) < x < x_1(t) ; \left| \dfrac{\partial}{\partial t}(f_t)(x) \right| > n \right\}.$$

Suppose that for all $n \in \mathbb{Z}_{\ge 1}$ and all $t >0$ small enough, we have $A_n \cap C_a^3(t) \ne \emptyset$. As $A_n$ is semialgebraic, by the curve selection lemma there is an arc $(\psi_n,0) \subseteq (A_n,0)$. Writing $\psi_n = \psi_n(t) = (x(t), y(t),t)$ and noting that $\psi_n$ is semialgebraic, there exists $0< \varepsilon < \varepsilon_0$ small enough so that $ \psi_n^{\prime}(t) = (x^{\prime}(t),y^{\prime}(t),1)$ is well defined and continuous, for all $t \in [0 ,\varepsilon]$, where $\psi_n^{\prime} (0) =\displaystyle \lim_{t \to 0^{+}} (\psi_n^{\prime} (t))$. As $(\psi_n,0) \subset (C_a^3,0)$, we have $\psi_n^{\prime} (0) = (x^{\prime}(0),y^{\prime}( 0),1) \in C_a^3$, where $(x^{\prime}(0))^2+(y^{\prime}(0))^2 \le a^2$ and then $ \left| x^{\prime}(0) \right|, \left| y^{\prime}(0) \right| \le a$. Thus, by the continuity of $\psi_n^{\prime}$ around 0,assume that $\varepsilon >0$ is small enough such that $\left| x^{\prime}(\varepsilon) \right|, \left| y^{\prime}(\varepsilon) \right| \le a+a = 2a$.

Consider $\psi_n(\varepsilon) = (x,y,\varepsilon)$ and, for each $t_0 \in \mathbb{R}$ of sufficiently small absolute value, consider the points
$$P(t_0) = \psi_n(\varepsilon+t_0) = (x_{P(t_0)},f_{\varepsilon + t_0}(x_{P(t_0)}),\varepsilon+t_0),$$
$$Q(t_0) = (x,f_{\varepsilon+t_0}(x),\varepsilon+t_0).$$
Notice that $P(0)=Q(0)=(x,y,\varepsilon)$. Using that
$$\psi_n^{\prime} (\varepsilon) = (x^{\prime}(\varepsilon),y^{\prime}(\varepsilon),1) = (r_1,s_1,1),$$
$$Q^{\prime}(\varepsilon) = \left( 0, \dfrac{\partial}{\partial t}(f_t)\left( \varepsilon \right),1 \right) = (0,s_2 ,1),$$
we can write
$$x_{P(t_0)} = x + r_1 \cdot t_0 + o(t_0),$$
$$f_{\varepsilon+t_0}(x_{P(t_0)}) = y + s_1 \cdot t_0 + o(t_0),$$
$$x_{Q(t_0)} = x, f_{\varepsilon+t_0}(x) = y + s_2 \cdot t_0 + o(t_0).$$
Thus, for $t_0 \ne 0$ small enough, and from $|r_1|,|s_1| \le 2a, |s_2| >n$, we obtain
$$\left| x_{P(t_0)} - x \right| = \left| r_1\cdot t_0 + o(t_0) \right|< (|r_1|+1)\cdot |t_0| \le (2a+1) \cdot |t_0|$$
$$\left| f_{\varepsilon+t_0}(x_{P(t_0)}) - f_{\varepsilon+t_0}(x) \right| \ge \left| f_{\varepsilon+t_0}(x) - y \right| - \left| f_{\varepsilon+t_0}(x_{P(t_0)}) - y \right| =$$
$$=\left| s_2\cdot t_0 + o(t_0) \right| - \left| s_1\cdot t_0 + o(t_0) \right| > (|s_2| - 1)\cdot |t_0| - (|s_1| + 1) \cdot |t_0| = $$
$$= (|s_2| - |s_1| - 2) \cdot |t_0| > (n-2a-2) \cdot |t_0|.$$
Therefore, since each $f_t$ has derivative bounded by $M$, we get
$$M \ge \left| \dfrac{f_{\varepsilon+t_0}(x_{P(t_0)}) - f_{\varepsilon+t_0}(x)}{x_{P(t_0)} - x} \right| > \dfrac{(n-2a-2) \cdot |t_0|}{(2a+1) \cdot |t_0|} = \dfrac{n-2a-2 }{2a+1},$$
a contradiction for large enough $n$. Then, there is $n \in \mathbb{Z}_{\ge 1}$ and $\varepsilon > 0$ small enough such that $A_n \cap C_a^3[\varepsilon] = \emptyset$. Since $A_1 \supseteq A_2 \supseteq \dots$, it follows that $A_n \cap C_a^3[\varepsilon] = \emptyset$, for all $n \ge N$. We conclude that $\left| \dfrac{\partial}{\partial t}(f_t)(x) \right| < N$, for all $0 < t < \varepsilon$ and $x_0(t) < x < x_1(t)$, proving the first inequality.

The proof of the second inequality is analogous, and we will write it for completeness. The arcs $(\gamma_0 , 0)$ and $(\gamma_1, 0 )$ are semialgebraic and are in $(C_a^3,0)$. Taking the Puiseux series at each coordinate, there are $\alpha,\beta \in \mathbb{Q}_{\ge 1}$ and $r,s>0$ such that
$$ x_0(t) = r\cdot t^{\alpha} + o(t^\alpha)$$
$$x_1(t) - x_0(t) = s \cdot t^{\beta}+o(t^\beta).$$

As $x_u (t)= u \cdot x_1 (t) + (1- u) \cdot x_0 (t) = x_0(t) + u \cdot (x_1(t)-x_0(t))$ , when differentiating in $t$, we obtain
$$x_u^{\prime} (t) = r\alpha \cdot t^{\alpha - 1}+o(t^{\alpha -1})+ u \cdot (s \beta \cdot t^ {\beta -1} +o(t^{\beta -1})).$$
Hence, for $\varepsilon_0 > 0$ small enough, $\left| x_u^{\prime} (t) \right| < N_0$, for all $u\in [0,1]$ and for some constant $N_0>0$ that does not depend on $t$ and $u$. Now, for every $n \in \mathbb{Z}_{\ge 1}$, define
$$B_n = \left\{ (u,t) \in (0,1) \times (0,2\varepsilon_0) \mid \left| \gamma_{u}^{\prime} (t) \right| > n \right\}.$$
Suppose that for every $n \in \mathbb{Z}_{\ge 1}$ and every $\varepsilon >0$ small enough, $B_n \cap C_a^3(\varepsilon) \ne \emptyset$. As $B_n$ is semialgebraic, by the curve selection lemma there are arcs $(\psi_n,0) \subseteq (B_n,0)$. Writing $\psi_n = \psi_n(t) = (x(t), y(t),t)$ and noting that $\psi_n$ is semialgebraic, there exists $0< \varepsilon < \varepsilon_0$ such that $\psi_n ^{\prime}(t) = (x^{\prime}(t),y^{\prime}(t),1)$ is well defined and continuous, for all $t \in [0,\varepsilon]$, where $\psi_n^{\prime} (0) =\displaystyle \lim_{t \to 0^{+}} (\psi_n^{\prime} (t))$. As $(\psi_n,0) \subset (C_a^3,0)$, we have $\psi_n^{\prime} (0) = (x^{\prime}(0),y^{\prime}( 0),1) \in C_a^3$, where $(x^{\prime}(0))^2+(y^{\prime}(0))^2 \le a^2$ and then $ \left| x^{\prime}(0) \right|, \left| y^{\prime}(0) \right| \le a$. Thus, by the continuity of $\psi_n^{\prime}$, we can assume that $\varepsilon$ is small enough such that $\left| x^{\prime}(\varepsilon) \right|, \left| y^{\prime}(\varepsilon) \right| \le a+a = 2a$.

Consider $\psi_n(\varepsilon) = \gamma_{u}(\varepsilon)$, for some $u \in [0,1]$, and, for every $t_0 \in \mathbb{R}$ of sufficiently small absolute value, consider the points
$$P(t_0) = \psi_n(\varepsilon+t_0) = (x_{P(t_0)},f_{\varepsilon + t_0}(x_{P(t_0)}),\varepsilon+t_0),$$
$$Q(t_0) = \gamma_{u}(\varepsilon + t_0)=(x_{Q(t_0)},f_{\varepsilon + t_0}(x_{Q(t_0)}),\varepsilon+t_0).$$
Notice that $P(0)=Q(0)=(x,y,\varepsilon)$. Since
$$\psi_n^{\prime} (\varepsilon) = (x^{\prime}(\varepsilon),y^{\prime}(\varepsilon),1) = (r_1,s_1,1),$$
$$\gamma_{u}^{\prime} (\varepsilon) = \left( x_u^{\prime}(\varepsilon),(f_{t}(x_u(t)))^{\prime} (\varepsilon),1 \right) = (r_2,s_2,1),$$
we have
$$x_{P(t_0)} = x + r_1 \cdot t_0 + o(t_0),$$
$$f_{\varepsilon+t_0}(x_{P(t_0)}) = y + s_1 \cdot t_0 + o(t_0),$$
$$x_{Q(t_0)} = x + r_2 \cdot t_0 + o(t_0),$$
$$f_{\varepsilon+t_0}(x_{Q(t_0)}) = y + s_2 \cdot t_0 + o(t_0).$$
Thus, for $t_0 \ne 0$ small enough, and from $|r_1|,|s_1| \le 2a, |r_2| < N_0, |s_2| >n$, we obtain
$$\left| x_{P(t_0)} - x_{Q(t_0)} \right| \le \left|x_{P(t_0)} - x \right| + \left|x - x_{Q(t_0)}\right| = \left| r_1\cdot t_0 + o(t_0) \right| + \left| r_2\cdot t_0 + o(t_0) \right| <$$
$$<(|r_1|+1)\cdot |t_0| + (|r_2|+1)\cdot |t_0| = (|r_1|+|r_2|+2)\cdot |t_0| < (2a+N_0+2) \cdot |t_0|$$
$$\left| f_{\varepsilon+t_0}(x_{P(t_0)}) - f_{\varepsilon+t_0}(x_{Q(t_0)}) \right| \ge \left| f_{\varepsilon+t_0}(x_{Q(t_0)}) - y \right| - \left| f_{\varepsilon+t_0}(x_{P(t_0)} - y) \right| $$
$$=\left| s_2\cdot t_0 + o(t_0) \right| - \left| s_1\cdot t_0 + o(t_0) \right| > (|s_2| - 1)\cdot |t_0| - (|s_1| + 1) \cdot |t_0| = $$
$$= (|s_2| - |s_1| - 2) \cdot |t_0| > (n - 2a -2) \cdot |t_0|.$$
Therefore, as each $f_t$ has derivative bounded by $M$,
$$M > \left| \dfrac{f_{\varepsilon+t_0}(x_{P(t_0)}) - f_{\varepsilon+t_0}(x_{Q(t_0)})}{x_{P(t_0)} - x_{Q( t_0)}} \right| > \dfrac{(n-2a-2) \cdot |t_0|}{(2a+N_0+2) \cdot |t_0|} = \dfrac{n-2a-2}{2a+N_0+2},$$
a contradiction for $n$ large enough. Notice that $a,N_0$ does not depend on $n$. Therefore, there are $n \in \mathbb{Z}_{\ge 1}$ and $\varepsilon > 0$ small enough such that $B_n \cap C_a^3[\varepsilon] = \emptyset$. Since $B_1 \supset B_2 \supset \dots$, it follows that $B_n \cap C_a^3[\varepsilon] = \emptyset$, for every $\varepsilon > 0$ small enough and every large $n \in \mathbb{N }$.
\end{proof}

\begin{corollary}\label{synch-derivative-along-t-curv-rect}
Let $a,M>0$ and let $(T_1,0), (T_2,0) \subset (C_a^3,0)$ be $C^1$, $M$-bounded synchronized triangles, aligned on the boundary arcs. Let $\{ \gamma_{u,v} \}_{(u,v)\in[0,1]^2}$ be the arc coordinate system of the curvilinear rectangle bounded by $(T_1,0)$ and $( T_2,0)$. Then, there are $\varepsilon>0,N>1$ such that, for all $(u,v,t) \in [0,1]^2 \times [0,\varepsilon)$, $\left| \gamma_{u,v}^{\prime}(t) \right| < N$.
\end{corollary}

\begin{proof}
By Proposition \ref{synch-derivative-along-t}, there are $\varepsilon > 0$ and $N_0 > 1$ such that $|\gamma_{u,v}^{\prime} (t)| < N_0$ for all $(u,v,t) \in [0,1] \times \{0,1\} \times [0,\varepsilon)$. So

$$|\gamma_{u,0}^{\prime} (t)| = \left| \left( x_u^{\prime}(t), (g_{t}(x_u(t)))^{\prime}(t),1 \right) \right| < N_0 \Rightarrow \left| x_u^{\prime}(t) \right|, \left| (g_{t}(x_u(t))^{\prime}(t) \right| < N_0,$$
$$|\gamma_{u,1}^{\prime} (t)| = \left| \left( x_u^{\prime}(t), (f_{t}(x_u(t))^{\prime}(t),1 \right) \right| < N_0 \Rightarrow \left| (f_{t}(x_u(t))^{\prime}(t) \right| < N_0.$$
Then,
$$(\sigma_{u,v}(t))^{\prime}(t) = v \cdot (f_t (x_u(t))^{\prime}(t) + (1- v) \cdot (g_t (x_u(t))^{\prime}(t) \Rightarrow$$
$$|(\sigma_{u,v}(t))^{\prime}(t)| \le v\cdot \left| (f_{t}(x_u(t))^{\prime}(t) \right| + (1-v) \cdot \left| (g_{t}(x_u(t))^{\prime}(t) \right| < v\cdot N_0 + (1-v) \cdot N_0 = N_0,$$
where
$$|\gamma_{u,v}^{\prime} (t)| = \left| (x_u^{\prime}(t),(\sigma_{u,v}(t))^{\prime}(t),1) \right| < \sqrt{N_0^2 + N_0^2 +1} = N.$$
\end{proof}

\section{Convex Decomposition}

The propositions in this section essentially use the techniques developed in \cite{parusinski1994lipschitz} and in \cite{Kurdyka1992OnAS}. We will present a proof with some specific adaptations to obtain synchronized triangles, similarly to what is done in \cite{birbrair1999}.

For each angle $\theta \in [0,2\pi]$, define the axis rotation map $r_{\theta} : \mathbb{R}^3 \to \mathbb{R}^3$ as
$$r_{\theta}(x,y,t) = (x\cdot \cos{\theta} + y\cdot \sin{\theta}, -x\cdot \sin{\theta} + y\cdot \cos{\theta},t).$$
Notice that $r_{\theta}$ is an ambient bi-Lipschitz map, since $r_{\theta}$ is an isometry.

\begin{proposition}\label{synch-decomp}
Let $a>0$ and let $(X,0) \subset (C_{a}^{3},0)$ be a pure, semialgebraic, closed 2-dimensional surface germ. Then, there exists $n \in \mathbb{Z}_{\ge 1}$ and $M>0$ such that
\begin{enumerate}
\item $(X,0)$ is the union of H\"older triangles $(X_1,0), \dots (X_n, 0)$, where, for each $i,j \in \{ 1, \dots , n\}\, ; \, i \ne j$, we have that $(X_i, 0) \cap (X_j, 0)=\{ 0 \}$ or $(X_i, 0) \cap (X_j, 0)$ is an arc. Furthermore, if $\Gamma := \{ (X_i, 0) \cap (X_j, 0) : i,j \in \{ 1, \dots , n\}; i \ne j \}$, then the elements of $\Gamma - \{ 0 \}$ are boundary arcs of $(X_1, 0), \dots (X_n , 0)$;
\item $(sing(X),0) \subset \Gamma$;
\item there are angles $\theta_1, \dots, \theta_n$ such that $(T_i,0)$ is a $C^1$, $M$-bounded synchronized triangle, where $T_i = r_{\theta_i}(X_i)$, for each $i=1,\dots,n$.
\end{enumerate}
Any decomposition $(X,0) = \bigcup_{i=1}^{n} (X_i,0)$ of the germ $(X,0)$ satisfying such conditions is defined as a \emph{synchronized decomposition of $(X ,0)$}.
\end{proposition}

\begin{proof}
Let $\{ e_x,e_y,e_z \}$ be the canonical orthonormal basis in $\mathbb{R}^3$, $R_{ij} = span \{e_i,e_j\}$ and $\pi_{ij} : \mathbb{R}^3 \to R_{ij}$ be the orthogonal projection. Consider the Gauss map
$$G: X\setminus sing(X) \to G(2,3) \; ; \; G(p)=T_p(X) \in G(2,3) \; ; \; \forall p \in X \setminus sing(X).$$
For each $\delta>0$ and $\{i,j\} \in \{x,y,z\}$, define
$$U_{ij}(\delta)=\left\{ m \in G(2,3) \mid \angle(m,R_{ij}) < \dfrac{\pi}{2} - \delta \right\}.$$

\begin{claim}\label{grassmanian}
There are $\varepsilon, \delta >0$ small enough such that
$$p \in (X\setminus sing(X)) \cap C_a^{3}[\varepsilon] \Rightarrow G(x) \in U_{xz}(\delta) \cup U_{yz}(\delta). $$
\end{claim}
\begin{proof}
For each $p\in X\setminus sing(X)$, let $(x_p,y_p,z_p)$ be the unit vector orthogonal to $T_p(X)$. Notice that $p \notin U_{xz}(\delta) \cup U_{yz}(\delta) \Leftrightarrow |x_p|,|y_p| \le \sin{\delta}$. Therefore, if we prove that there is $\rho = 1 - \sqrt{1-2\sin^2(\delta)}>0$ small enough such that, for some $\varepsilon>0$,
$$p \in (X\setminus sing(X)) \cap C_a^{3}[\varepsilon] \Rightarrow |z_p| < 1 - \rho,$$
then the claim will be proved, since $x_p^2 + y_p^2 = 1 - z_p^2 > 2\sin^2(\delta)$ would imply $|x_p| > \sin{\delta}$ or $|y_p| > \sin{\delta}$. Take $\rho = 1-\sqrt{\frac{2a^2}{1+2a^2}}$ and suppose that, for every $\varepsilon > 0$, there is $p\in (X\setminus sing(X )) \cap C_a^{3}[\varepsilon]$ such that $|z_p| \ge 1-\rho$. By the curve selection lemma, there is $\varepsilon_0 >0$ and a curve $\gamma: [0,\varepsilon_0] \to C_a^{3}$, $\gamma(t)=(x(t),y(t),t)$ such that $|z_{\gamma(t)}| \ge 1 - \rho$, for all $0 < t < \varepsilon_0$. As $(x^{\prime}(t),y^{\prime}(t),1) \in T_{\gamma(t)}(X)$, for every $0\le t < \varepsilon_0$ , we have
$$0=(x^{\prime}(t),y^{\prime}(t),1) \cdot (x_{\gamma(t)},y_{\gamma(t)},z_{\gamma(t)}) \Rightarrow |-z_{\gamma(t)}| = |x^{\prime}(t) \cdot x_{\gamma(t)} + y^{\prime}(t) \cdot y_{\gamma(t)}| \le$$
$$\le \sqrt{(x^{\prime}(t))^2+(y_{\prime}(t))^2}\sqrt{(x_{\gamma(t)})^2+ (y_{\gamma(t)})^2}=\sqrt{(x^{\prime}(t))^2+(y^{\prime}(t))^2}\sqrt{1- z_{\gamma(t)}^2}$$
$$\Rightarrow (x^{\prime}(t))^2+(y^{\prime}(t))^2 \ge \dfrac{z_{\gamma(t)}^2}{1- z_{\gamma(t)}^2} \ge 2a^2.$$
On the other hand, as $(\gamma,0) \subset (C_a^{3},0)$ is semialgebraic, we also have $(x^{\prime}(0),y^{\prime}(0),1) \in C_a^{3}$, where $x^{\prime}(0)+y^{\prime}(0) \le a^2$, a contradiction.
\end{proof}

Now, consider $X_x = \overline{G^{-1}(U_{xz}(\delta)}$ and $X_y =\overline{ G^{-1}(U_{yz}(\delta)) \setminus X_x}$. For each $t=x,y$, there is a pancake decomposition $\{X_t^{k}\}_{1 \le k\le K_t}$ of $X_t$ (see \cite{birbrair2000normal}) satisfying the following conditions:
\begin{enumerate}
\item All sets $X_t^k$ are semialgebraic satisfying $0 \in X_t^k$;
\item $(X_t^k,0)$ is a germ of a curvilinear triangle with vertex at the origin (in particular, a H\"older triangle);
\item For $k_1 \ne k_2$, $(X_t^{k_1},0) \cap (X_t^{k_2},0)$ is $\{0\}$ or an arc.
\item for each $1\le k\le K_t$, $\pi_{tz}|_{X_t^k\cap C_a^{3}(\varepsilon)}$ is a homeomorphism over its image.
\end{enumerate}

\begin{figure}[h]
\centering
\includegraphics[width=12cm]{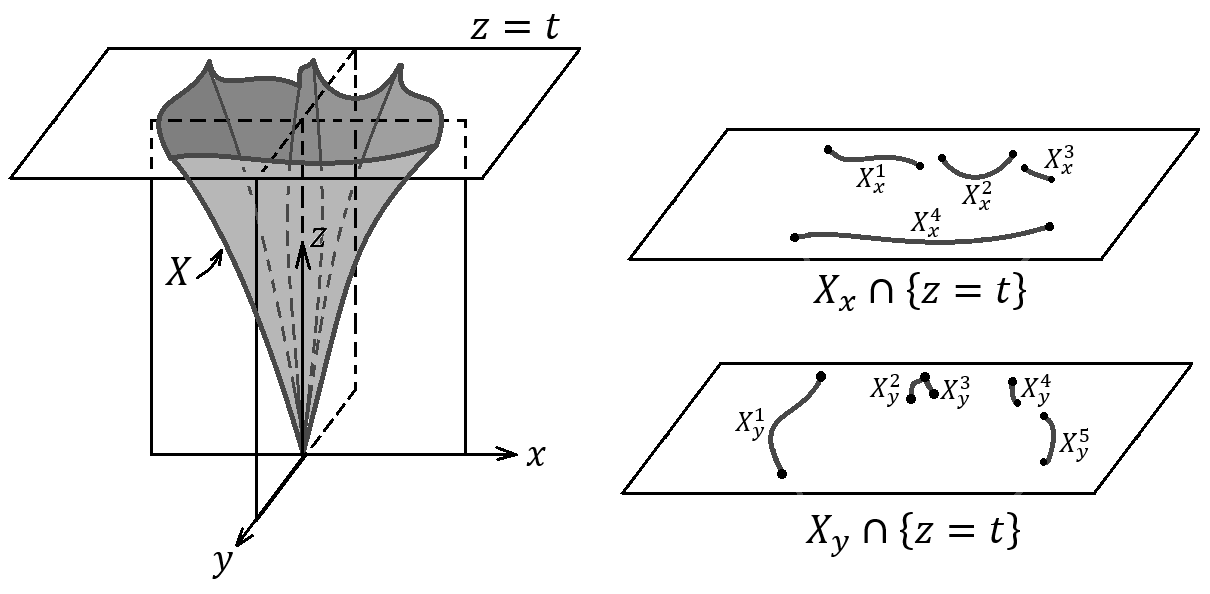}
\label{fig6}
\caption{Proof of Proposition \ref{synch-decomp}.}
\end{figure}

Finally, taking $T_x^{k} = r_{0}(X_x^{k})$ and $T_y^{k} = r_{\frac{\pi}{2}}(X_y^{k})$, and considering $\{T_1, \dots , T_n \}$ a re-indexation of the sets $T_x^{k},T_y^{k}$, we have that each $(T_i,0)$ is a germ of a $C^1$, $M$-bounded synchronized triangle.
\end{proof}

\begin{remark}\label{refinement-decomp}
Given a synchronized decomposition $P$ of $(X,0)$, notice that if we choose a H\"older triangle germ $(X_k,0)=(T(\gamma_1, \gamma_2),0)$ from this decomposition, an arc $\gamma \in V(X_k)$ and divide $(X_k,0)$ into two germs of H\"older triangles $(X_{k,0},0)=(T(\gamma_1,\gamma),0)$ and $(X_{k,1},0)=(T(\gamma,\gamma_2),0)$, then letting $ P'$ be the partition $(X,0) = \left( \bigcup_{i=1, i \ne k}^{n} (X_i,0) \right) \cup (X_{k,0}, 0) \cup (X_{k,1},0)$, $P'$ will also fulfill the conditions of Proposition \ref{synch-decomp}, thus being a synchronized decomposition of $(X,0)$. The partition $P'$ is defined as a \emph{direct refinement of $P$}. We also say that a synchronized decomposition $P'$ of a germ $(X,0)$ is a \emph{refinement of the synchronized decomposition $P$} of $(X,0)$ if there is a finite sequence of decompositions $P=P_0,P_1 ,\dots,P_n=P'$ such that $P_i$ is a direct refinement of $P_{i-1}$, for all $i=1,\dots,n$.
\end{remark}

\begin{proposition}[Convex Decomposition] \label{convex-decomp}
Let $a>0$ and let $(X,0) \subset (C_{a}^{3},0)$ be a pure, semialgebraic, closed 2-dimensional surface germ. Given any synchronized decomposition of $(X,0)$, we have that, for every $\delta>0$, there is a synchronized decomposition $(X,0) = \bigcup_{i=1}^{n} (X_i, 0)$ which is a refinement of the initial decomposition, such that
\begin{enumerate}
\item If $\{f_{t}\}$ is the generating functions family of $(T_i,0) = (r_{\theta_i}(X_i),0)$, then for each $t>0$ small enough $f_{t} : [x_0(t),x_1(t)] \to \mathbb{R}$ is a convex function;
\item For every $t>0$ small enough, we have:
$$\left| \dfrac { \partial f_t}{\partial x_{+}}(x_0(t)) - \dfrac{ \partial f_t}{\partial x_{-}}(x_1(t)) \right| < \delta.$$
\end{enumerate}
Any synchronized decomposition of the surface germ $(X,0)$ that fulfills these conditions is defined as a \emph{$\delta$-convex decomposition of $(X,0)$}.
\end{proposition}

\begin{proof}
We will prove the proposition in the case where $(X,0)$ is the germ of a $C^1$, $M$-bounded synchronized triangle, since the general case follows from this, when applying the result in each synchronized triangle germ of the decomposition of $(X,0)$ obtained from Proposition \ref{synch-decomp}. Consider $\{f_{t}\} \; ; \; f_{t}: [x_0(t),x_1(t)] \to \mathbb{R}$ the generating functions family of $(X,0)$, and let
$$A=\left\{ (x,f_t(x),t) \in X\cap C_a^{3}[\varepsilon] \mid f_t^{\prime \prime}(x) = 0 \right \}.$$
The germ $(A,0)\subseteq (X,0)$ is semialgebraic with dimension lesser than or equal to 2, and thus $(A,0)$ is either the empty set or a finite union of arcs and/or H\"older triangles. In the first case, we have that $f_t^{\prime \prime}(x)$ is always positive or negative, and then $f_t$ is a convex or concave function for all $0<t \ll 1$ (in this last one, $-f_t$ is convex, and so we can consider the germ $(r_{\pi}(X),0)$ instead of $(X, 0)$). In the second case, consider $(\gamma_1,0), \dots, (\gamma_n,0)$ the isolated arcs in $A$, together with the boundary arcs of the H\"older triangles in $A$ and the boundary arcs of $X$. We can assume, through a re-indexing of indexes, that, for $\varepsilon>0$ small enough and that, for each $0\le i\le n$, we have
$$\gamma_i(t)=(x_i(t),f_t(x_i(t)),t); 0<t < \varepsilon.$$

Therefore, $f_t|_{[x_i(t),x_{i+1}(t)]}$ is convex, for all $0<t \ll 1$, or concave, for all $0<t  \ll 1$. Thus, if we divide $(X,0)$ into $n$ germs of H\"older triangles, with $(X_i,0)$ being the germ delimited by $(\gamma_{i-1},0),(\gamma_i ,0)$ ($i=1,\dots,n$), then we can choose angles $\theta_1, \dots, \theta_n$ such that each $(T_i,0)$ has convex functions as its generating functions family ( $\theta_i = 0$, if $f_t|_{[x_i(t),x_{i+1}(t)]}$ is convex and $\theta_i = \pi$, if $f_t|_{[x_i (t),x_{i+1}(t)]}$ is concave).

\begin{figure}[h]
\centering
\includegraphics[width=13cm]{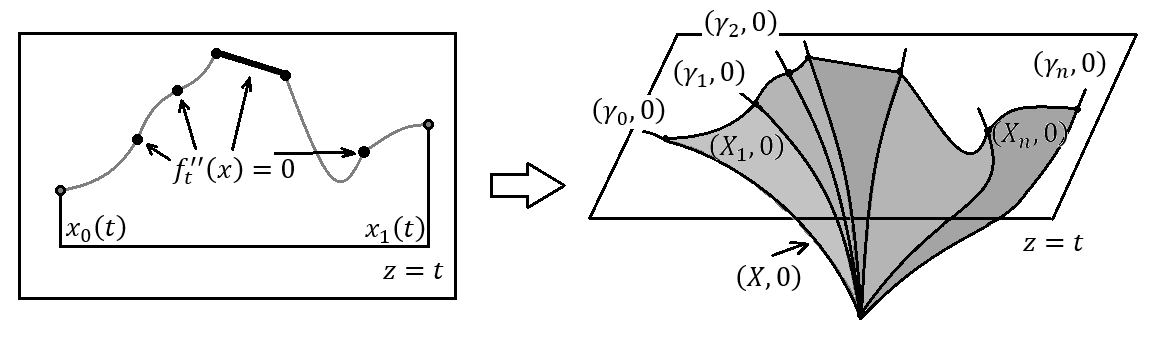}
\label{fig7}
\caption{Decomposition by arcs determined by $A$.}
\end{figure}

We have shown that every synchronized partition has a refinement that fulfills the first condition of the proposition, and thus it suffices to prove the proposition assuming that the generating functions family of $(X,0)$ is a family of convex functions, because in this way we can consider a refinement of the refinement obtained here, in which the second condition is also fulfilled. Assume this is the case. and define, for each $0<t \ll 1$,
$$\rho_t: [x_0(t),x_1(t)] \to \left( -\frac{\pi}{2},\frac{\pi}{2} \right); \rho_t(x)= \arctan{f_t^{\prime}(x)},$$
where, for convenience, we denote $f_t^{\prime}(x_0(t))= \dfrac { \partial f_t}{\partial x_{+}}(x_0(t))$ and $f_t^{\prime} (x_1(t))= \dfrac { \partial f_t}{\partial x_{-}}(x_1(t))$. Notice that, by the convexity of $f_t$, we have, for all $0<t \ll 1$ and for each $x,y \in [x_0(t),x_1(t) ]$, with $x<y$, that $\rho_t(x) \le \rho_t(y)$. Hence, if $a= \displaystyle \lim_{t \to 0^{+}} \rho_t(x_0(t))$ and $b= \displaystyle \lim_{t \to 0^{+}} \rho_t( x_1(t))$, then $a\le b$ and, for every $\kappa > 0$, $0<t  \ll 1$ and every $x_0(t) \le x \le x_1(t)$,
$$a - \kappa <\rho_t(x_0(t) ) \le \rho_t(x) \le \rho_t(x_1(t)) < b+ \kappa.$$
We have now 2 cases to consider:

\textit{Case 1:} if $2\tan(b-a) < \delta$, then take $\theta = a-\kappa$ and let $(T,0) = (r_\theta (X),0)$ . Notice that $(T,0)$ is a $C^1$, convex and $\tan{(b-a + 2\kappa)}$-bounded synchronized triangle. Taking $\kappa$ small enough such that $2\tan{(b-a + 2\kappa)} < \delta$, the result follows.

\begin{figure}[h]
\centering
\includegraphics[width=11cm]{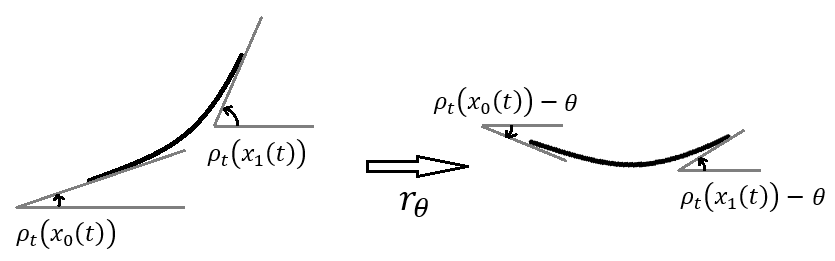}
\label{fig8}
\caption{Angle adjustment for a $\delta$-convex decomposition.}
\end{figure}

\textit{Case 2:} if $2\tan(b-a) \ge \delta$, then consider $n \in \mathbb{Z}_{\ge 1}$ such that $2\tan(\frac{b-a}{n}) < \delta $, and for every $0<t \ll 1$, $i=0,1, \dots, n$ let $p_i(t) \in [x_0(t),x_1(t)]$ be the only real number satisfying
$$\rho_t(p_i(t)) = \frac{1}{n}\left( (n-i)\cdot \rho_t(x_0(t)) + i \cdot \rho_t(x_1(t)) \right).$$

The uniqueness is guaranteed by the convexity of functions $f_t$ and by the inequality $\rho_t(x_0(t)) < \rho_t(x_1(t))$. If $\gamma_i(t) = (p_i(t),f_t(p_i(t)),t)$, then when we divide $(X,0)$ into $n$ H\"older triangles $(X_1,0 ), \dots , (X_n,0)$, with $(X_i,0)$ being the germ delimited by $(\gamma_{i-1},0),(\gamma_i,0)$ ($i=1 , \dots, n$), and then we apply the same argument of Case 1 to each $(X_i,0)$.

\begin{figure}[h]
\centering
\includegraphics[width=14cm]{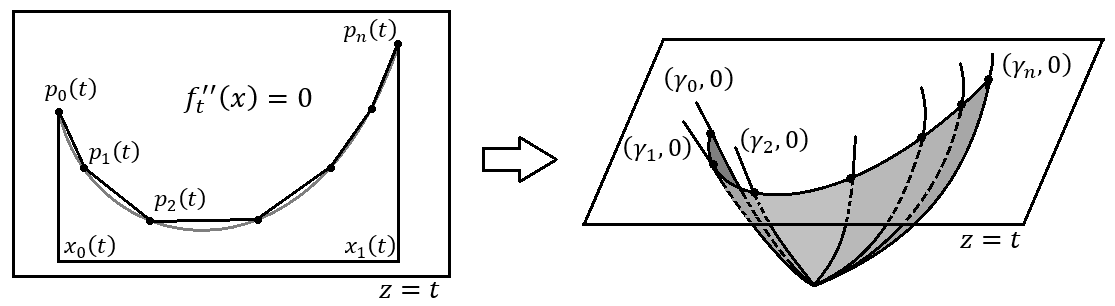}
\label{fig9}
\caption{$\delta$-convex decomposition on a synchronized triangle.}
\end{figure}
\end{proof}

\section{Ambient Isotopy in Curvilinear Rectangles}

This section is devoted to define the concept of ambient bi-Lipschitz isotopy and to prove the ambient isotopy theorem in curvilinear rectangles. This structural core result is, in a certain way, an extension of the $\mathcal{K}$-bi-Lipschitz equivalence between function germs studied in \cite{birbrair2007K}.

\begin{definition}\normalfont \label{amb-isotopy}
Let $X, X_0, X_1 \subseteq \mathbb{R}^n$ be sets such that $X_1, X_2 \subseteq X$. We say that $X_1, X_2$ are \textit{ambient bi-Lipschitz isotopic in $X$} if there is a continuous map $\varphi : X\times [0,1] \to X$ such that, if we denote $\varphi_ {\tau}(p) = \varphi(p,\tau)$, then
\begin{enumerate}
\item $\varphi_{\tau}: X \to X$ is an outer bi-Lipschitz map, for all $0 \le \tau \le 1$;
\item $\varphi_0 = id_{X}$;
\item $\varphi_1(X_1)=X_2$.
\end{enumerate}

The map $\varphi$ is an \textit{ambient bi-Lipschitz isotopy in $X$, taking $X_1$ into $X_2$}. We also say that the isotopy $\varphi$ is \textit{invariant on the boundary} if $\varphi_{\tau}|_{\partial X}=id_{\partial X}$, for all $0 \le \tau \le 1$.
\end{definition}

Analogously, we can define ambient bi-Lipschitz isotopy for germs as follows.

\begin{definition}\normalfont \label{amb-isotopy-equiv}
Let $(X,0), (X_0,0), (X_1,0) \subseteq (\mathbb{R}^n,0)$ be set germs such that $(X_1,0), (X_2,0) \subseteq (X,0)$. We say that $(X_1,0), (X_2,0)$ are \textit{ambient Bi-Lipschitz isotopic in $(X,0)$} if there is a continuous map $\varphi : (X,0)\times [ 0,1] \to (X,0)$ such that, if we denote $\varphi_{\tau}(p) = \varphi(p,\tau)$, then:
\begin{enumerate}
\item $\varphi_{\tau}: (X,0) \to (X,0)$ is an bi-Lipschitz map, for all $0 \le \tau \le 1$;
\item $\varphi_0 = id_{(X,0)}$;
\item $\varphi_1((X_1,0))=(X_2,0)$.
\end{enumerate}

The map $\varphi$ is an \textit{ambient bi-Lipschitz isotopy in $(X,0)$, taking $(X_1,0)$ into $(X_2,0)$}. We also say that the isotopy $\varphi$ is \textit{invariant on the boundary} if $\varphi_{\tau}|_{(\partial X,0)}=id_{(\partial X,0)}$, for all $0 \le \tau \le 1$.
\end{definition}

\begin{theorem}[Ambient bi-Lipschitz isotopy in curvilinear rectangles]\label{amb-isotopy-criterion}
Let $a>0$ and let $(T_1,0),(T_2,0),(W_1,0),(W_2,0) \subset (C_a^{3},0)$ be synchronized triangles, each two of them aligned on boundary arcs. Suppose that for every $t>0$ small enough, there is $M>1$ such that the following holds:
\begin{itemize}
\item $(T_1,0),(T_2,0),(W_1,0),(W_2,0)$ are $M$-bounded and $\{f_t \}, \{g_t \}, \{a_t \}, \{b_t \}$ are their respective generating function families;
\item $(T_1,0)$ has $(\gamma_0,0),(\gamma_1,0)$ as boundary arcs, where
$$\gamma_0 = \gamma_0(t) = (x_0(t), y_0(t), t) \; ; \; \gamma_1 = \gamma_1 (t) = (x_1(t),y_1(t),t)$$
and $x_0(t) < x_1(t)$;
\item for all $x \in (x_0(t),x_1(t))$, we have
$$g_t(x) < a_t(x),b_t(x) < f_t(x),$$
$$\dfrac{1}{M} \le \dfrac{a_t(x)-g_t(x)}{f_t(x) - g_t(x)}, \dfrac{b_t(x)-g_t(x)} {f_t(x) - g_t(x)} \le 1-\dfrac{1}{M}.$$
\end{itemize}
If $(R,0)$ is the curvilinear rectangle bounded by $(T_1,0),(T_2,0)$, then there is a continuous map $\varphi : (R,0)\times [0,1] \to (R,0)$ such that, if we denote $\varphi_{\tau}(p) = \varphi(p,\tau)$, we have
\begin{enumerate}
\item $\varphi_{\tau}: (R,0) \to (R,0)$ is an outer bi-Lipschitz map, for all $0 \le \tau \le 1$;
\item $\varphi_0 = id_{(R,0)}$ and $\varphi_{\tau}|_{(T_1,0)\cup(T_2,0)}=id_{(T_1 \cup T_2,0) )}$, for every $0 \le \tau \le 1$;
\item For all small $t>0$ and $x\in [x_0(t),x_1(t)]$,
$$\varphi_1(x,f_t(x),t) = (x,f_t(x),t),$$
$$\varphi_1(x,g_t(x),t) = (x,g_t(x),t),$$
$$\varphi_1(x,a_t(x),t) = (x,b_t(x),t).$$
\end{enumerate}

In other words, $\varphi$ is an ambient bi-Lipschitz isotopy in $(R,0)$, taking $(W_1,0)$ into $(W_2,0)$. Furthermore, if $(T_1,0)$ and $(T_2,0)$ have the same boundary arcs, $\varphi$ is also invariant on the boundary.
\end{theorem}

\begin{proof}
Let us first consider the case where $(T_1,0),(T_2,0),(W_1,0),(W_2,0)$ are $C^1$. By Proposition \ref{synch-derivative-along-t} we can take $N>1$ large enough such that, for all $t> 0$ small enough and $x_0(t) < x < x_1(t)$,
$$\left| \frac{\partial x_0}{\partial t}(t) \right|, \left| \frac{\partial x_1}{\partial t}(t) \right|, \left| \dfrac{\partial}{\partial t}(f_t)(x) \right|, \left| \dfrac{\partial}{\partial t}(g_t)(x) \right|, \left| \dfrac{\partial}{\partial t}(a_t)(x) \right|, \left| \dfrac{\partial}{\partial t}(b_t)(x) \right| < N.$$
Let $\{\gamma_{u,v}\}_{(u,v) \in [0,1]^2}$ be the arc coordinate system of $(R,0)$. For all $ t>0$ small and all $x\in (x_0(t),x_1(t))$, $x=x_u(t)$, with $u\in (0,1)$, we can write
$$(x,g_t(x),t)=(x_u(t),\sigma_{u,0}(t),t)=\gamma_{u,0}(t),$$
$$(x,f_t(x),t)=(x_u(t),\sigma_{u,1}(t),t)=\gamma_{u,1}(t),$$
$$(x,a_t(x),t)=(x_u(t),\sigma_{u,\alpha_t(u)}(t),t)=\gamma_{u,\alpha_t(u)}( t),$$
$$(x,b_t(x),t)=(x_u(t),\sigma_{u,\beta_t(u)}(t),t)=\gamma_{u,\beta_t(u)}( t),$$
where, for each $t>0$, we consider the functions $\alpha_t ,\beta_t : [0,1] \to [0,1]$ as
$$\alpha_t(u)=\dfrac{a_t(x_u(t))-g_t(x_u(t))}{f_t(x_u(t))-g_t(x_u(t))}; \; \beta_t(u)=\dfrac{b_t(x_u(t))-g_t(x_u(t))}{f_t(x_u(t))-g_t(x_u(t))}, \forall u \in[0,1];$$
$$\alpha_t(0) = \displaystyle \lim_{u \to 0^{+}} (\alpha_t(u)); \alpha_t(1) = \displaystyle \lim_{u \to 1^{-}} (\alpha_t(u)); \beta_t(0) = \displaystyle \lim_{u \to 0^{+}} (\beta_t(u)), \beta_t(1) = \displaystyle \lim_{u \to 1^{-}} (\beta_t(u)).$$

Notice that, from hypothesis, we have $\frac{1}{M} \le \alpha_t(u),\beta_t(u) \le 1-\frac{1}{M}$. Furthermore,
$$ \sigma_{u,\alpha_t(u)}(t) = g_t(x_u(t)) + \alpha_t(u) \cdot (f_t(x_u(t)) - g_t (x_u(t) ))=$$
$$ = g_t(x_u(t)) + (a_t(x_u(t))-g_t(x_u(t))) = a_t(x_u(t)),$$
$$ \sigma_{u,\beta_t(u)}(t) = g_t(x_u(t)) + \beta_t(u) \cdot (f_t(x_u(t)) - g_t (x_u(t) )) = $$
$$ =g_t(x_u(t)) + (b_t(x_u(t))-g_t(x_u(t))) = b_t(x_u(t)),$$
which shows that $\alpha_t , \beta_t$ are well defined for each $u \in (0,1)$. Furthermore, as $g_t,f_t,a_t,b_t$ are $C^1$, $\alpha_t , \beta_t$ are also $C^1$ and this also implies that $\alpha_t,\beta_t$ are also defined in $\{0,1\}$. Now, for every small $t>0$, $(u,\tau) \in [0,1]^2$, define $\beta_t^{\tau}(u)=(1-\tau) \cdot \alpha_t(u) + \tau \cdot \beta_t(u)$ and $\psi_{u,t}^{\tau}:[0,1] \to [0,1]$ as
$$
\psi_{u,t}^{\tau}(v) =
   \begin{cases}
       \left( \dfrac{\beta_t^{\tau}(u)}{\alpha_t(u)} \right)v, & 0\le v \le \alpha_t(u) \\
       1- \left( \dfrac{1-\beta_t^{\tau}(u)}{1-\alpha_t(u)} \right)(1-v), & \alpha_t(u)\le v \le 1
   \end{cases}.
$$

Notice that $\psi_{u,t}^{\tau}(v)$ is well defined, as $\psi_{u,t}^{\tau}(\alpha_t(u))=\beta_t^{\tau}(u)$ in both rules, $\psi_{u,t}^{\tau}(0)=0$, $\psi_{u,t}^{\tau}(1)=1$ and $\psi_{u,t}^{\tau}$ is piecewise linear, hence continuous. Notice also that, from the inequalities $\frac{1}{M} \le \alpha_t(u)\le 1-\frac{1}{M}$, the denominators involved in the expressions are all positive. We also have $\psi_{u,t}^{0} = id_{[0,1]}$ and $(\psi_{u,t}^{\tau})^{-1} : [0, 1] \to [0,1]$ is given by

$$
(\psi_{u,t}^{\tau})^{-1}(v) =
   \begin{cases}
       \left( \dfrac{\alpha_t(u)}{\beta_t^{\tau}(u)} \right)v, & 0\le v \le \beta_t^{\tau}(u) \\
       1- \left( \dfrac{1-\alpha_t(u)}{1-\beta_t^{\tau}(u)} \right)(1-v), & \beta_t^{\tau}(u) \le v \le 1
   \end{cases}.
$$

Notice also that $(\psi_{u,t}^{\tau})^{-1}(v)$ is well defined, as $(\psi_{u,t}^{\tau})^{- 1}(\beta_t^{\tau}(u))=\alpha_t(u)$ in both rules, $(\psi_{u,t}^{\tau})^{-1}(0)= 0$, $(\psi_{u,t}^{\tau})^{-1}(1)=1$ and $(\psi_{u,t}^{\tau})^{-1} $ is piecewise linear, hence continuous. We also have $\frac{1}{M} \le \alpha_t(u), \beta_t(u) \le 1-\frac{1}{M} \Rightarrow \frac{1}{M} \le \beta_t^{\tau}(u) \le 1-\frac{1}{M}$, where the denominators involved in the expressions are all positive.

Now, consider the following sets in $R$:
$$R_1 = \{ (x_u(t),\sigma_{u,v}(t),t) | \; t>0; \; 0<u<1; \; 0 < v < \alpha_t(u) \} \subset R,$$
$$R_1^{\prime} = \{ (x_u(t),\sigma_{u,v}(t),t) | \; t>0; \; 0<u<1; \; 0 < v < \beta_t(u) \} \subset R,$$
$$R_2 = \{ (x_u(t),\sigma_{u,v}(t),t) | \; t>0; \; 0<u<1; \; \alpha_t(u) < v < 1 \} \subset R,$$
$$R_2^{\prime} = \{ (x_u(t),\sigma_{u,v}(t),t) | \; t>0; \; 0<u<1; \; \beta_t(u) < v < 1 \} \subset R.$$

Notice that $(R_1,0),(R_2,0),(R_1^{\prime},0),(R_2^{\prime},0)$ are open subset germs of $(R,0)$ (with the induced topology of $\mathbb{R}^3$) and $(\overline{R_1} \cup \overline{R_2},0), (\overline{R_1^{\prime}} \cup \overline{ R_2^{\prime}},0)=(R,0)$. Finally, define the family of maps $\varphi_{\tau}: (R,0) \to (R,0)$, $0 \le \tau \le 1$, as
$\varphi_{\tau}(\gamma_{u,v}(t)) = \gamma_{u, \psi_{u,t}^{\tau}(v)}(t) \; ; \; (u,v) \in [0,1]^2.$

\begin{figure}[h]
\centering
\includegraphics[width=15cm]{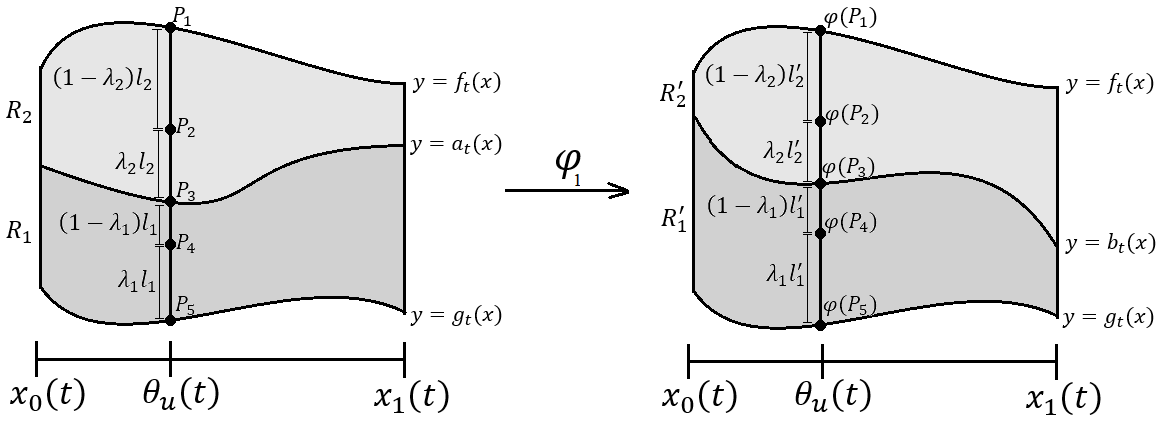}
\label{fig5b}
\caption{Geometric interpretation of  $\varphi_1$ in Theorem {amb-isotopy-criterion}.}
\end{figure}

A direct calculation shows that
$$\varphi_0 = id_{(R,0)}\; , \; \varphi_{\tau}(\gamma_{u,0}(t))=\gamma_{u,0}(t) \; , \; \varphi_{\tau}(\gamma_{u,1}(t))=\gamma_{u,1}(t),$$
$$\varphi_1((R_1,0)) = (R_1^{\prime},0) \; ; \; \varphi_1((R_2,0))=(R_2^{\prime},0).$$
The map $\varphi_\tau$ has an inverse $(\varphi_\tau)^{-1}: (R,0) \to (R,0)$ given by
$ (\varphi_{\tau})^{-1}(\gamma_{u,v}(t)) = \gamma_{u, (\psi_{u,t}^{\tau})^{- 1}(v)}(t)$. As $\psi_{u,t}^{\tau}$, $(\psi_{u,t}^{\tau})^{-1}$ are continuous, $\varphi_\tau$ is a homeomorphism. If we prove that $\varphi_{\tau}|_{(R_1,0)}$ and $\varphi_{\tau}|_{(R_2,0)}$ fulfies the conditions of Proposition \ref{jacobiano-lip} , then the proof is complete, since if such maps are inner bi-Lipschitz, then as $\varphi_{\tau}|_{\overline{(R_1,0)}}$ and $\varphi_{\tau }|_{\overline{(R_2,0)}}$ are continuous, by Proposition \ref{bi-lip-extende-fecho}, $\varphi_{\tau}|_{\overline{(R_1,0)}}$ and $\varphi_{\tau}|_{\overline{(R_2,0)}}$ are inner bi-Lipschitz as well. Moreover,
$$\emptyset \ne (\overline{R_1} \cap \overline{R_2},0) = \{ (x_u(t),\sigma_{u,\alpha_t(u)}(t),t) \; ; \; t>0 \; ; \;0 \le u \le 1 \} \subset (R,0),$$
$$(\overline{R_1} \cup \overline{R_2},0) = (R,0)$$
and from the fact that $\overline {R_1}, \overline {R_2}$ are connected semialgebraic sets, we obtain that, by Proposition \ref{colagem-bi-lip-inner}, $\varphi_{\tau}$ is an inner bi-Lipschitz map , which implies that it is also an outer bi-Lipschitz map, because $(R,0)$ is LNE, by Proposition \ref{curv-rect-is-LNE}.

We will study the jacobian $J_{\varphi_{\tau}} (p)$ for all $p\in (R_1,0)$ (the case $p\in (R_2,0)$ is analogous). Notice that if $p=\gamma_{u,v}(t)=(x,y,t)$, then $\varphi_{\tau}(p)=(x,h,t)$, where
$$x=x_u(t)=u\cdot x_0(t) + (1-u)\cdot x_1(t),$$
$$y=v\cdot f_t(x) + (1-v)\cdot g_t(x),$$
$$h=f_t(x)+ \left( \dfrac{\beta_t^{\tau}(u)}{\alpha_t(u)} \right)(y-f_t(x)).$$
Then
$$h=\sigma_{u,\left( \frac{\beta_t^{\tau}(u)}{\alpha_t(u)} \right)v}\left( t \right)= \left( \dfrac{\beta_t^{\tau}(u)}{\alpha_t(u)}\cdot v \right)g_t(x_u(t))+\left(1-\dfrac{\beta_t^{\tau} (u)}{\alpha_t(u)}\cdot v \right)f_t(x_u(t))=$$
$$=f_t(x_u(t))+\left( \dfrac{\beta_t^{\tau}(u)}{\alpha_t(u)} \right)\left( g_t(x_u^{\tau }(t))-f_t(x_u^{\tau}(t)) \right)\cdot v =$$
$$= f_t(x)+\left( \dfrac{\beta_t^{\tau}(u)}{\alpha_t(u)} \right)(g_t(x)-f_t(x))\left( \dfrac{y-f_t(x)}{g_t(x)-f_t(x)} \right) = f_t(x)+ \left( \dfrac{\beta_t^{\tau}(u)}{\alpha_t( u)} \right)(y-f_t(x)).$$

Calculating the jacobian matrix, we get:
$$
J_{\varphi_{\tau}} (x,y,t)= \left(
\begin{array}{ccc}
\frac{\partial x}{\partial x} & \frac{\partial x}{\partial y} & \frac{\partial x}{\partial t} \\
\frac{\partial h}{\partial x} & \frac{\partial h}{\partial y} & \frac{\partial h}{\partial t} \\
\frac{\partial t}{\partial x} & \frac{\partial t}{\partial y} & \frac{\partial t}{\partial t} \\
\end{array}
\right)(x,y,t) = \left(
\begin{array}{ccc}
1 & 0 & \frac{\partial x}{\partial t} \\
\frac{\partial h}{\partial x} & \frac{\partial h}{\partial y} & \frac{\partial h}{\partial t} \\
0 & 0 & 1 \\
\end{array}
\right)(x,y,t).
$$

We prove that there exists $T>0$ such that $\| J_{\varphi_{\tau}} (x,y,t) \| < T$, since $\| J_{\varphi_{\tau}} (x,y,t) \| \ge 1$. In order to do so, let us study each of the partial derivatives separately. Calculating $\frac{\partial x}{\partial t}$ and remembering that $\left| \frac{\partial x_0}{\partial t}(t) \right|, \left| \frac{\partial x_1}{\partial t}(t) \right| < N$, we have
$$\dfrac{\partial x}{\partial t}=\dfrac{\partial}{\partial t}(u\cdot x_0(t)+(1-u)\cdot x_1(t))=u\cdot \dfrac{\partial x_0}{\partial t}(t)+ (1-u)\dfrac{\partial x_1}{\partial t}(t) \Rightarrow$$
$$\Rightarrow \left| \dfrac{\partial x}{\partial t} \right| \le u\left| \dfrac{\partial x_0}{\partial t}(t) \right| + (1-u) \left| \dfrac{\partial x_1}{\partial t}(t) \right| \le u\cdot N+(1-u)\cdot N=N.$$
Expanding $\frac{\partial h}{\partial x}$ and denoting $u^{\prime}(x)=\frac{\partial u}{ \partial x}(x)$ the derivative at $x$ , we get
$$\dfrac{\partial h}{\partial x}=\dfrac{\partial }{\partial x}\left( f_t(x) + \left( \dfrac{\beta_t^{\tau}(u) }{\alpha_t(u)} \right)(y-f_t(x)) \right) =$$
$$= \dfrac{\partial }{\partial x}\left( (1-\tau) \cdot y + \tau \left( f_t(x) + \left( \dfrac{b_t(x)-g_t( x)}{a_t(x)-g_t(x)} \right)(y-f_t(x)) \right) \right) =$$
$$=\tau \cdot \big[ f_t^{\prime}(x)-\left( \dfrac{b_t(x)-g_t(x)}{a_t(x)-g_t(x)} \right) f_t^{\prime}(x)+$$
$$+\left( (b_t^{\prime}(x) -g_t^{\prime}(x))\left( \dfrac{y-f_t(x)}{a_t(x)-g_t(x) } \right) - \left( \dfrac{b_t(x)-g_t(x)}{a_t(x)-g_t(x)} \right)\left( \dfrac{y-f_t(x)}{a_t (x)-g_t(x)} \right)(a_t^{\prime}(x)-g_t^{\prime}(x)) \right) \big].$$
Since $y \ge g_t(x)$,
$$\left| \dfrac{y-f_t(x)}{a_t(x)-g_t(x)} \right| \le \left| \dfrac{f_t(x)-g_t(x)}{a_t(x)-g_t(x)} \right| < M,$$
and since $b_t(x) \le f_t(x)$, $\frac{1}{M}(f_t(x)-g_t(x)) \le a_t(x)-g_t(x)$, we have:
$$\left| \dfrac{b_t(x)-g_t(x)}{a_t(x)-g_t(x)} \right| \le \left| \dfrac{f_t(x)-g_t(x)}{a_t(x)-g_t(x)} \right| \le \left| \dfrac{f_t(x)-g_t(x)}{\frac{1}{M}(f_t(x)-g_t(x))} \right| = M.$$
Finally, from $|f_t^{\prime}(x)|,|g_t^{\prime}(x)|,|a_t^{\prime}(x)|,|b_t^{\prime}(x) | < M$, by the triangular inequality we obtain
$$\dfrac{\partial h}{\partial x} <1 \cdot [ M + M\cdot M + ((M+M)\cdot M + M \cdot M \cdot (M+M))] = T_2.$$
Calculating $\frac{\partial h}{\partial y}$, we have
$$\frac{\partial h}{\partial y}=\frac{\partial}{\partial y}\left( f_t(x)+ \left( \dfrac{\beta_t^{\tau}(u) }{\alpha_t(u)} \right)(y-f_t(x)) \right)=$$
$$=(1-\tau)+\tau \cdot \left( \dfrac{\beta_t(u)}{\alpha_t(u)} \right)= (1-\tau)+\tau \cdot \left ( \dfrac{b_t(x)-g_t(x)}{a_t(x)-g_t(x)} \right).$$
Hence, $\frac{1}{M} < \left| \frac{\partial h}{\partial y} \right| <M$, because
$$\dfrac{|b_t(x)-g_t(x)|}{|a_t(x)-g_t(x)|} > \dfrac{|b_t(x)-g_t(x)|}{|f_t( x)-g_t(x)|} \ge \dfrac{\frac{1}{M}|f_t(x)-g_t(x)|}{|f_t(x)-g_t(x)|}=\frac {1}{M}$$
$$\dfrac{|b_t(x)-g_t(x)|}{|a_t(x)-g_t(x)|} < \dfrac{|f_t(x)-g_t(x)|}{|a_t( x)-g_t(x)|} \le \dfrac{|f_t(x)-g_t(x)|}{\frac{1}{M}|f_t(x)-g_t(x)|}=M.$$
Expanding $\frac{\partial h}{\partial t}$ and denoting $\dot u (x) = \frac{\partial u}{\partial t}(x)$ the derivative in $t$, we have
$$\dfrac{\partial h}{\partial t}=\dfrac{\partial }{\partial t}\left( f_t(x) + \left( \dfrac{\beta_t^{\tau}(u)}{\alpha_t(u)} \right)(y-f_t(x)) \right) = $$
$$ = \dfrac{\partial }{\partial t}\left((1-\tau) \cdot y + \tau \cdot \left( f_t(x) + \left( \dfrac{b_t(x)-g_t(x)}{a_t(x)-g_t(x)} \right)(y-f_t(x)) \right) \right)=$$
$$=\tau \cdot \big[ \dot f_t(x)-\left( \dfrac{b_t(x)-g_t(x)}{a_t(x)-g_t(x)} \right) \dot f_t(x)+$$
$$+\left( ( \dot b_t(x) -  \dot g_t(x))\left( \dfrac{y-f_t(x)}{a_t(x)-g_t(x)} \right) - \left( \dfrac{b_t(x)-g_t(x)}{a_t(x)-g_t(x)} \right)\left( \dfrac{y-f_t(x)}{a_t(x)-g_t(x)} \right)(\dot a_t(x)-\dot g_t(x)) \right) \big].$$

We have already seen in the development of $\frac{\partial h}{\partial x}$ that
$$\left| \frac{y-f_t(x)}{a_t(x)-g_t(x)} \right|, \left| \frac{b_t(x)-g_t(x)}{a_t(x)-g_t(x)} \right|< M,$$
and since $|\dot f_t(x)|,|\dot g_t(x)|,|\dot a_t(x)|,|\dot b_t(x)| < N$, it follows, by the triangular inequality, that
$$\dfrac{\partial h}{\partial t} < 1\cdot [ N + M\cdot N + \left( \left( N+N \right) \cdot M + M \cdot M \cdot \left ( N+N \right) \right) ] = T_3 .$$
Therefore, for every $p\in R_1$,
$$1 \le \|J_{\varphi_{\tau}} (p)\| < 1 + 0 + N + T_2 + M + T_3 + 0 + 0 + 1 = T$$
$$\frac{1}{M} < \left| \frac{\partial h}{\partial y} \right| = |\det{(J_{\varphi_{\tau}}(p))}| < M,$$
finishing the proof of this case.

For the general case, where $(T_1,0),(T_2,0),(W_1,0),(W_2,0)$ are not necessarily $C^1$, by Proposition \ref{decomp-synch-differential}, if we consider the decomposition
$$sing(T_1\cup T_2 \cup W_1 \cup W_2,0) = \bigcup_{i=1}^{n} (\gamma_i,0),$$
with $\gamma_i(t)=(r_i(t),s_i(t),t)$, for $i=1,\dots,n$, and assuming $r_1(t) < \dots < r_n(t) $, for every small $t>0$, then, for $i=1,\dots, n-1$, let
$$D_i= \{ (x,y,t) \in C_a^{3} \mid t>0 \; , \; r_i(t) \le x \le r_{i+1}(t) \},$$
$$(T_1^{(i)},0) = (T_1,0)\cap (D_i,0) \, , \, (T_2^{(i)},0) = (T_2,0)\cap (D_i,0);$$
$$(W_1^{(i)},0) = (W_1,0)\cap (D_i,0) \, , \, (W_2^{(i)},0) = (W_2,0)\cap (D_i,0).$$
We have that each $(T_1^{(i)},0),(T_2^{(i)},0),(W_1^{(i)},0),(W_2^{(i)},0 )$ is a $C^1$, $M$-bounded synchronized triangle germ. Therefore, if $(R_i,0) = (R \cap D_i,0)$ is the curvilinear rectangle bounded by $(T_1^{(i)},0)$ and $(T_2^{(i)},0 )$, by the previous case there is an inner bi-Lipschitz map $(\varphi_i)_{\tau} : (R_i,0) \to (R_i,0)$ satisfying the theorem conditions, for every $0 \le \tau \le 1$. Notice also that, by the way $(\varphi_i)_{\tau}$ was defined for each $i$, we have
$$(\varphi_i)_{\tau} (r_{i+1}(t), y, t) = (\varphi_{i+1})_{\tau} (r_{i+1}(t ), y, t) \; ; \; \forall \; (r_{i+1}(t), y, t) \in (R_i \cap R_{i+1},0).$$
Since $R_i$ is path connected, then applying Proposition \ref{colagem-bi-lip-inner} consecutively to the germs $(R_1,0) \cup (R_2,0)$, $(R_1,0) \cup (R_2,0) \cup (R_3,0)$ and so on, we can define $\varphi_{\tau} : (R,0) \to (R,0)$ satisfying all the desired conditions.
\end{proof}

\begin{remark} \label{amb-isotopy-extension}
If $(T_1,0)$ and $(T_2,0)$ have the same boundary arcs, $\varphi$ is an ambient bi-Lipschitz isotopy taking $(W_1,0)$ and $(W_2,0)$, invariant on the boundary $(T_1\cup T_2,0)$. By Proposition \ref{extensão-bi-lip-invariante-bola}, $\varphi$ can be extended to an ambient bi-Lipschitz isotopy $\Phi: (C_a^3,0)\times[0,1] \to (C_a^3,0)$ taking $(W_1,0)$ into $(W_2,0)$, invariant on the boundary of $(C_a^3,0)$. In particular, the map $\Phi_1$ shows that $W_1, W_2$ are ambient bi-Lipschitz equivalents in $(C_a^3,0)$.
\end{remark}

\section{Linear and Kneadable Triangles}

\subsection{Linear Triangles Determined by Arcs}

In this section, we will introduce the definition of linear triangles determined by arcs and its related concepts. As an application of Theorem \ref{amb-isotopy-criterion}, we will also prove that such triangles are ambient bi-Lipschitz equivalent to the normal embedding of some H\"older triangle.

\begin{definition}\normalfont \label{linear-triangle}
Let $n \in \mathbb{N}_{\ge 1}$, $a>0$ and $\gamma_1,\gamma_2 \subset C_a^{n+1}$ be two distinct arcs, with $\gamma_i(t)=(x_i( t),t) \in \R^n \times \R$ ($i=1,2$), for every $t>0$ small enough. We define the \textit{linear triangle determined by $\gamma_1,\gamma_2$} as the germ at the origin of the set
$$\overline{\gamma_1 \gamma_2}=\{ (\lambda x_1(t)+(1-\lambda)x_2(t), t) \in \R^n \times \R \mid t>0 \; ; \; 0 \le \lambda \le 1 \}.$$
For each $t>0$ small enough, we also define the unit vector
$\overrightarrow{\gamma_1 \gamma_2}(t) = \frac{\gamma_2(t) - \gamma_1 (t)}{\| \gamma_2(t) - \gamma_1 (t) \|}$. As $\gamma_1,\gamma_2$ are semialgebraic, the limit $\lim_{t \to 0^+} \frac{\gamma_2(t) - \gamma_1 (t)}{\| \gamma_2(t) - \gamma_1 (t) \|}$ exists and we denote it as $\overrightarrow{\gamma_1 \gamma_2}$. 

Given three arcs $\gamma_1,\gamma_2,\gamma_3 \subset C_a^{n+1}$ satisfying $tord(\gamma_i,\gamma_j) \ne \infty$, for $i \ne j$, we define, for each $t>0$ , the angle $\angle \gamma_1 \gamma_2 \gamma_3 (t)$ as the angle formed by $\overrightarrow{\gamma_2 \gamma_1} (t)$ and $\overrightarrow{\gamma_2 \gamma_3}(t)$. Similarly, we define the angle $\angle \gamma_1 \gamma_2 \gamma_3$ as the angle formed by $\overrightarrow{\gamma_2 \gamma_1}$ and $\overrightarrow{\gamma_2 \gamma_3}$.
\end{definition}

\begin{remark}
	Notice that $\overline{\gamma_1 \gamma_2} = \overline{\gamma_2 \gamma_1}$, $\overrightarrow{\gamma_1 \gamma_2} = - \overrightarrow{\gamma_2 \gamma_1}$ and $\overrightarrow{\gamma_1 \gamma_2}(t) = -\overrightarrow{\gamma_2 \gamma_1}(t)$, for every $t>0$ small enough.
\end{remark}

\begin{remark}\label{angle-linear-triangle}
If $(\overline{\gamma_1 \gamma_2} \cup \overline{\gamma_2 \gamma_3},0)$ is a LNE surface germ, then $\angle \gamma_1 \gamma_2 \gamma_3 > 0$ (such statement follows directly from Theorem \ref{Edson-Rodrigo} and Lemma \ref{abertura-limitada}). Hence, given $\varepsilon > 0$, we can assume that for all $t>0$ small enough, $\angle \gamma_1 \gamma_2 \gamma_3 (t) \in (\angle \gamma_1 \gamma_2 \gamma_3 -\varepsilon,\angle \gamma_1 \gamma_2 \gamma_3 +\varepsilon)$.
\end{remark}

\begin{proposition}\label{linear-triangle-is-holder}
Let $a>0$, $(\gamma_1,0), (\gamma_2,0) \subset (C_a^{3},0)$ be two arcs such that $tord(\gamma_1,\gamma_2) = \alpha \ne \infty$. Then $(\overline{\gamma_1 \gamma_2},0)$ is ambient bi-Lipschitz equivalent to the standard $\alpha$-H\"older triangle embedded in $\mathbb{R}^3$.
\end{proposition}

\begin{proof}
Let $\psi$ be the map of Proposition \ref{reduction}. If $T$ denotes the standard $\alpha$-H\"older triangle embedded in $\mathbb{R}^3$, then the germ of $\psi(T)$ is $(Y,0)$, where
$$Y = \{ (x, 0, t) \in C_a^{3} \mid 0 \le x \le h(t)\},$$
for some real function $h$ whose Puiseux series is $h(t)=d_{\alpha}t^{\alpha}+o(t^{\alpha})$; $d_\alpha >0$. Our goal is to prove that $(X,0)$ is ambient bi-Lipschitz equivalent to $(Y,0)$ in $(C_a^{3},0)$, since if this is done, then the result will follow by Proposition \ref{suficiencia-em-cone}.

Initially, notice that by Lemma \ref{translation}, applied to the arc $\gamma_2$, we can assume without loss of generality that $\gamma_2(t)=(0,0,t)$ and $\gamma_1(t )=\gamma(t)=(x(t),y(t),t)$. Since the limit $ \lim_{t \to 0^+} \frac{(x(t),y(t))}{\sqrt{x(t)^2+y(t)^2}}$ is well defined, we can apply a rotation of axes and assume that
$$x(t)>0 \; , \; \displaystyle \lim_{t \to 0^+} \frac{y(t)}{x(t)}=0.$$

Now consider the sets
$$T_1^{(1)}=\{ (x,x+x(t),t) \in C_a^{3} \mid -x(t) \le x \le 0 \},$$
$$T_1^{(2)} = \{ (x,x(t),t) \in C_a^{3} \mid 0 \le x \le 2x(t) \},$$
$$T_1^{(3)} = \{ (x,-x+3x(t),t) \in C_a^{3} \mid 2x(t) \le x \le 3x(t) \},$$
$$T_2^{(1)}=\{ (x,-x-x(t),t) \in C_a^{3} \mid -x(t) \le x \le 0 \},$$
$$T_2^{(2)} = \{ (x,-x(t),t) \in C_a^{3} \mid 0 \le x \le 2x(t) \},$$
$$T_2^{(3)} = \{ (x,x-3x(t),t) \in C_a^{3} \mid 2x(t) \le x \le 3x(t) \},$$
$$W_1^{(1)}=\left\{ (x,0,t) \in C_a^{3} \mid -x(t) \le x \le 0 \right\},$$
$$W_1^{(2)}=\left\{ (x,y,t) \in C_a^{3} \mid 0 \le x \le x(t) \; ; \; y=\left(\frac{x}{x(t)}\right)y(t) \right\},$$
$$W_1^{(3)}= \left\{ (x,y,t) \in C_a^{3} \mid x(t) \le x \le 2x(t) \; ; \; y=\left(\frac{2x(t)-x}{x(t)}\right)y(t) \right\},$$
$$W_1^{4}=\left\{ (x,0,t) \in C_a^{3} \mid 2x(t) \le x \le 3x(t) \right\},$$
$$W_2 = \left\{ (x,0,t) \in C_a^{3} \mid -x(t) \le x \le 3x(t) \right\},$$
and define
$$T_1 =T_1^{(1)}\cup T_1^{(2)}\cup T_1^{(3)},$$
$$T_2 =T_2^{(1)}\cup T_2^{(2)}\cup T_2^{(3)},$$
$$W_1 =W_1^{(1)}\cup W_1^{(2)}\cup W_1^{(3)}\cup W_1^{(4)}.$$

\begin{figure}[h]
\centering
\includegraphics[width=15cm]{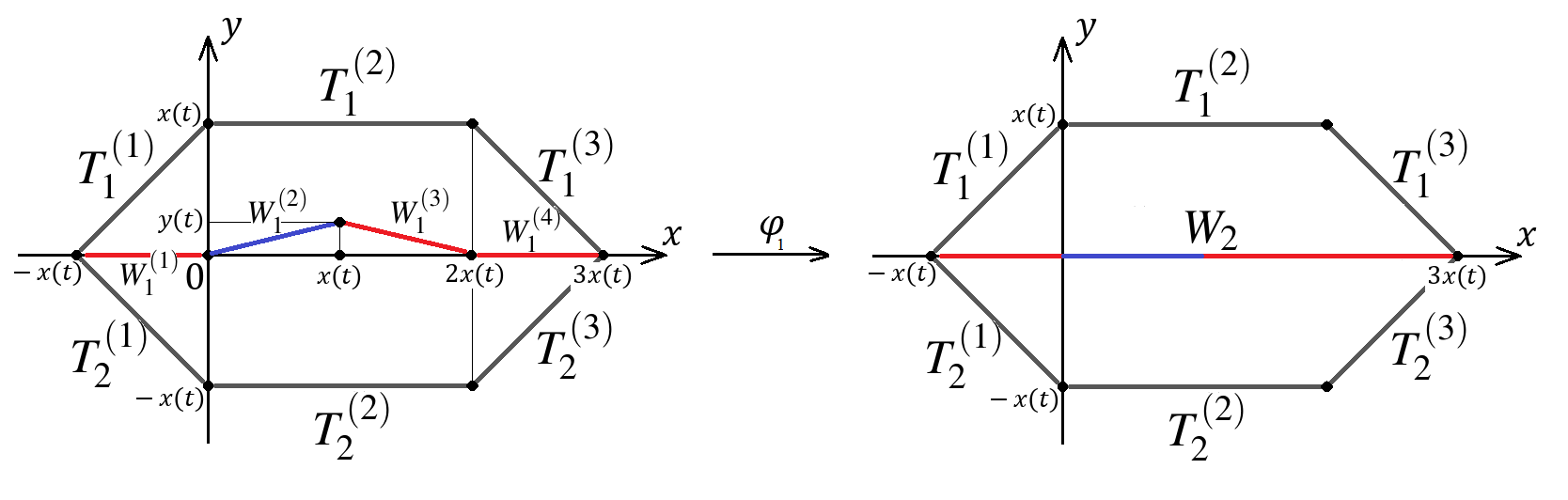}
\label{fig5a}
\caption{Proof of Proposition \ref{linear-triangle-is-holder}.}
\end{figure}

For each $\varepsilon>0$, the synchronized triangles $(T_1,0),(T_2,0),(W_1,0),(W_2,0)$ satisfies the conditions of Theorem \ref{amb-isotopy-criterion} for $M=2+\varepsilon$, as the limits
$$\displaystyle \lim_{t \to 0^+} \frac{x(t) - 0}{x(t) - (-x(t))} \; , \displaystyle \lim_{t \to 0^+} \frac{0 - (-x(t))}{x(t) - (-x(t))},$$
$$\displaystyle \lim_{t \to 0^+} \frac{x(t) - \frac{x}{x(t)}.y(t)}{x(t) - (-x(t ))} \; , \displaystyle \lim_{t \to 0^+} \frac{\frac{x}{x(t)}.y(t)-(-x(t))}{x(t) - (-x( t))},$$
$$\displaystyle \lim_{t \to 0^+} \frac{x(t) - \left(\frac{2x(t)-x}{x(t)}\right).y(t)} {x(t) - (-x(t))} \; , \displaystyle \lim_{t \to 0^+} \frac{\left(\frac{2x(t)-x}{x(t)}\right ).y(t)-(-x(t))}{x(t) - (-x(t))},$$
are equal to $\frac{1}{2}$, and the respective generating function families, which are piecewise linear, have derivatives bounded by $1$. Hence, by Remark \ref{amb-isotopy-extension}, there is an ambient bi-Lipschitz isotopy $\varphi: (C_a,0)\times[0,1] \to (C_a,0)$ taking $(W_1,0) $ into $(W_2,0)$. In particular, $(W_1^{(2)},0)=(\overline{\gamma_1 \gamma_2},0)$ is ambient bi-Lipschitz equivalent to the germ of $\tilde Y = \left\{ (x,0,t) \in C_a^{3} \mid 0 \le x \le x(t) \right\}$. Finally, as $tord(\gamma_1,\gamma_2) = \alpha \ne \infty$, and such value is Lipschitz invariant, it follows that the Puiseux expansion of $x(t)$ is $x(t)=c_{ \alpha}x^{\alpha} + o(t^{\alpha})$, for some $c_{\alpha}>0$. Applying Lemma \ref{dilatation} to the function $f(t)=\frac{h(t)}{x(t)}$ shows that $(\tilde Y ,0)$ is ambient bi-Lipschitz  equivalent to $(Y,0)$.
\end{proof}

\subsection{Kneadable Triangles and Supporting Envelopes}

Next, we will look at the concepts of kneadable triangles, supporting and kneading envelopes. We will also consider examples of such structures and a property relating them to synchronized triangles.

\begin{definition}\normalfont \label{kneadable}
Let $(T,0) \subset (C_a^3,0)$ be a H\"older triangle with main vertex at the origin, with $\gamma_1,\gamma_2$ as its boundary arcs, and let $(U,0)$ be a closed set germ containing $(T,0)$. We say that $(T,0)$ is \textit{kneadable in $(U,0)$} if there exists an ambient bi-Lipschitz isotopy in $U$ that takes $(T,0)$ into $(\overline{ \gamma_1 \gamma_2},0)$, invariant on the boundary of $(U,0)$.
\end{definition}

\begin{definition}\normalfont \label{supporting-envelope}
Let $a,M,\delta>0$ and let $(T,0) \subset (C_a^{3},0)$ be a $M$-bounded synchronized triangle germ. Let $\gamma_i(t)=(x_i(t),y_i(t),t)$, $i=0,1$, be the boundary arcs of $(T,0)$, with $x_0(t) < x_1(t)$, and let $\{ a_t \}$ be the generating functions family of $(T,0)$. For each small $t>0$ and for $i=0,1$, define
$$m_{t,i} = \inf\left\{\dfrac{a_t(x)-y_i(t)}{x-x_i(t)} \; ; \; x_0(t) < x < x_1(t)\right\},$$
$$M_{t,i} = \sup\left\{\dfrac{a_t(x)-y_i(t)}{x-x_i(t)} \; ; \; x_0(t) < x < x_1(t)\right\}.$$
Since $(T,0)$ is $M$-bounded, then $|m_{i,t}|,| M_{i,t}| \le M$. We define the \textit{$\delta$-supporting envelope of $(T,0)$} as the germ, at the origin, of 
$$U_{\delta}(T)= \left\{ (x,y,t) \mid t>0 \; ; \; x_0(t) \le x \le x_1(t); g_t(x) \le y \le f_t(x) \right\},$$
where
$$g_t(x)=max\{y_0(t)+(m_{t,0}-\delta)(x-x_0(t)); y_1(t)+(M_{t,1}+\delta)(x-x_1(t))\},$$
$$f_t(x)=min\{y_0(t)+(M_{t,0}+\delta)(x-x_0(t)); y_1(t)+(m_{t,1}-\delta)(x-x_1(t))\}.$$
\end{definition}

\begin{remark} \label{supporting-envelope-remark}
$(U_{\delta}(T),0)$ is the region delimited by the synchronized triangles $(T_1,0)$, $(T_2,0)$, whose generating functions families are $\{f_t\}$ and $\{g_t\}$, respectively. In the proof of Proposition \ref{synch-triangle-is-kneadable}, it will be shown that $g_t(x) \le a_t(x) \le f_t(x)$, for every small $t>0$ and every $x \in [x_1 (t),x_2(t)]$, which implies that $(U_{\delta}(T),0)$ is a germ of a closed semialgebraic set containing $(T,0)$.
\end{remark}

\begin{figure}[h]
\centering
\includegraphics[width=14cm]{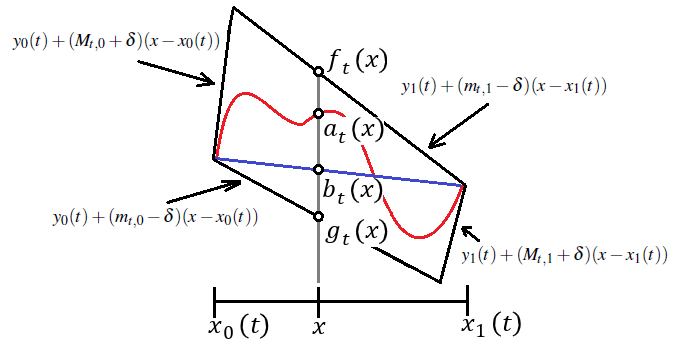}
\label{fig33}
\caption{Geometric representation of the $\delta$-supporting envelope of $(T,0)$.}
\end{figure}

\begin{example}\label{supporting-envelope-ex1}\normalfont
Consider the synchronized triangle given as the germ of
$$T= \{(x,-t-x,t) \in C_2^3 \mid -t \le x \le 0 \} \cup \{ (x,x-t,t) \in C_2^3 \mid 0 \le x \le t\}.$$
For every $0<\delta <1$, we have $U_{\delta}(T) = U_1 \cup U_2$, where
$$ U_1 =\{(x,y,t) \in C_2^3 \mid -t \le x \le 0 \; ; \; -(1+\delta)(t+x) \le y \le \delta (t+x) \},$$
$$ U_2 =\{(x,y,t) \in C_2^3 \mid 0 \le x \le t \; ; \; (1+\delta)(x-t) \le y \le \delta (t-x) \}.$$
More generally, if $(T,0) \subset (C_a^3,0)$ is a synchronized triangle, such that its plane link $T(t)$ is the union of two line segments $\overline{pq} , \overline{qr}$, with $x_p < x_q < x_r$, and if $u = \frac{y_q-y_p}{x_q-x_p}$, $v = \frac{y_r-y_p}{x_r-x_p }$, $w = \frac{y_r-y_q}{x_r-x_q}$ are the angular coefficients of $\overline{pq},\overline{pr},\overline{qr}$, with $u<v< w$, then $(U_{\delta}(T))(t)$ is the region bounded by the rays with origin at $p$, given as
$$l_{p,-} = \{(x,y,t) \mid y= y_p + (u-\delta)(x-x_p)\} \; , \; l_{p,+} = \{(x,y,t) \mid y= y_p + (v+\delta)(x-x_p)\},$$
and the rays with origin at $r$, given as
$$l_{r,-} = \{(x,y,t) \mid y= y_r + (w-\delta)(x-x_r)\} \; ; \; l_{r,+} = \{(x,y,t) \mid y= y_r + (v+\delta)(x-x_r)\}.$$

\begin{figure}[h]
\centering
\includegraphics[width=15cm]{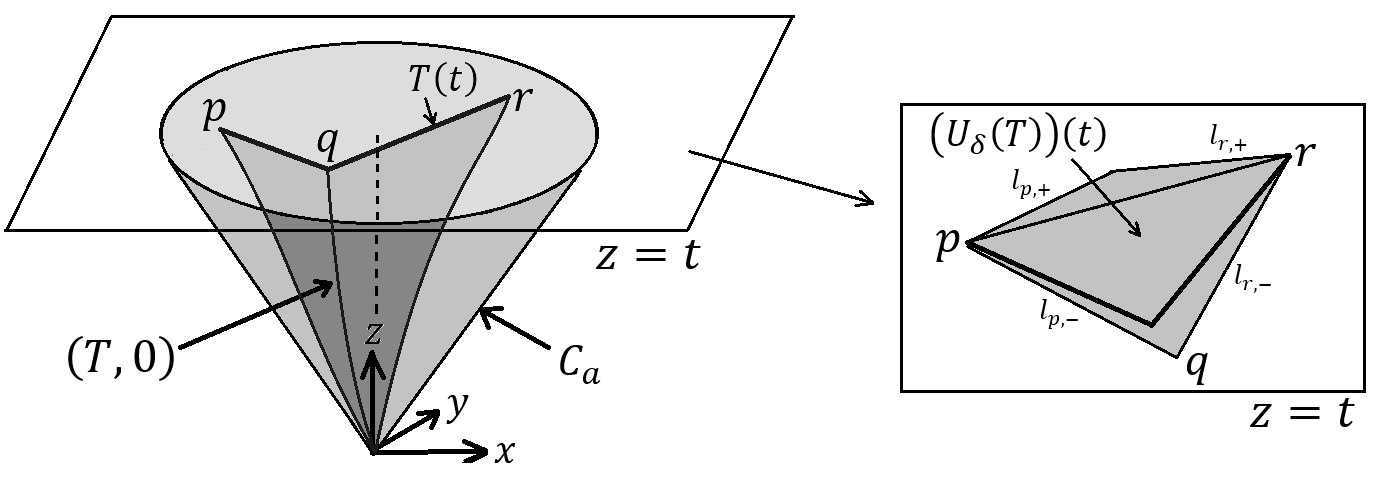}
\label{fig12}
\caption{Graphic Representation of Example \ref{supporting-envelope-ex1}.}
\end{figure}
\end{example}

\begin{example}\label{supporting-envelope-ex2}\normalfont
Consider the synchronized triangle given as the germ of
$$T= \{(x,x^2,t) \in C_2^3 \mid -t \le x \le t \}.$$
For every $0<\delta <1$, we have $U_{\delta}(T) = U_1 \cup U_2$, where
$$ U_1 =\{(x,y,t) \in C_2^3 \mid -t \le x \le 0 \; ; \; -(2t+\delta)(t+x)+t^2 \le y \le \delta (t+x) + t^2 \},$$
$$ U_2 =\{(x,y,t) \in C_2^3 \mid 0 \le x \le t \; ; \; (2t+\delta)(x-t)+t^2 \le y \le \delta (t-x) + t^2 \}.$$
More generally, if $(T,0) \subset (C_a^3,0)$ is a convex synchronized triangle, such that its plane link $T(t)$ is a convex function $f_t: [x_0,x_1] \to \mathbb{R}$, with
$$u = \frac{\partial f}{\partial x_+}(x_0) \; , \; v = \frac{f(x_1)-f(x_0)}{x_1-x_0} \; , \; w = \frac{\partial f}{\partial x_-}(x_1),$$
then $u<v<w$, and thus $(U_{\delta}(T))(t)$ is the region bounded by the rays with origin at $p=(x_0,f(x_0))$, given as
$$l_{p,-} = \{(x,y,t) \mid y= f(x_0) + (u-\delta)(x-x_0)\} \; , \; l_{p,+} = \{(x,y,t) \mid y= f(x_0) + (v+\delta)(x-x_0)\},$$
and by the rays with origin at $r=(x_1,f(x_1))$, given as
$$l_{r,-} = \{(x,y,t) \mid y= f(x_1) + (w-\delta)(x-x_1)\} \; , \; l_{r,+} = \{(x,y,t) \mid y= f(x_1) + (v+\delta)(x-x_1)\}.$$

\begin{figure}[h]
\centering
\includegraphics[width=15cm]{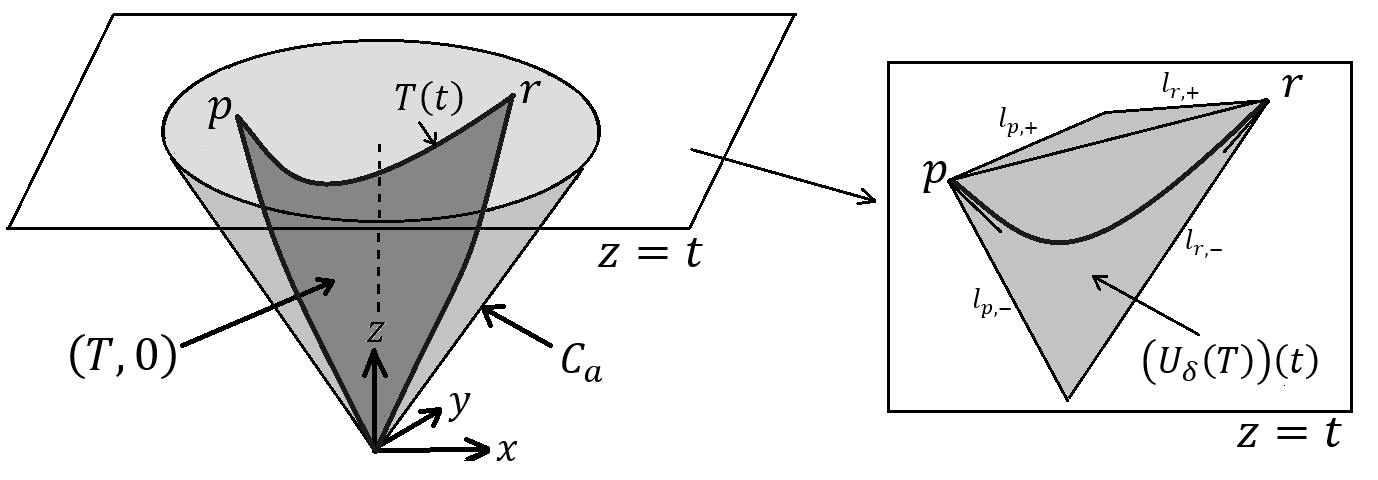}
\label{fig13}
\caption{Graphic representation of Example \ref{supporting-envelope-ex2}.}
\end{figure}
\end{example}

\begin{proposition}\label{synch-triangle-is-kneadable}
Let $M, \delta>0$ and let $(T,0)$ be a $M$-bounded synchronized triangle germ. Then $(T,0)$ is kneadable in its $\delta$-supporting envelope $U_\delta(T)$.
\end{proposition}

\begin{proof}
Here we use the notations as in Definition \ref{supporting-envelope} and Remark \ref{supporting-envelope-remark}. For each $t>0$, let $r_t$ be the angular coefficient of the line connecting $\gamma_0(t)$ and $\gamma_1(t)$ and let $b_t: [x_0(t),x_1(t)] \to \mathbb{R}$, given by $b_t(x)=y_0(t)+(x-x_0(t))r_t$, $x_0(t) \le x \le x_1(t)$. Let $(W_1,0)$, $(W_2,0)$ be the synchronized triangles whose generating functions families are $\{a_t\}$, $\{b_t\}$, respectively. Let us check if $f_t,g_t, a_t,b_t$ fulfies the conditions of Theorem \ref{amb-isotopy-criterion}. By the definition of $M_{t,i},m_{t,i}$, we have that, for each $t>0$ and $x \in (x_0(t),x_1(t))$, $m_{ t,0}, m_{t,1} \le r_t \le M_{t,0}, M_{t,1}$. 
\begin{enumerate}
\item $g_t(x) < a_t(x),b_t(x)$, because if $g_t(x)=y_0(t)+(m_{t,0}-\delta)(x-x_0(t) )$, then
$$g_t(x)=y_0(t)+(m_{t,0}-\delta)(x-x_0(t)) < y_0(t) + m_{t,0}(x-x_0(t) ) \le$$
$$\le y_0(t) + \left(\dfrac{a_t(x)-y_0(x)}{x-x_0(t)}\right)(x-x_0(t))=a_t(x),$$
$$g_t(x)=y_0(t)+(m_{t,0}-\delta)(x-x_0(t)) < y_0(t) + m_{t,0}(x-x_0(t) ) \le$$
$$\le y_0(t) + \left(r_t{x-x_0(t)}\right)(x-x_0(t))=b_t(x),$$
and if $g_t(x) = y_1(t)+(M_{t,1}+\delta)(x-x_1(t))$, then
$$g_t(x)=y_1(t)+(M_{t,1}+\delta)(x-x_1(t)) < y_1(t) + M_{t,1}(x-x_1(t) ) \le$$
$$\le y_1(t) + \left(\dfrac{a_t(x)-y_1(x)}{x-x_1(t)}\right)(x-x_1(t))=a_t(x),$$
$$g_t(x)=y_1(t)+(M_{t,1}+\delta)(x-x_1(t)) < y_1(t) + M_{t,1}(x-x_1(t) ) \le$$
$$\le y_1(t) + \left(r_t{x-x_1(t)}\right)(x-x_1(t))=b_t(x).$$

\item $f_t(x) > a_t(x), b_t(x)$ (the proof is analogous).

\item If $\tilde{M} = \max \left\{M, \dfrac{2(\delta + M)}{\delta} \right\}$, then we will prove that the functions $a_t,b_t,f_t,g_t $ have derivatives bounded by $\tilde{M}$ and are such that, for all $x \in (x_0(t),x_1(t))$,
$$\dfrac{1}{\tilde{M}} \cdot \left( f_t(x) - g_t(x) \right) \le a_t(x)-g_t(x),b_t(x)-g_t( x) \le \left(1- \dfrac{1}{\tilde{M}}\right) \cdot \left( f_t(x) - g_t(x) \right).$$
\end{enumerate}

Let's separate the analysis into two cases, according to the value of $g_t(x)$.

\textit{Case 1:} If $g_t(x)=y_0(t)+(m_{t,0}-\delta)(x-x_0(t))$, then, from the definition of $f_t(x)$, we have $f_t(x ) \le y_0(t)+(M_{t,0}+\delta)(x-x_0(t))$ and so
$$\dfrac{a_t(x)-g_t(x)}{f_t(x)-g_t(x)} = \dfrac{\dfrac{a_t(x)-y_0(t)}{x-x_0(t) }-\dfrac{g_t(x)-y_0(t)}{x-x_0(t)}}{\dfrac{f_t(x)-y_0(t)}{x-x_0(t)}-\dfrac{ g_t(x)-y_0(t)}{x-x_0(t)}} \ge \dfrac{\dfrac{a_t(x)-y_0(t)}{x-x_0(t)}-(m_{t ,0}-\delta)}{(M_{t,0}+\delta)-(m_{t,0}-\delta)} \ge $$
$$\ge \dfrac{m_{t,0}-(m_{t,0}-\delta)}{(M_{t,0}+\delta)-(m_{t,0}-\delta) }=\dfrac{\delta}{2\delta + (M_{t,0}-m_{t,0})} \ge \dfrac{\delta}{2(\delta + M)} \ge \dfrac {1}{\tilde{M}}.$$

\textit{Case 2:} $g_t(x)=y_1(t)+(M_{t,1}+\delta)(x-x_1(t))$; then, from the definition of $f_t(x)$, we have $f_t(x) \le y_1(t)+(m_{t,1}-\delta)(x-x_1(t))$ and so
$$\dfrac{a_t(x)-g_t(x)}{f_t(x)-g_t(x)} = \dfrac{\dfrac{a_t(x)-y_1(t)}{x-x_1(t) }-\dfrac{g_t(x)-y_1(t)}{x-x_1(t)}}{\dfrac{f_t(x)-y_1(t)}{x-x_1(t)}-\dfrac{ g_t(x)-y_1(t)}{x-x_1(t)}} \ge \dfrac{\dfrac{a_t(x)-y_1(t)}{x-x_1(t)}-(M_{t ,1}+\delta)}{(m_{t,1}-\delta)-(M_{t,1}+\delta)} \ge $$
$$\ge \dfrac{M_{t,1}-(M_{t,1}+\delta)}{(m_{t,0}-\delta)-(M_{t,0}+\delta) }=\dfrac{\delta}{2\delta + (M_{t,0}-m_{t,0})} \ge \dfrac{\delta}{2(\delta + M)} = \dfrac{ 1}{\tilde{M}}.$$

So, in both cases, $\frac{1}{\tilde{M}} \cdot \left( f_t(x) - g_t(x) \right) \le a_t(x)-g_t(x)$. Similarly, we have $\frac{1}{\tilde{M}} \cdot \left( f_t(x) - g_t(x) \right) \le f_t(x)-a_t(x)$, which implies
$$a_t(x)-g_t(x) = (f_t(x)-g_t(x)) - (f_t(x) - a_t(x)) \le \left(1- \dfrac{1}{\tilde {M}}\right) \cdot \left( f_t(x) - g_t(x) \right).$$

Analogously,

$$\dfrac{1}{\tilde{M}} \cdot \left( f_t(x) - g_t(x) \right) \le b_t(x)-g_t(x) \le \left(1- \dfrac{1}{\tilde{M}}\right) \cdot \left( f_t(x) - g_t(x) \right). $$
\end{proof}

\section{Reduction to Linear Triangles}

In this section, we will use the ideas seen in the previous section to prove that the ambient bi-Lipschitz classification problem can be reduced to the analysis of surfaces germs that are a finite union of linear triangles. We also solve the problem for germs which link is either a circle or two lines, which is the base case for the main theorem, for the connected link case.

\begin{proposition}\label{2linear-knead-to-1linear}
Let $(\gamma_i,0) \in (C_a^{3},0)$ ($i=1,2,3$) be distinct arcs such that
$$(Y,0)=(\overline{\gamma_1 \gamma_2} \cup \overline{\gamma_2 \gamma_3}, 0)$$
is a LNE germ. Then, $(X,0)$ is ambient bi-Lipschitz equivalent to $(\overline{\gamma_1 \gamma_3}, 0)$.
\end{proposition}

\begin{proof}
For $i=1,2,3$, let $\theta_i (t) = \angle \gamma_{i-1}\gamma_i \gamma_{i+1} (t)$ and $\theta_i = \angle \gamma_ {i-1}\gamma_i \gamma_{i+1}$ (consider indices modulo 3). According to Remark \ref{angle-linear-triangle}, $\theta_2 > 0$. Since $\theta_1 + \theta_2 + \theta_3 = \pi$, we have two cases to consider.

\textit{Case 1:} $\theta_1 , \theta_3 < \frac{\pi}{2}$. Taking a rotation of axes if necessary, we can assume that $\overrightarrow{\gamma_1 \gamma_3} = (1,0,0)$, so that $(Y,0)$ is a $M$-bounded synchronized triangle, where $M= \max\{ \tan(\theta_1+\varepsilon),\tan(\theta_3+\varepsilon)\}$, for some small enough $\varepsilon>0$ (see Example \ref{supporting-envelope-ex1}). By Proposition \ref{synch-triangle-is-kneadable}, the result is proved in this case.

\textit{Case 2:} $\theta_1 \ge \frac{\pi}{2}$ or $\theta_3 \ge \frac{\pi}{2}$. Let us just analyze $\theta_1 \ge \frac{\pi}{2}$, as the other case is analogous. For every $t>0$, let $\gamma(t) \in \overline{\gamma_2(t) \gamma_3(t)}$ such that $\| \gamma_2(t) - \gamma_1(t) \| = \| \gamma_2(t)-\gamma(t) \|$. Notice that $\gamma=\gamma(t)$ is a semialgebraic set, hence an arc. So, taking a rotation of axes, if necessary, we can assume that $\overrightarrow{\gamma_2\gamma} = (1,0,0)$, so that $(\overline{\gamma_1\gamma_2}\cup \overline {\gamma_2 \gamma},0)$ is the germ of a $M$-bounded synchronized triangle, for some $M>0$, since $\angle \gamma_2 \gamma \gamma_1 = \angle \gamma_2 \gamma_1 \gamma = \frac{\pi- \theta_2} {2} < \frac{\pi}{2}$. By Proposition \ref{synch-triangle-is-kneadable}, $(Y,0)$ is ambient bi-Lipschitz equivalent to $(\overline{\gamma_1\gamma}\cup \overline{\gamma \gamma_3},0)$, and since $\angle \gamma \gamma_1 \gamma_3 , \angle \gamma \gamma_3 \gamma_1 < \frac{\pi}{2}$ (because $\angle \gamma_1 \gamma \gamma_3 = \frac{\pi + \theta_2 }{2} > \frac{\pi}{2}$), the result reduces to Case 1.

\begin{figure}[h]
\centering
\includegraphics[width=15cm]{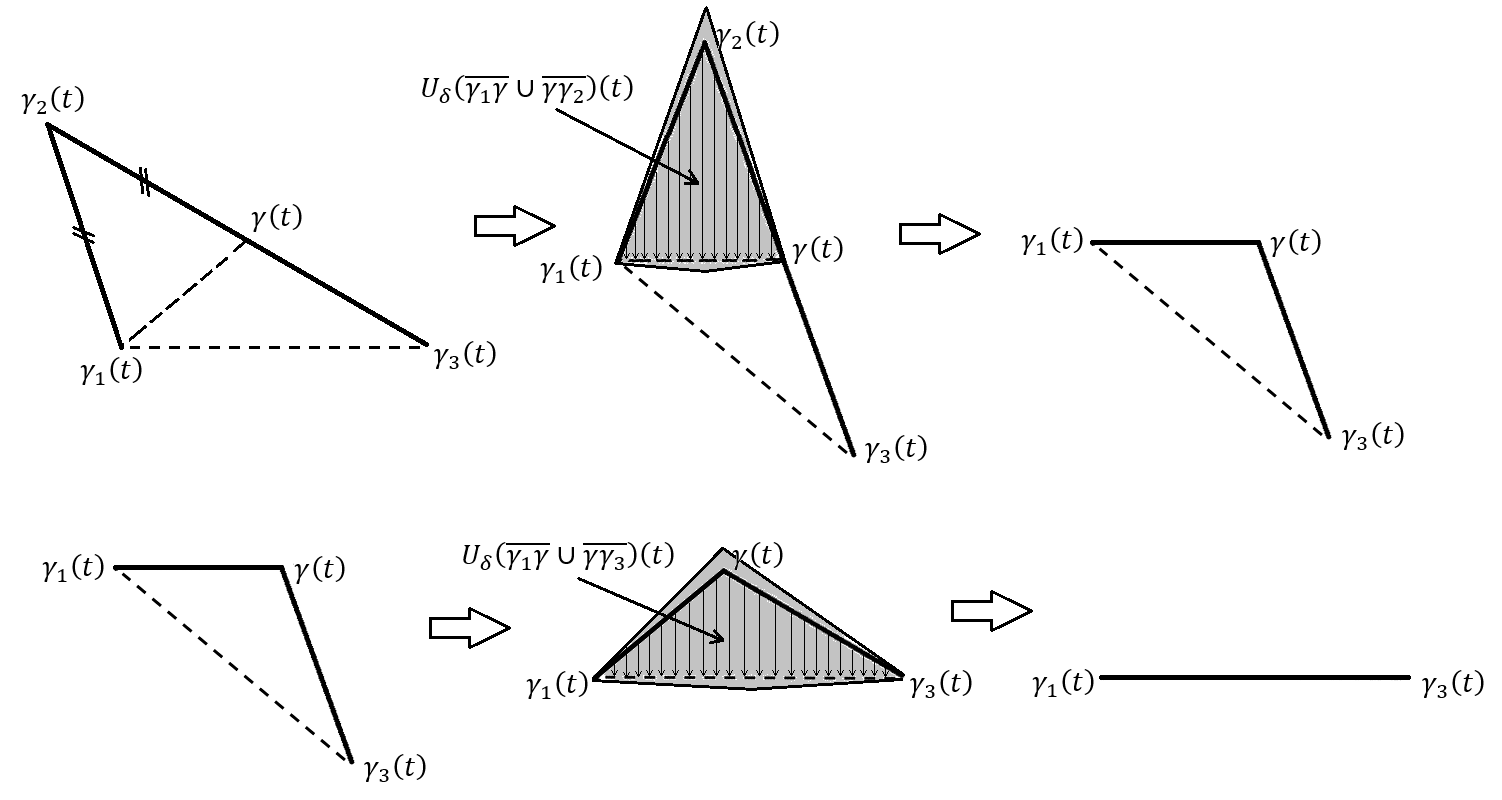}
\label{fig20}
\caption{Case 2 of Proposition \ref{2linear-knead-to-1linear}.}
\end{figure}
\end{proof}

\begin{definition}\normalfont \label{kneading-envelope}
Let $\gamma_1, \gamma_2, \gamma_3$ be arcs and let $Y = \overline{\gamma_1 \gamma_2} \cup \overline{ \gamma_2 \gamma_3}$ be a surface whose germ is LNE. Given $\theta>0$, for each $t>0$ and for $i=1, 3$, let $r_{i,+}(t), r_{i,-}(t)$ be lines passing through $\gamma_i(t)$ and external to the triangle $\gamma_1 \gamma_2 \gamma_3 (t)$ such that
$$\angle (r_{i,-}(t),\overline{\gamma_1 \gamma_3}(t)) = \angle (r_{i,+}(t), \overline{\gamma_i \gamma_2}( t)) = \theta.$$
For $\theta$ small enough, the lines $r_{1,+} (t) ,r_{3,+} (t)$ intersect at a point $\gamma_+(t)$ and the lines $r_{1 ,-} (t) ,r_{3,-} (t)$ intersect at a point $\gamma_-(t)$. Let the arcs $\gamma_+ = \gamma_+(t)$, $\gamma_- = \gamma_-(t)$, $V_{\theta}(t)$ be the convex quadrilateral with vertices $\gamma_1(t)$ , $\gamma_+(t)$, $\gamma_3(t)$ and $\gamma_-(t)$ and $V_{\theta}(Y) = \cup_{t>0} U_{\theta}( t) \cup \{0\}$. We define the \textit{$\theta$-kneading envelope of $Y$} as the germ of the set $V_{\theta}(Y)$ and denote $\gamma_+$ as the boundary arc of this $\theta $-envelope closest to $\gamma_2$.
\end{definition}

\begin{figure}[h]
\centering
\includegraphics[width=10cm]{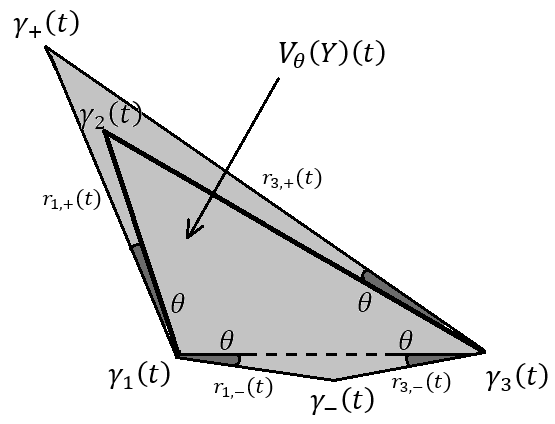}
\label{fig21}
\caption{The $\theta$-kneading envelope of $Y=\overline{\gamma_1 \gamma_2} \cup \overline{ \gamma_2 \gamma_3}$.}
\end{figure}

\begin{remark}\label{kneading-envelope-remark}
The $\theta$-kneading envelope of $Y$ is not necessarily a supporting envelope of $Y$, since it is possible that $\angle \gamma_2 \gamma_1 \gamma_3 > \frac{\pi}{2} $. However, for every $\theta>0$ small enough, we can take $\delta>0$ such that the $\delta$-supporting envelopes (or their respective images per rotation of axes) involved in the proof of Proposition $\ref {2linear-knead-to-1linear}$ are all contained in the $\theta$-kneading envelope of $Y$, so that $Y$ is kneadable in $V_\theta(Y)$. In particular, if $X \subset C_a^{3}$ is a closed semialgebraic surface such that $(Y,0) \subseteq (X,0)$, and if there exists $\theta>0$ such that $(V_ {\theta}(Y),0) \cap (X \setminus Y,0) = \emptyset$, then $(X,0)$ is ambient bi-Lipschitz equivalent to $((X \setminus Y)\cup \overline{\gamma_1 \gamma_3},0)$.
\end{remark}

\begin{proposition}\label{linear-triangle-is-horn}
Let $(\gamma_i,0) \in (C_a^{3},0)$ ($i=1,2,3$) be distinct arcs such that
$$(X,0)=(\overline{\gamma_1 \gamma_2} \cup \overline{\gamma_2 \gamma_3} \cup \overline{\gamma_3 \gamma_1},0)$$
is LNE. Then, there is $\beta \in \mathbb{Q}_{\ge 1}$ such that $(X,0)$ is ambient bi-Lipschitz equivalent to the $\beta$-standard horn germ $(H_{\beta},0)$.
\end{proposition}

\begin{proof}
For $i=1,2,3$, let $\theta_i (t) = \angle \gamma_{i-1}\gamma_i \gamma_{i+1} (t)$ and $\theta_i = \angle \gamma_ {i-1}\gamma_i \gamma_{i+1}$ (consider indices modulo 3). By Remark \ref{angle-linear-triangle}, $\theta_1,\theta_2,\theta_3 > 0$, which implies that there exists $\beta \in \mathbb{Q}_{\ge 1}$ such that
$$tord(\gamma_1,\gamma_2) = tord(\gamma_2,\gamma_3) = tord(\gamma_3,\gamma_1) = \beta.$$
Thus, there are $a_1,a_2,a_3 > 0$ such that:
$$ \| \gamma_2(t) - \gamma_3(t) \| = a_1t^{\beta} + o(t^{\beta}),$$
$$ \| \gamma_3(t) - \gamma_1(t) \| = a_2t^{\beta} + o(t^{\beta}),$$
$$\| \gamma_1(t) - \gamma_2(t) \| = a_3t^{\beta} + o(t^{\beta}).$$
Notice that $a_1,a_2,a_3$ are the lengths of the sides of a triangle. It is well known that the circunradius of a triangle with sides $x,y,z$ and area:
$$S=\frac{\sqrt{(x+y+z)(-x+y+z)(x-y+z)(x+y-z)}}{4}$$
is $R=\frac{xyz}{4S}$. Then, if
$$c=\frac{a_1a_2a_3}{\sqrt{(a_1+a_2+a_3)(-a_1+a-2+a_3)(a_1-a_2+a_3)(a_1+a_2-a_3)}},$$
we have that, for all $t>0$ small enough, if $\Gamma(t)$ is the circuncircle of triangle $\gamma_1(t) \gamma_2 (t) \gamma_3 (t)$ (which is non-degenerate, since $\theta_i >0$), the radius $R(t)$ of $\Gamma(t)$ has a Puiseux expansion given by $R(t) = ct^\beta + o(t^{\beta })$. By Lemma \ref{translation}, we can also assume that $\Gamma(t)$ has center at the origin. Defining $\Gamma := \left(\cup_{t>0} \Gamma(t) \right) \cup \{0\} \subset C_a^3$, we will prove that $(\Gamma,0)$ is ambient bi-Lipschitz equivalent to $(X,0)$ in $C_a^3$. Let us prove this by dividing into two cases.

\textit{Case 1:} if $\theta_1,\theta_2,\theta_3 < \frac{\pi}{2}$, we can assume that $t>0$ is small enough so that $\gamma_1(t)\gamma_2 (t)\gamma_3(t)$ is an acute triangle, whose internal angles satisfy, for $i=1,2,3$, $\varepsilon < \theta_i(t) < \frac{\pi}{2} - \varepsilon$, for some small $\varepsilon > 0$. In this way, if $\Gamma_i (t)$ is the arc of $\Gamma(t)$ delimited by $\gamma_{i-1}(t)$ and $\gamma_{i+1}(t)$ and which does not contain $\gamma_i(t)$ (indices modulo 3) and if $\Gamma_i = \left(\cup_{t>0} \Gamma_i(t) \right) \cup \{0\} \subset C_a^ 3$, by a rotation of axes, if necessary, we can see $\Gamma_i(t)$ as a plane link of a $M$-bounded, convex synchronized triangle $(T,0)$, with $M=\cot{\varepsilon}$. This is due to the fact that the tangents to $\Gamma$ by $\gamma_{i-1}(t)$ and $\gamma_{i+1}(t)$ form with $\overline{ \gamma_{i-1}\gamma_{i+1}(t)}$ angles smaller than $\frac{\pi}{2} - \varepsilon$. 

\begin{figure}[h]
	\centering
	\includegraphics[width=8cm]{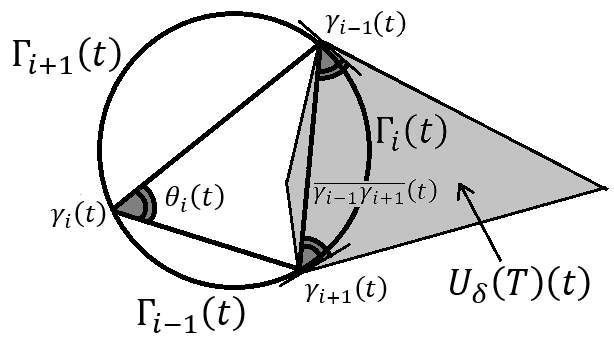}
	\label{fig14}
	\caption{Construction of the $\delta$-supporting envelope of $\Gamma_i(t)$.}
\end{figure}

Hence, there exists $\delta>0$ such that the $\delta$-supporting envelope of $\Gamma_i(t)$ does not intersect $(X(t)\cup \Gamma (t)) \setminus (\overline{ \gamma_{i-1} \gamma_{i+1} (t)} \cup \Gamma_i (t))$, where $(T,0)$ is ambient bi-Lipschitz equivalent to $(\overline{\gamma_ {i-1}\gamma_{i+1}},0)$, by Proposition \ref{synch-triangle-is-kneadable}. Therefore, each $\Gamma_i$ is kneadable in its respective supporting envelope. Applying such transformations on each supporting envelope, we obtain that $(\Gamma,0)$ is ambient bi-Lipschitz equivalent to $(X,0)$ in $(C_a^3,0)$.

\begin{figure}[h!]
\centering
\includegraphics[width=15cm]{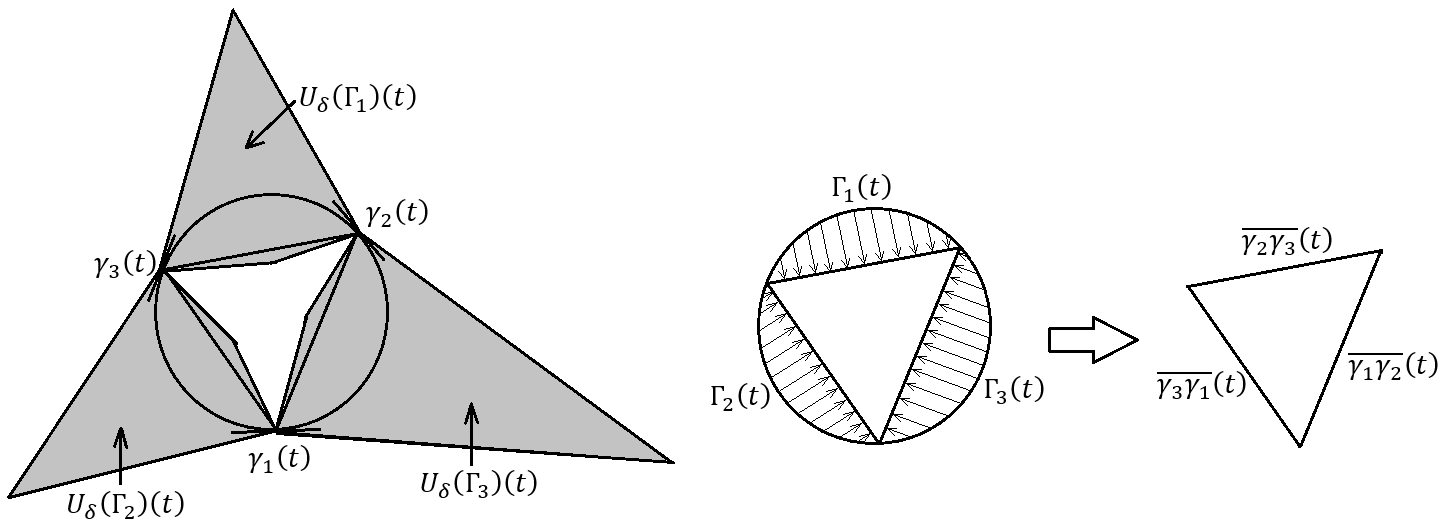}
\label{fig15}
\caption{"Kneading" the arcs $\Gamma_1, \Gamma_2, \Gamma_3$ in Case 1 of Proposition \ref{linear-triangle-is-horn}.}
\end{figure}

\textit{Case 2:} there is $i \in \{ 1, 2, 3 \}$ such that $\theta_i \ge \frac{\pi}{2}$ (say $i=1$). In this case, $\Gamma_2 ,\Gamma_3$ are kneadable in their respective supporting envelopes, but the same is not true for $\Gamma_1$. To work around this situation, we take, for each $t>0$, $\gamma(t)$ as the midpoint of the arc $\Gamma_1(t)$ and define $\Gamma_{1,2}(t)$ as the region in $\Gamma_1(t)$ bounded by $\gamma_2(t), \gamma(t)$ and $\Gamma_{1,3}(t)$ the region in $\Gamma_1(t)$ bounded by $\gamma_3(t), \gamma(t)$. Also denote $\Gamma_{1,j} = \left(\cup_{t>0} \Gamma_{1,j}(t) \right) \cup \{0\} \subset C_a^3$ ($j =2.3$). By the same arguments of Case 1, we have that $\Gamma_2, \Gamma_3, \Gamma_{1,2}, \Gamma_{1,3}$ are kneadable in their respective supporting envelopes, where $(\Gamma,0)$ is ambient bi-Lipschitz equivalent to $(\overline{\gamma_2 \gamma_1} \cup \overline{\gamma_1 \gamma_3} \cup \overline{\gamma_3 \gamma} \cup \overline{\gamma \gamma_2},0) $ in $(C_a^3.0)$.

Since $\angle \gamma \gamma_2 \gamma_3 (t) = \angle \gamma \gamma_3 \gamma_2 (t) = \frac{\theta_1}{2} < \frac{\pi}{2} - \varepsilon$ , we have that $\overline{\gamma_2 \gamma} \cup \overline{\gamma \gamma_3}$ is kneadable in a $\delta$-supporting envelope that does not intersect interior points of $\overline{\gamma_2 \gamma_1} \cup \overline{\gamma_3 \gamma_1}$, and so $(\overline{\gamma_2 \gamma_1} \cup \overline{\gamma_1 \gamma_3} \cup \overline{\gamma_3 \gamma} \cup \overline{\gamma_2 \gamma},0)$ is ambient bi-Lipschitz equivalent to $(\overline{\gamma_2 \gamma_1} \cup \overline{\gamma_1 \gamma_3} \cup \overline{\gamma_2 \gamma_3},0) = (X ,0)$ into $(C_a^3,0)$, as we wanted. As $(X,0)$ is ambient bi-Lipschitz equivalent to $(\Gamma,0)$ in $C_a^3$, by Corollary \ref{ajuste-cornetas}, we obtain the result.
\end{proof}

\begin{figure}[h!]
\centering
\includegraphics[width=15cm]{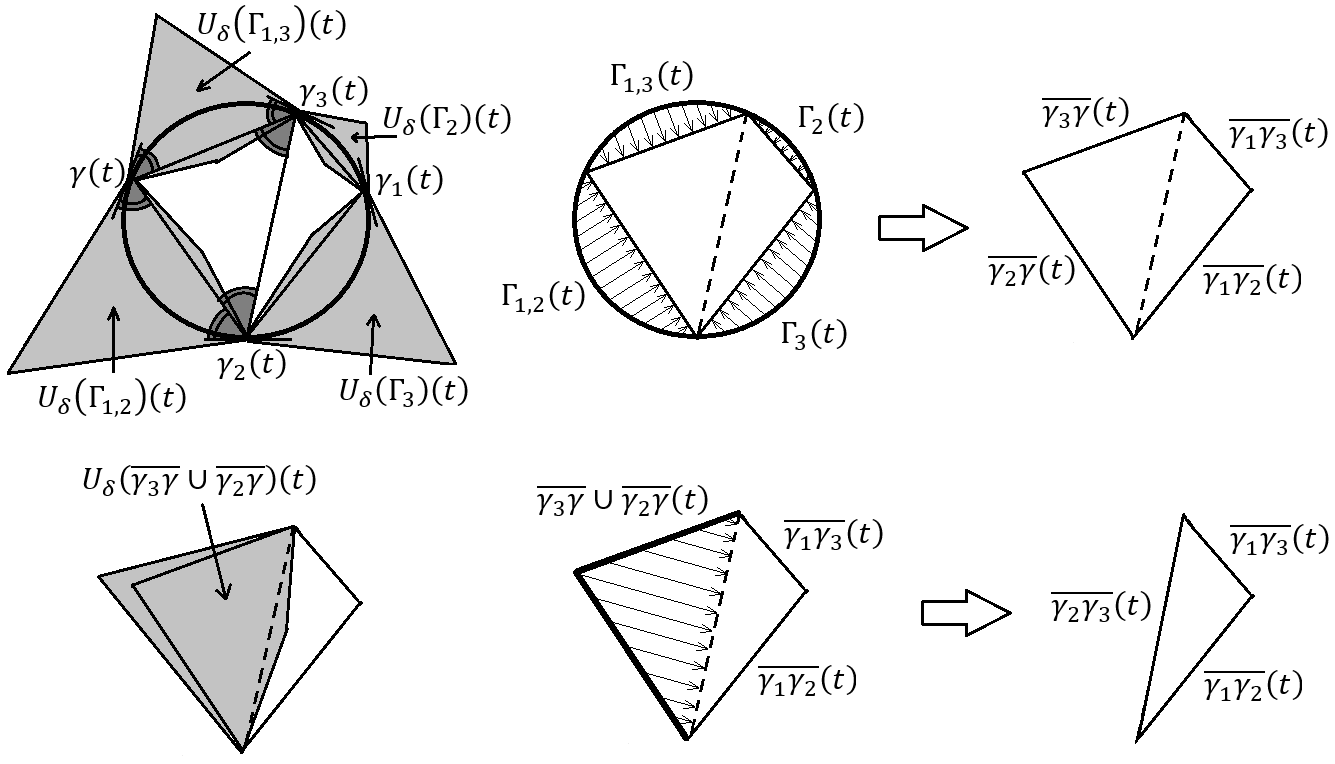}
\label{fig16}
\caption{The $\delta$-envelopes and "kneads" in Case 2 of Proposition \ref{linear-triangle-is-horn}.}
\end{figure}

\begin{proposition}\label{linear-decomposition}
Let $a>0$ and let $(X,0) \subset (C_{a}^{3},0)$ be a pure, closed, semialgebraic, 2-dimensional LNE surface germ with connected link. Then, $(X,0)$ is ambient bi-Lipschitz equivalent to a germ of a surface formed by a finite union of linear triangles germs.
\end{proposition}

\begin{proof}
By Proposition \ref{convex-decomp}, for each $\delta>0$, there is a $\delta$-convex decomposition $(X,0) = \bigcup_{i=1}^{n} (X_i,0 )$. For $i=1,\dots, n$, let $(\gamma_{i,0},0)$, $(\gamma_{i,1},0)$ be the boundary arcs of the triangle $(X_i, 0)$. If $T_i = r_{\theta_i}(X_i)$, $U_{\delta}(T_i)$ is the $\delta$-supporting envelope of $T_i$ and $U_i=r_{-\theta_i}(U_ {\delta}(T_i))$, then since $T_i$ is a $\delta$-bounded convex synchronized triangle, we have that, for every $C>1$, there is $\delta>0$ small enough such that, for all $t>0$ small enough and all points $p,q \in U_i(t)$,
$$\angle p \gamma_{i,0}(t) q , \angle p \gamma_{i,1}(t) q < \arctan{\frac{1}{4C}}.$$

By Theorem \ref{Edson-Rodrigo}, there is $C>1$ such that $X(t)$ is $C$-LNE, for every small $t>0$. By taking $C$ fulfilling this condition and $\delta$ as above, we claim that, for $t>0$ and all $i \le i < j \le n$, $U_i(t) \cap U_j(t) $ is equal to a point on some boundary arc or equal to $\emptyset$. Suppose the opposite, that is, there is $P \in U_i(t) \cap U_j(t)$, $P \ne \gamma_{i,0}(t), \gamma_{j,0}(t),\gamma_{i,1}(t), \gamma_{j,1}(t)$. For $k=i,j$, the perpendicular, by $P$, to $\overline{\gamma_{k,0} \gamma_{k,1}}(t)$ intersects $X_k(t)$ in $ P_k$ and $\overline{\gamma_{k,0} \gamma_{k,1}}(t)$ in $Q_k$. Let $A \in \{\gamma_{i,0}(t),\gamma_{i,1}(t) \}$, $B \in \{ \gamma_{j,0}(t),\gamma_{j,1}(t)\}$ such that the path in $X(t)$ of minimum length connecting $P_i$ to $P_j$ passes through $A$ and $B$. In particular,
$$d_{X(t)}(P_i,P_j) \ge \| P_i - A\| +\| A-B\| + \|B- P_j \|.$$

\begin{figure}[h!]
\centering
\includegraphics[width=8cm]{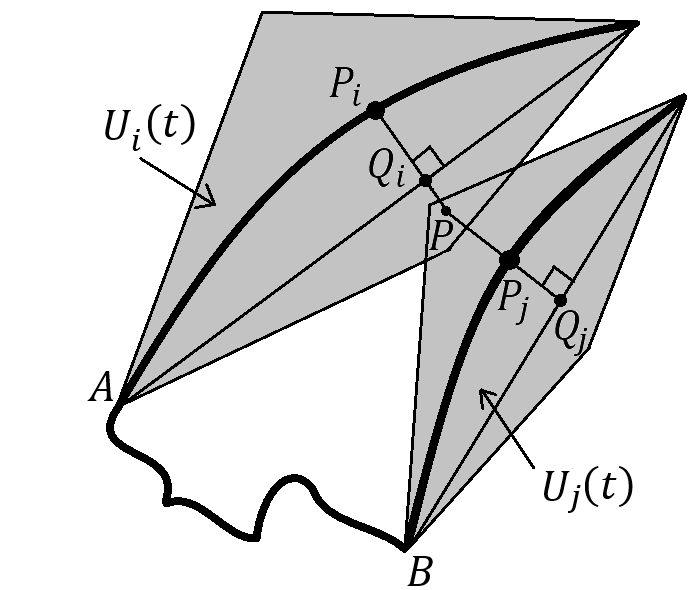}
\label{fig17}
\caption{Proof that $U_i(t) \cap U_j(t)$ is a boundary arc point, or empty.}
\end{figure}

We have
$$\| P_i - P_j\| \le \|P_i - P \| + \|P - P_j \| \le \| P_i - Q_i \| + \| Q_i - P\| + \| P - P_j\| + \|P_j - Q_j \|.$$
Since $\angle P_i A Q_i, \angle PAQ_i, \angle PBQ_j, \angle P_j B Q_j < \arctan{\frac{1}{4C}}$, we get
$$\| P_i - Q_i \| , \| Q_i - P\| \le \frac{\|A - Q_i \|}{4C} \le \frac{\|A - P_i \|}{4C};$$
$$\| P - P_j\| , \|P_j - Q_j \| \le \frac{\| B - Q_j \|}{4C} \le \frac{\| B - P_j\|}{4C}.$$
Therefore,
$$\| P_i - P_j\| \le \frac{\|P_i - A \|+ \| B - P_j\|}{2C} \le \frac{\| P_i - A\| +\| A-B\| + \|B- P_j \|}{2C} \le \frac{d_{X(t)}(P_i,P_j)}{2C},$$
a contradiction. Thus, each $(X_i,0)$ is kneadable into each corresponding $(U_i,0)$.
\end{proof}

\begin{figure}[h]
\centering
\includegraphics[width=15cm]{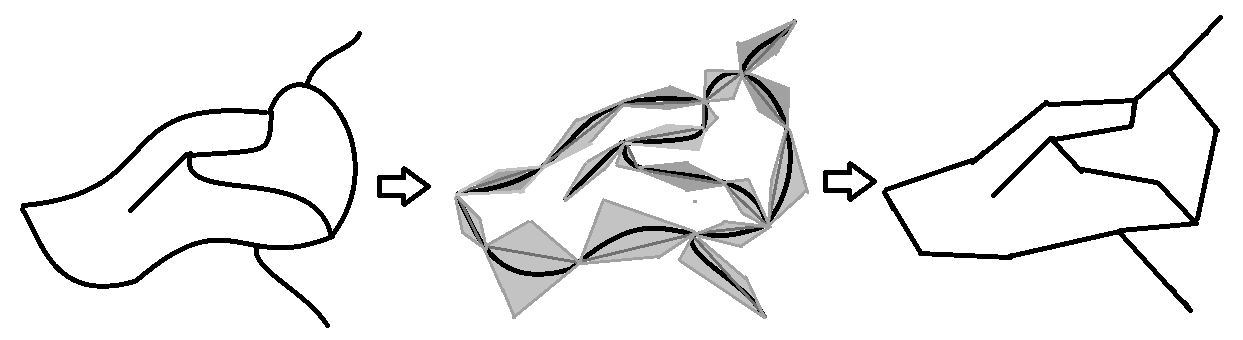}
\label{fig34}
\caption{Proof of Proposition \ref{linear-decomposition}.}
\end{figure}

\section{Polygonal Surfaces}

In this section we consider polygonal surfaces, i.e. LNE surface germs such that the intersection with a parallel plane, situated near the singular point, are polygons. We prove some auxiliary lemmas in order to develop tools for the edge reduction that will be done on the next section.

\begin{definition}\normalfont \label{poly-surface-1}
Let $a>0$ and let $(X,0) \subset (C_a^{3},0)$ be a closed, semialgebraic surface germ. We say that $(X,0)$ is a \textit{polygonal surface germ} if there are $n \in \mathbb{Z}_{\ge 2}$, distinct arcs $\gamma_1, \dots, \gamma_n \in V(X)$ and one of the two following situations holds:
\begin{enumerate}
\item the plane link of $X$ is homeomorphic to $[0,1]$ and $(X,0) = (\cup_{i=1}^{n-1} \overline{\gamma_i \gamma_{i +1}},0)$. In this case, we say that $X$ is a \textit{open polygonal surface germ} (or a open $(n-1)$-gonal), and $n$ is denoted as the number of vertices of $X$;
\item the plane link of $X$ is homeomorphic to $\mathbb{S}^1$, $n\ge 3$ and $(X,0) = (\cup_{i=1}^{n} \overline {\gamma_i \gamma_{i+1}},0)$, where $\gamma_{n+1}=\gamma_1$. In this case, we say that $X$ is a \textit{closed polygonal surface germ} (or a closed $n$-gonal), and $n$ is denoted as the number of vertices of $X$.
\end{enumerate}
In both cases, we denote $(X,0)$ as $(\gamma_1, \dots, \gamma_n)$. The surfaces $\overline{\gamma_1\gamma_2}$, $\dots$, $\overline{\gamma_{n-1}\gamma_n}$ are defined as the \textit{edge surfaces of $(X,0)$} (if $ (X,0)$ is a closed polygonal surface germ, $\overline{\gamma_1\gamma_n}$ is also an edge surface of $(X,0)$).
\end{definition}

\begin{definition}\normalfont
\label{poly-surface-2}
Let $(X,0)=(\gamma_1,\dots,\gamma_n)$ be a polygonal surface germ. We say that $(X,0)$ is a \textit{non-degenerate $n$-gonal surface} if $n \ge 3$ and the following conditions are satisfied:
\begin{itemize}
\item $(X,0)$ is a open $n$-gonal and $\angle \gamma_{i-1}\gamma_i \gamma_{i+1} < \pi$, for $i=2,\dots,n-1$ , or $(X,0)$ is a closed $n$-gonal and $\angle \gamma_{i-1}\gamma_i \gamma_{i+1} < \pi$, for $i=1,\dots,n$ ($ \gamma_0 = \gamma_n$, $\gamma_1 = \gamma_{n+1}$);
\item for all $t>0$ small enough and all $1\le i < j < k \le n$, the points $\gamma_i(t),\gamma_j (t),\gamma_k(t)$ are not collinear.
\end{itemize}
If one of these two conditions are not satisfied, we say that $(X,0)$ is a \textit{degenerate polygonal surface germ}.
\end{definition}

\begin{lemma}\label{reduction-lemma1}
Let $(X,0)=(\gamma_1,\dots,\gamma_n)$ be a (open or closed) LNE polygonal surface germ and let $i \in \{1,\dots,n\}$ such that
\begin{itemize}
\item $\overline{\gamma_{i-1}\gamma_i}$, $\overline{\gamma_{i}\gamma_{i+1}}$ are edge surfaces of $X$;
\item $\alpha = tord(\gamma_{i-1},\gamma_i)=tord(\gamma_{i},\gamma_{i+1})$;
\item $0 < \angle \gamma_{i-1} \gamma_i \gamma_{i+1} < \pi$.
\end{itemize}
Then there is $\varepsilon_0 >0$ such that, for all $\varepsilon \in (0,\varepsilon_0)$, if $(\gamma_{\varepsilon},0)$ is the arc defined by
$$\gamma_{\varepsilon} (t) \in \overline{\gamma_{i-1}(t)\gamma_i(t)} \; ; \; \| \gamma_{\varepsilon}(t) - \gamma_i(t) \| = \varepsilon.t^{\alpha} \; (t>0),$$
Then $(X,0)$ is ambient bi-Lipschitz equivalent to the germ of the (open or closed) polygonal surface germ $(\gamma_1, \dots, \gamma_{i-1}, \gamma_{\varepsilon}, \gamma_ {i+1}, \dots, \gamma_n)$.
\end{lemma}

\begin{proof}
Since $\alpha = tord(\gamma_{i},\gamma_{i+1})= tord(\gamma_{i-1},\gamma_{i})$, there are $b_{i},b_{i -1}>0$ such that
$$\| \gamma_i(t) - \gamma_{i+1}(t) \| = b_{i}.t^{\alpha}+o(t^{\alpha}) \; ; \; \| \gamma_{i-1}(t) - \gamma_{i}(t) \| = b_{i-1}.t^{\alpha}+o(t^{\alpha}).$$
Then, for every $t>0$ small enough, $\|\gamma_{i}(t) - \gamma_{i+1}(t)\| > \frac{b_{i}}{2}.t^{\alpha}$ and $\|\gamma_{i-1}(t) - \gamma_{i}(t)\| > \frac{b_{i-1}}{2}.t^{\alpha}$, and since $(X,0)$ is LNE, by Theorem \ref{Edson-Rodrigo} there is $C>1$ such that $X(t)$ is $C$-LNE. Furthermore, as $0 < \angle \gamma_{i-1} \gamma_i \gamma_{i+1} < \pi$, we have that there is $M>1$ such that $\sin{(\theta_i (t))} > \frac{2}{M}$, where $\theta_i (t) = \angle \gamma_{i-1} (t) \gamma_i (t) \gamma_{i+1} (t)$, and also that $\angle \gamma_{\varepsilon} \gamma_{i+1} \gamma_i$ is a decreasing function on $\varepsilon$, with $\angle \gamma_{\varepsilon} \gamma_{i+1} \gamma_i \rightarrow 0$ when $\varepsilon \rightarrow 0$.

Consider $\varepsilon_0 < \frac{b_i}{4C}, \frac{b_{i-1}}{4C}$ such that $\angle \gamma_{\varepsilon_0} (t) \gamma_{i+1} ( t) \gamma_i (t) < \arcsin{\left( \frac{1}{2MC} \right)}$ and $\varepsilon \in (0,\varepsilon_0)$. Given $\theta>0$ small enough, the $\theta$-kneading envelope $V_{\theta}(Y)$ of $Y=\overline{\gamma_{\varepsilon} \gamma_i } \cup \overline{ \gamma_i \gamma_{i+1}}$ is $(\gamma_{\varepsilon}, \gamma_+, \gamma_{i+1}, \gamma_-)$, where $\gamma_+$ is the boundary arc of $V_{\theta}(Y)$ closest to $\gamma_i$. Define $\gamma(t)$ as the intersection of the lines $\overleftrightarrow{\gamma_{\varepsilon}(t) \gamma_i (t)}$ and $\overleftrightarrow{\gamma_{+}(t) \gamma_{i+ 1}(t)}$. Take $\theta$ small enough such that, for every $t>0$, and for any point $p$ in the triangle $\gamma_+ (t) \gamma_{\varepsilon} (t) \gamma (t )$ and any point $q$ in the quadrilateral $\gamma_{\varepsilon} (t) \gamma(t) \gamma_{i+1} (t) \gamma_{-}(t)$,

\begin{itemize}
\item $\|p- \gamma_i(t) \| \le \| \gamma_{\varepsilon} (t) - \gamma_i(t) \| = \varepsilon.t^{\alpha};$
\item $\angle q \gamma_{i+1}(t) \gamma_i(t) < \beta$, where $\beta \in (0,\frac{\pi}{2})$ is such that $ \sin{(\theta_i(t) \pm \beta)} > \frac{1}{M}$ and $\sin{(\beta)} < \frac{1}{MC}$. Such an angle $\beta$ always exists, because $\sin{\theta_i(t)}>\frac{2}{M} > \frac{1}{M}$ and
$$\beta < \angle \gamma_{\varepsilon_0} (t) \gamma_{i+1} (t) \gamma_i (t) + 2\theta<\arcsin{\left( \frac{1}{2MC} \right)} + 2 \theta < \arcsin{\left( \frac{1}{MC} \right)};$$
\item $\| q - r \| \le \| \gamma(t) - \gamma_{\varepsilon} (t) \| = \| \gamma(t) - \gamma_{i} (t) \| + \| \gamma_i(t) - \gamma_{\varepsilon} (t) \| < 2\varepsilon.t^\alpha$, where $r$ is the point on the segment $\overline{\gamma_{i+1}(t) \gamma_i(t)}$ such that $\overline{qr} $ and $\overline {\gamma_{\varepsilon}(t) \gamma_i (t)}$ are parallel.
\end{itemize}

\begin{figure}[h]
\centering
\includegraphics[width=10cm]{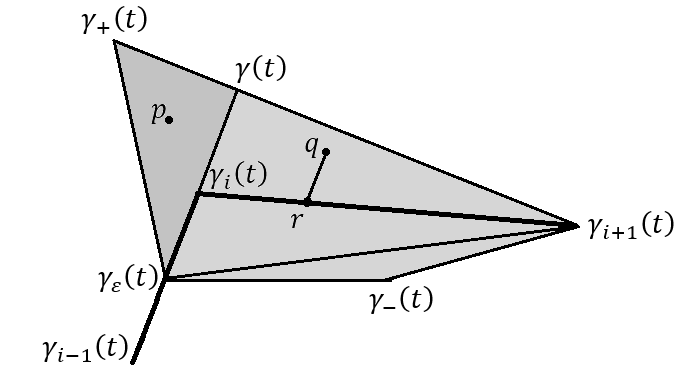}
\label{fig22}
\caption{Proof of the Lemma \ref{reduction-lemma1}.}
\end{figure}

We prove that $V_{\theta}(Y)(t) \cap (X\setminus Y)(t) = \emptyset$. Suppose the contrary. Then, there is at least one point at such intersection. We will divide the analysis in two cases.

\textit{Case 1:} the intersection point $p$ is in triangle $\gamma_+ (t) \gamma_{\varepsilon} (t) \gamma (t)$. In this case, we have

$$\|p- \gamma_i(t) \| \le \varepsilon . t^{\alpha} < \frac{1}{C} \left( \frac{b_i}{2}.t^{\alpha} \right) < \frac{1}{C} \| \gamma_i (t) - \gamma_{i+1}(t) \|,$$
$$\|p- \gamma_i(t) \| \le \varepsilon . t^{\alpha} < \frac{1}{C} \left( \frac{b_{i-1}}{2}.t^{\alpha} \right) < \frac{1}{C} \| \gamma_{i-1} (t) - \gamma_{i}(t) \|.$$
Since $d_{X(t)}(p, \gamma_i(t)) \ge \min \{ \| \gamma_i (t) - \gamma_{i+1}(t) \|, \| \gamma_{i-1} (t) - \gamma_{i}(t) \| \}$, it follows that $d_{X(t)}(p, \gamma_i(t)) > C\|p- \gamma_i(t) \|$, a contradiction.

\textit{Case 2}: the intersection point $q$ is in the quadrilateral $\gamma_{\varepsilon} (t) \gamma(t) \gamma_{i+1} (t) \gamma_{-}(t )$. In this case, we have

$$\|q- r \| \le 2\varepsilon.t^{\alpha} \le \frac{1}{C}.\frac{b_{i-1}}{2}.t^{\alpha} < \frac{1}{ C}\| \gamma_{i-1} (t) - \gamma_{i}(t) \|.$$
For $q$ in the triangle $\gamma_i(t) \gamma(t) \gamma_{i+1}(t)$, we have
$$\angle r q \gamma_{i+1}(t) = \pi - \angle qr \gamma_{i+1}(t) - \angle r \gamma_{i+1}(t) q = $$
$$=\theta_i(t) - \angle r \gamma_{i+1}(t) q \in (\theta_i(t) - \beta, \theta_i(t)).$$
For $q$ in the quadrilateral $\gamma_i(t) \gamma_{\varepsilon}(t) \gamma_{-}(t) \gamma_{i+1}(t)$, we have
$$\angle r q \gamma_{i+1}(t) = \pi - \angle qr \gamma_{i+1}(t) - \angle r \gamma_{i+1}(t) q =$$
$$= (\pi - \theta_i(t))-\angle r \gamma_{i+1}(t) q \in (\pi - \theta_i(t) - \beta, \pi - \theta_i(t )).$$
In any case, $\sin{(\angle r q \gamma_{i+1}(t))}>\frac{1}{M}$. If $h$ is the distance from $r$ to the line $\overleftrightarrow{\gamma_{i+1}(t)q}$, then
$$\|q- r \| = \frac{h}{\sin{\angle r q \gamma_{i+1}(t)}} <M.h =$$
$$= M\| r - \gamma_{i+1}(t)\|\sin{(\angle q\gamma_{i+1}(t) \gamma_i(t))} < \frac{1}{C}\| r - \gamma_{i+1}(t)\|.$$
Since $d_{X(t)}(q, r) \ge \min \{ \| r - \gamma_{i+1}(t) \|, \| \gamma_{i-1} (t) - \gamma_{i}(t) \| \}$, it follows that $d_{X(t)}(q,r) > C\|q- r\|$, a contradiction.

Therefore, from Remark \ref{kneading-envelope-remark}, $(X,0)$ is ambient bi-Lipschitz equivalent to the polygonal surface germ $(\gamma_1, \dots, \gamma_{i-1}, \gamma_{ \varepsilon}, \gamma_{i+1}, \dots, \gamma_n)$.
\end{proof}

\begin{lemma} \label{reduction-lemma2}
Let $(X,0) = (\gamma_1,\dots,\gamma_n)$ be a (open or closed) LNE polygonal surface germ and let $i \in \{1,\dots,n\}$ such that
\begin{itemize}
\item $\overline{\gamma_{i-1}\gamma_i}$, $\overline{\gamma_{i}\gamma_{i+1}}$ are edge surfaces of $X$;
\item $tord(\gamma_{i-1},\gamma_i)>tord(\gamma_{i},\gamma_{i+1})$;
\item $0 < \angle \gamma_{k-1} \gamma_k \gamma_{k+1} < \pi$, for all $k$.
\end{itemize}
Then the following statements are true:
\begin{enumerate}
\item If $X$ is open $n$-gonal and $i=2$, then $(X,0)$ is ambient bi-Lipschitz equivalent to the germ of the open polygonal surface germ ($\gamma_1$, $\gamma_3$, $\dots$, $\gamma_n)$;
\item If $\overline{\gamma_{i-2} \gamma_{i-1}}$ is an edge surface of $(X,0)$ and $tord(\gamma_{i-2}, \gamma_{i-1} ) \le tord(\gamma_{i},\gamma_{i+1})$, then $(X,0)$ is ambient bi-Lipschitz equivalent to the germ of the (open or closed) polygonal surface germ ($\gamma_1 $, $\dots$, $\gamma_{i-1}$, $\gamma_{i+1}$, $\dots$, $\gamma_{n})$.
\end{enumerate}
\end{lemma}

\begin{proof}
Let $\theta_i = \angle \gamma_{i-1} \gamma_i \gamma_{i+1}$ and $C>1$ such that $X(t)$ is $C$-LNE, for $t>0 $ (Theorem \ref{Edson-Rodrigo}). As in the proof of the previous lemma, given $\theta>0$ small enough, the $\theta$-kneading envelope $V_{\theta}(Y)$ of $Y=\overline{\gamma_{\varepsilon} \gamma_i } \cup \overline{\gamma_i \gamma_{i+1}}$ is $(\gamma_{\varepsilon}, \gamma_+, \gamma_{i+1}, \gamma_-)$, where $\gamma_+$ is the boundary arc of $V_{\theta}(Y)$ closest to $\gamma_i$. Let $\gamma(t)$ be the intersection of lines $\overleftrightarrow{\gamma_{i-1} \gamma_i (t)}$ and $\overleftrightarrow{\gamma_{+} \gamma_{i+1} (t) }$. Since $tord(\gamma_{i-1},\gamma_i)>tord(\gamma_{i},\gamma_{i+1})$, we have $\angle \gamma_{i-1} \gamma_{i +1} \gamma_i = 0$ and $\angle \gamma_{i+1} \gamma_{i-1} \gamma_i = \pi - \theta_i$.

We also have $\angle\gamma \gamma_i \gamma_{i+1} = \pi - \theta_i$, $\angle \gamma \gamma_{i+1} \gamma_i = \angle \gamma \gamma_{i-1 } \gamma_{+}=\theta$, $\angle \gamma_i \gamma \gamma_{i+1} = \theta_i - \theta$ and $\angle \gamma_{i-1} \gamma_{+} \gamma = \theta_i - 2\theta$, for $\theta <\frac{\theta_i}{2}$. Hence, if $tord(\gamma_i, \gamma_{i+1}) = \alpha$, then when we consider the arcs $\gamma_i$, $\gamma$, $\gamma_{i+1}$, we obtain
$$tord(\gamma_i, \gamma)=tord(\gamma, \gamma_{i+1})=\alpha \Rightarrow tord(\gamma_{i-1},\gamma)=\alpha,$$
and considering the arcs $\gamma_{i-1}$, $\gamma$, $\gamma_{+}$, we have, from $tord(\gamma_{i-1},\gamma)=\alpha$,
$$tord(\gamma_{i-1}, \gamma_{+})=tord(\gamma_{+}, \gamma)=\alpha.$$
\begin{figure}[h]
\centering
\includegraphics[width=10cm]{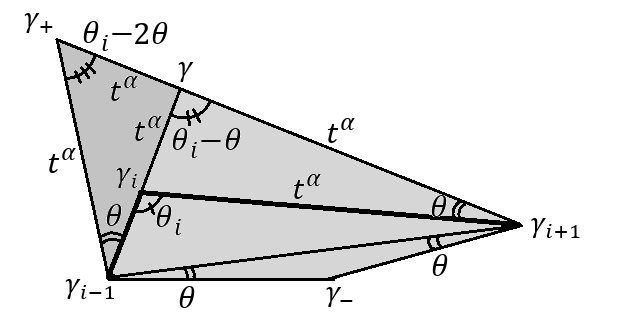}
\label{fig23a}
\caption{Determining orders of tangency in Lemma \ref{reduction-lemma2}.}
\end{figure}

Furthermore, there is $b_{\theta}>0$ such that the Puiseux expansion of $\|\gamma_{i-1}(t)- \gamma(t)\|$ is $b_{\theta}.t ^{\alpha}+o(t^{\alpha})$, with $b_{\theta} \to 0$ when $\theta \to 0$. Thus, for each $t>0$, the greatest distance connecting two points in the triangle $\gamma_{i-1}(t)\gamma_{+}(t)\gamma(t)$ has a Puiseux expansion equal to $ c_{\theta}.t^{\alpha}+o(t^{\alpha})$, with $c_{\theta}>0$ and $c_{\theta} \to 0$ when $\theta \to 0$. In particular, there exists $\theta>0$ such that, for $t>0$ and for every point $p$ in triangle $\gamma_+(t) \gamma_{i-1}(t) \gamma(t)$,
$$\|p-\gamma_i(t)\| < \frac{1}{C\tilde{C}}\|\gamma_i(t)-\gamma_{i+1}(t)\|  \; \; \; (*),$$
where $\tilde{C} \ge 1$ satisfies $\|\gamma_{i-2}(t) - \gamma_{i-1}(t)\| \ge \frac{1}{\tilde{C}}\|\gamma_{i-2}(t) - \gamma_{i-1}(t)\|$, for $t>0$ small ($ \tilde{C}$ exists, because $tord(\gamma_{i-2}, \gamma_{i-1}) \le tord(\gamma_{i},\gamma_{i+1})$). As it was done in the previous lemma, there is $\theta>0$ small enough such that every point $q$ in the quadrilateral $\gamma_{i-1}(t) \gamma(t) \gamma_{i+1} (t)\gamma_{-}(t)$,
$$ \| q - r \| < \frac{1}{C\tilde{C}}\| \gamma_{i+1}(t) - r \| \le \frac{1}{C\tilde{C}}\| \gamma_{i}(t) - \gamma_{i+1}(t) \| \; \; \; (**).$$

\begin{claim}\label{tord-analysis}
If $\overline{\gamma_{i-2}\gamma_{i-1}}$ is an edge surface of $(X,0)$ and $tord(\gamma_{i-2}, \gamma_{i-1}) < tord(\gamma_{i-1},\gamma_i)$, then there is $\theta>0$ small enough such that $\overline{\gamma_{i-2}\gamma_{i-1}}(t) \cap V_{\theta}(Y) = \{ \gamma_{i-1}(t)\}$, for $t>0$.
\end{claim}

\begin{proof}
Suppose the contrary, that is, for all $\theta>0$, there is $t>0$ small enough such that the ray $s(t)=\overrightarrow{\gamma_{i-1} (t) \gamma_ {i-2} (t)}$ intersects the interior of quadrilateral $\gamma_{i-1}(t) \gamma_{+}(t) \gamma_{i+1}(t) \gamma_{-}(t) $. Since $\angle \gamma_{i-2} \gamma_{i-1} \gamma_i >0$, we can take $\theta$ small enough such that $s(t)$ does not intersect the interior of triangle $\gamma_{ i-1}(t), \gamma_+(t), \gamma(t)$. Therefore, $s(t)$ intersect the interior of quadrilateral $\gamma_{i-1}(t), \gamma(t), \gamma_{i+1}(t), \gamma_{-}(t)$ , which implies $\angle \gamma_{i-2} \gamma_{i-1} \gamma_i \le \pi - \theta_i$. We now consider two cases.

\textit{Case 1:} If $\angle \gamma_{i-2} \gamma_{i-1} \gamma_i < \pi - \theta_i$, then $s(t)$ intersects $\overline{\gamma( t) \gamma_{i+1}(t)}$ at a point $\tilde{\gamma}(t)$. Let the arc $\tilde{\gamma} = \tilde{\gamma}(t)$. Because $\angle \gamma_i (t) \tilde{\gamma}(t) \gamma_{i-1}(t) = \pi - \angle \gamma_{i-2}(t) \gamma_{i-1 }(t) \gamma_i(t) - \theta_i >0$, we have
$$\alpha_1 = tord(\gamma_{i-1}, \gamma_i)=tord(\gamma_i, \tilde{\gamma})=tord(\tilde{\gamma},\gamma_{i-1}).$$
Hence, $\tilde{\gamma}(t)$ is the intersection of the segments $\overline{\gamma_{i-2}\gamma_{i-1}}(t)$ and $\overline{\gamma_{i }\gamma_{i+1}}(t)$, and it contradicts the fact that $X(t) \simeq [0,1]$ or $X(t) \simeq \mathbb{S}^1$.

\begin{figure}[h]
\centering
\includegraphics[width=9cm]{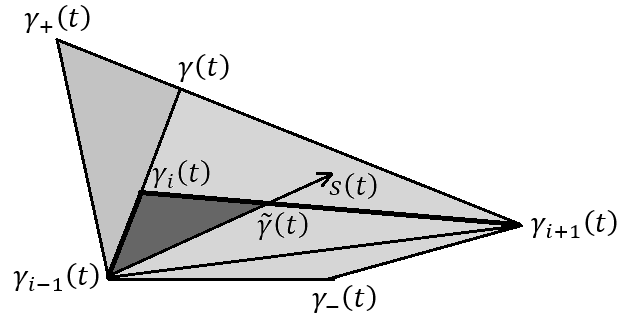}
\label{fig23}
\caption{Proof of Claim \ref{tord-analysis}, Case 1.}
\end{figure}

\textit{Case 2:} If $\angle \gamma_{i-2} \gamma_{i-1} \gamma_i = \pi - \theta_i$, let $\alpha_0 \in \mathbb{Q}_{\ge 1}$ such that $tord(\gamma_{i-2},\gamma_{i-1}),tord(\gamma_i, \gamma_{i+1}) < \alpha_0 < \alpha_1$ and let $\rho_0(t) \in \overline{\gamma_i(t) \gamma_{i+1}(t)}$ such that $\| \gamma_i(t) - \rho_0(t)\|=t^{\alpha_0}$. Consider $\rho_1(t) \in C_a^3(t)$ such that $\gamma_{i-1}(t) \gamma_i(t) \rho_0(t) \rho_1(t)$ is a parallelogram, let $\rho_2(t)$ be the intersection of lines $\overleftrightarrow{\rho_0(t) \rho_1(t)}$ and $\overleftrightarrow{\gamma_{i-2}(t) \gamma_{i-1}( t)}$, and define the arcs $\rho_k = \rho_k(t)$ ($k=0,1,2$). As $\angle \gamma_i \gamma_{i-1} \rho_1 = \pi - \theta_i = \angle \gamma_{i-2} \gamma_{i-1} \gamma_i = \angle \rho_{2} \gamma_ {i-1} \gamma_i$, we have $\angle \rho_2 \gamma_{i-1} \rho_1 = 0$ and $$\{\angle \rho_2 \rho_1 \gamma_{i-1} , \angle \rho_1 \rho_2 \gamma_{i-1}\} = \{ \theta_i, \pi - \theta_i\} \ne \{0,\pi \}.$$

\begin{figure}[h]
\centering
\includegraphics[width=10cm]{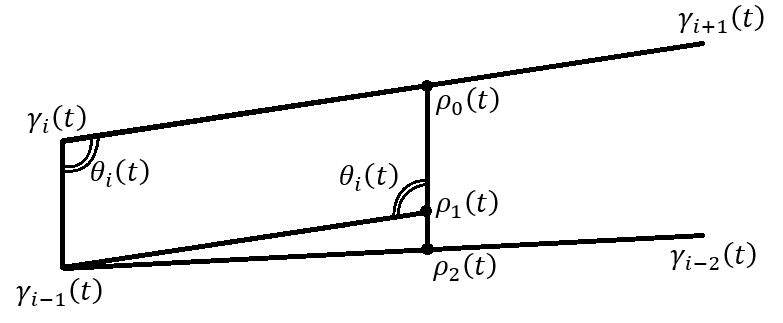}
\label{fig24}
\caption{Proof of Claim \ref{tord-analysis}, Case 2.}
\end{figure}
Thus, $tord(\gamma_{i-1}, \rho_2) = tord(\gamma_{i-1}, \rho_1) = tord(\gamma_{i}, \rho_0) = \alpha_0$, where $d_ {X(t)}(\rho_2(t), \rho_0(t))$ is
$$ \| \rho_2(t) - \gamma_{i-1}(t)\| + \|\gamma_{i-1}(t) - \gamma_i(t) \| + \|\gamma_i(t) - \rho_0(t) \| > \|\gamma_i(t) - \rho_0(t) \| = t^{\alpha_0}. $$
Furthermore, there is $\alpha_1 \in \mathbb{Q}_{\ge 1}$, $\alpha_1 > \alpha_0$, such that $tord(\rho_2, \rho_1) = \alpha_1$. Since $\alpha = tord(\rho_0, \rho_1)$ and $\rho_0(t)$, $\rho_1(t)$, $\rho_2(t)$ are collinear, it follows that there exists $b>0$ such that $\| \rho_2(t) - \rho_0(t)\| = b.t^{\tilde{\alpha}} + o(t^{\tilde{\alpha}})$, where $\tilde{\alpha}=\min\{\alpha,\alpha_1\} > \alpha_0 $. Therefore, there exists $0<t<\left( \frac{1}{2bC}\right)^{\frac{1}{\tilde{\alpha}-\alpha_0}}$ such that $\| \rho_2(t) - \rho_0(t)\| < 2b.t^{\tilde{\alpha}}$, and with this we get
$$\dfrac{\| \rho_2(t) - \rho_0(t)\|}{d_{X(t)}(\rho_2(t), \rho_0(t))}<\dfrac{2b.t^{\tilde{\alpha }}}{t^{\alpha_0}}=2b.t^{\tilde{\alpha}-\alpha_0}<\frac{1}{C},$$
a contradiction, since $X(t)$ is $C$-LNE. Then the claim is proved.
\end{proof}

Back to the proof of the Lemma, take $\theta>0$ satisfying $(*)$, $(**)$ and Claim \ref{tord-analysis}, in case the hypothesis $(2)$ is assumed. Let us show that $V_{\theta}(Y)(t) \cap (X\setminus Y)(t) = \emptyset$. Assume the opposite and divide the analysis into two cases.

\textit{Case 1:} the intersection point is $p$ in the triangle $\gamma_+ (t) \gamma_{i-1} (t) \gamma (t)$. If hypothesis $(1)$ is assumed, then
$$d_{X(t)}(p, \gamma_i(t)) \ge \|\gamma_{i}(t) - \gamma_{i+1}(t)\| \ge \frac{1}{\tilde{C}}\| \gamma_{i} (t) - \gamma_{i+1}(t) \|,$$
and if hypothesis $(2)$ is assumed, then
$$d_{X(t)}(p, \gamma_i(t)) \ge \min \{ \|\gamma_{i}(t) - \gamma_{i+1}(t)\|, \| \gamma_{i-2}(t) - \gamma_{i-1}(t)\| \} \ge \frac{1}{\tilde{C}}\| \gamma_{i} (t) - \gamma_{i+1}(t) \|.$$
In each case, we have
$$\|p- \gamma_i(t) \| < \frac{1}{C\tilde{C}}\| \gamma_{i} (t) - \gamma_{i+1}(t) \| \le \frac{1}{C}. d_{X(t)}(p, \gamma_i(t)),$$
a contradiction.

\begin{figure}[h]
\centering
\includegraphics[width=10cm]{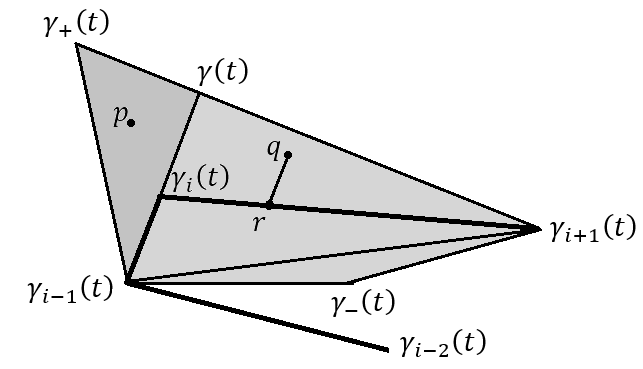}
\label{fig25}
\caption{Proof of Lemma \ref{reduction-lemma2}.}
\end{figure}

\textit{Case 2:} the intersection point is $q$ in quadrilateral $\gamma_{i-1} (t) \gamma(t) \gamma_{i+1} (t) \gamma_{-}(t )$. If hypothesis $(1)$ is assumed, then $d_{X(t)}(q, r) \ge \|\gamma_{i+1}(t) - r\| > C \| q-r \|$, and if the hypothesis $(2)$ is assumed, then $\|\gamma_{i-2}(t) - \gamma_{i-1}(t)\| \ge \frac{1}{\tilde{C}}\|\gamma_{i}(t) - \gamma_{i+1}(t)\| \ge \frac{1}{\tilde{C}}\|\gamma_{i+1}(t) - r\|$.
In any case, we obtain
$$d_{X(t)}(q, r) \ge \min \{ \|\gamma_{i+1}(t) - r\|, \|\gamma_{i-2}(t) - \gamma_{i-1}(t)\| \} \ge \frac{1}{\tilde{C}}\|\gamma_{i+1}(t) - r\|.$$
Thus, $\|q- r \| < \frac{1}{C\tilde{C}}\| \gamma_{i+1}(t) - r \| \le \frac{1}{C}d_{X(t)}(q, r)$, a contradiction. The result follows Remark \ref{kneading-envelope-remark}.
\end{proof}

\begin{lemma}\label{convert-non-degenerate}
Let $n \in \mathbb{N}_{\ge 3}$ and let $(X,0) = (\gamma_1 , \dots , \gamma_n)$ be a LNE, degenerate polygonal surface germ. Suppose there is $\alpha \in \mathbb{Q}_{\ge 1}$ such that all edge surfaces of $(X,0)$ are $\alpha$-H\"older triangles. If $0 <\angle \gamma_{i-1} \gamma_i \gamma_{i+1} < \pi$, for every $i$, then there is a non-degenerate polygonal surface germ $(\tilde{X},0)=(\gamma_1, \tilde \gamma_2, \dots, \tilde \gamma_{n-1}, \gamma_n)$ such that $(X,0)$ is ambient bi-Lipschitz equivalent to $(\tilde{X},0)$.
\end{lemma}

\begin{proof}
Perform the following algorithm successively for $i=2, \dots, n-1$.

\textit{Step 1:} choose $\varepsilon_1 >0$ small enough such that $\overrightarrow{\gamma_{\varepsilon_1} \gamma_{i+1}} \ne \pm \overrightarrow{\gamma_j \gamma_{k}}$ , for all distinct indices $j,k \in \{1,\dots,n\}$, and set $\gamma = \gamma_{\varepsilon_1}$.

\textit{Step 2:} choose $\varepsilon_2 >0$ small enough such that $\overrightarrow{\tilde{\gamma} \gamma_{i-1}} \ne \pm \overrightarrow{\gamma_j \gamma_{k}}$, for all distinct indices $j,k$ in $\{1,\dots,n\}$, and set $\tilde{\gamma} = \gamma_{\varepsilon_2}$. 

\textit{Step 3:} Label $\tilde{\gamma}$ as $\gamma_i$. If $i<n-1$, bo back to step 1, changing $i$ to $i+1$. If $i=n-1$, stop the algorithm.

\begin{figure}[h]
	\centering
	\includegraphics[width=14cm]{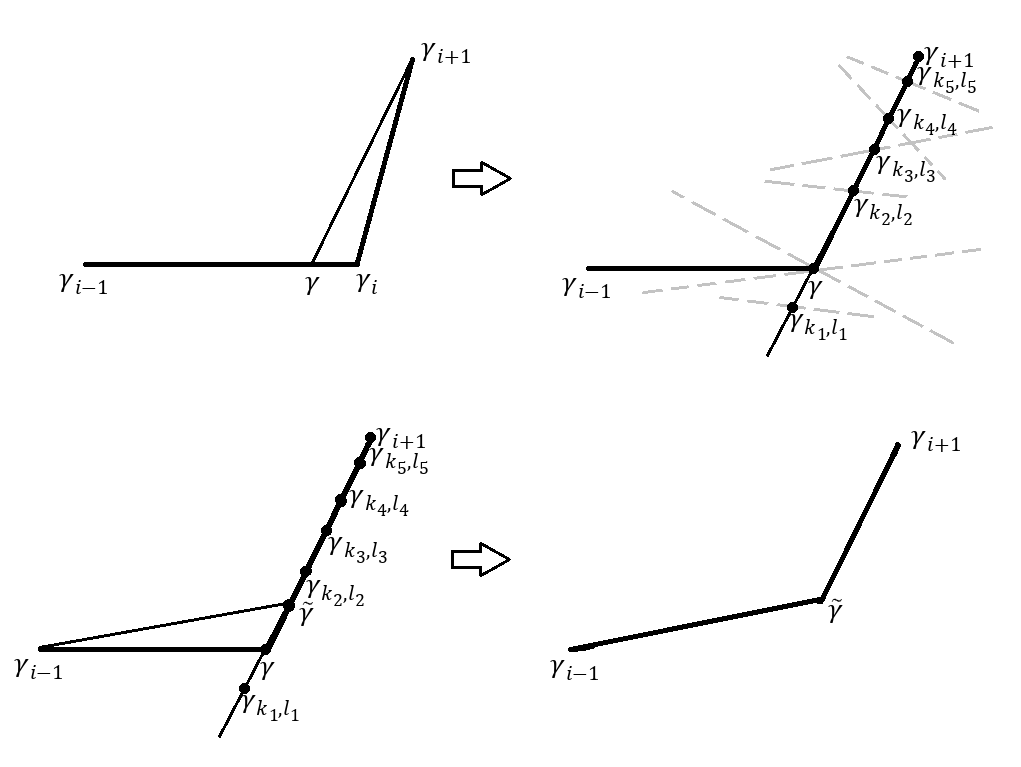}
	\label{fig29}
	\caption{transforming $(X,0)$ into a non-degenerate polygonal surface.}
\end{figure}

Step 1 is possible to perform, since $\pi > \angle \gamma_{i-1} \gamma_{i} \gamma_{i+1} >0$. By Lemma \ref{reduction-lemma1} there is $\varepsilon>0$ such that, for every $\varepsilon_1\in (0,\varepsilon)$, if $\gamma_{\varepsilon_1}(t) \in \overline{\gamma_{i-1}(t) \gamma_{i}(t) }$ is such that $\| \gamma_{\varepsilon_1}(t) - \gamma_i(t)\| = \varepsilon_1 . t^{\alpha}$ and $\gamma_{\varepsilon_1} = \gamma_{\varepsilon_1}(t)$. Then, $(X,0)$ is ambient bi-Lipschitz equivalent to the germ of the LNE polygonal surface germ $(\gamma_1,\dots, \gamma_{i-1},\gamma_{\varepsilon_1},\gamma_{i+1},\dots, \gamma_n)$. 

To see why Step 2 is possible to perform, for each $t>0$ and for all distinct indices $j,k$ in $\{1,\dots,n\}$, let $\gamma_{k, l}(t)$ the intersection of line $\overleftrightarrow{\gamma_k (t) \gamma_l(t)}$ with line $\overleftrightarrow{\gamma (t) \gamma_{i+1}(t)}$ . Since $\overrightarrow{\gamma \gamma_{i+1}} \ne \pm \overrightarrow{\gamma_j \gamma_{k}}$, such intersections always exists (it is also possible to have $\gamma_{k,l}(t) = \gamma(t)$). Let the arcs $\gamma_{j,k} =\gamma_{j,k}(t)$. For indices $j,k$ such that $(\gamma_{j,k},0) \ne (\gamma,0)$, there is $a_{j,k} >0$ and $\alpha_{j,k} \in \mathbb{Q}_{\ge 1}$ such that $\| \gamma(t) - \gamma_{j,k}(t)\| = a_{j,k}.t^{\alpha_{j,k}}+o(t^{\alpha_{j,k}})$. Since $\pi > \angle \gamma_{i-1} \gamma \gamma_{i+1} >0$, by Lemma \ref{reduction-lemma1} there is $\varepsilon>0$ such that, for every $\varepsilon_2\in (0,\varepsilon)$, if $\gamma_{\varepsilon_2}(t) \in \overline{\gamma(t) \gamma_{i+1}(t)}$ is such that $\| \gamma_{\varepsilon_2}(t) - \gamma(t)\| = \varepsilon_2 . t^{\alpha}$ and $\gamma_{\varepsilon_2} = \gamma_{\varepsilon_2}(t)$. Then $(X,0)$ is ambient bi-Lipschitz equivalent to the germ of the LNE polygonal surface germ $(\gamma_1,\dots, \gamma_{i-1},\gamma_{\varepsilon_2},\gamma_{i+1},\dots, \gamma_n)$. 

Finally, notice that when the algorithm is applied to $i$, $\gamma_i$ is transformed into an arc where $\gamma_i(t)$ does not belong to any line connecting $\gamma_j(t)$ to $\gamma_k(t)$, with $j,k \ne i$ being distinct indices in $\{1,\dots, n\}$. Since the arcs $\gamma_1, \gamma_n$ remained fixed, $(X,0)$ is ambient bi-Lipschitz equivalent to the germ of a non-degenerate polygonal surface $(\tilde{X},0)=(\gamma_1, \tilde \gamma_2, \dots, \tilde \gamma_{n-1}, \gamma_n)$, which shares the arcs $\gamma_1, \gamma_n$ with $(X,0)$.
\end{proof}

\section{Edge Reductions on Polygonal Surfaces}

Now we show that if the link of a LNE surface in $\R^3$ is a segment then the germ is ambient bi-Lipschitz equivalent to a H\"older triangle, and if the link is homeomorphic to a circle, then the corresponding reduction  procedure makes the link to be triangular, showing that it is ambient bi-Lipschitz equivalent to a either a horn or a cone. In other words, this section is devoted to show that every LNE surface germ with isolated singularity and connected link is ambient Lipschitz trivial.

\begin{lemma}\label{reduction-collinear}
Let $(X,0) = (\gamma_1 , \dots , \gamma_n)$ be a LNE, degenerated polygonal surface germ and consider $\gamma_{n+1} = \gamma_1$ when $X$ is a closed polygonal surface. If $\angle \gamma_{i-1} \gamma_i \gamma_{i+1} = \pi$, for some $i$, then $(X,0)$ is ambient bi-Lipschitz equivalent to $(\gamma_1 , \dots, \gamma_{i-1}, \gamma_{i+1}, \dots, \gamma_n)$.
\end{lemma}

\begin{proof}
Since $\angle \gamma_{i-1} \gamma_i \gamma_{i+1} = \pi$, we have $\angle \gamma_i \gamma_{i-1} \gamma_{i+1} = \angle \gamma_i \gamma_{i+1} \gamma_{i-1} = 0$, so there is $\theta>0$ small enough such that the $\theta$-kneading envelope $V_{\delta}(Y)$ of $(Y,0)=(\overline{\gamma_{i-1}\gamma_i}\cup \overline{\gamma_i\gamma_{i+1}},0)$ does not intersect $(X \setminus Y,0)$, by the same arguments of Proposition \ref{linear-decomposition} (by a rotation of axes, if necessary, $(Y,0)$ can be seen as a $\delta$-bounded convex synchronized triangle germ, for each $\delta>0$). Therefore, $(Y,0)$ is kneadable in $V_{\delta}(Y_1)$ and hence $(X,0)$ is ambient bi-Lipschitz equivalent to the germ of $(\gamma_1 , \dots, \gamma_{i-1}, \gamma_{i+1}, \dots, \gamma_n)$.

\begin{figure}[h]
\centering
\includegraphics[width=15cm]{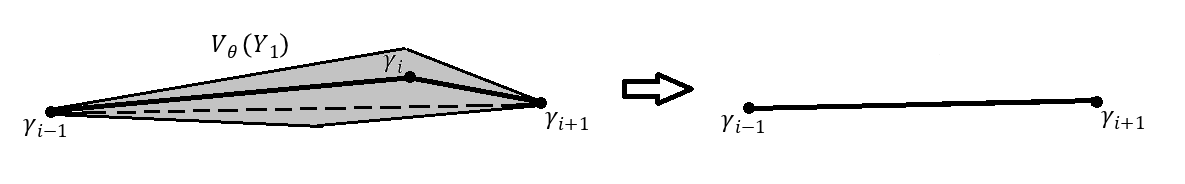}
\label{fig18}
\caption{Proof of Lemma \ref{reduction-collinear}.}
\end{figure}
\end{proof}

\begin{lemma} \label{triangulacao}
Let $n \in \mathbb{N}_{\ge 3}$ and let $P_1, P_2, \dots, P_n$ points in the plane, with any three of then non collinear, such that $X$ is either a simple, open polygonal line with edges $\overline{P_1P_2}, \dots, \overline{P_{n-1}P_n}$, or a simple, closed polygonal line with edges $\overline{P_1P_2}, \dots, \overline{P_{n-1}P_n}, \overline{P_nP_1}$. Then there is a set $I$ of pairs of indices $( i,j )$, with $1\le i < j \le n$, such that:
\begin{itemize}
\item $\cup_{\{i,j\} \in I} \overline{P_i P_j}$ are all the edges of a triangulation of the convex hull of $X$;
\item $(1,2),\dots,(n-1,n) \in I$, if $X$ is a open polygonal line;
\item $(1,2),\dots,(n-1,n), (1,n) \in I$, if $X$ is a closed polygonal line.
\end{itemize}
\end{lemma}

\begin{proof}
We will prove the result by induction on the number of edges of $X$. The base case with 2 edges is trivial. For the inductive step, let $\tilde{X}$ be the open polygonal line with edges $\overline{P_1P_2}$, $\dots$, $\overline{P_{n-2}P_{n-1}}$ if $X$ is a open polygonal line, and let $\tilde{X}$ be the open polygonal line with edges $\overline{P_1P_2}$, $\dots$, $\overline{P_{n-2}P_{n-1}}$, $\overline{P_nP_1}$ if $X$ is a closed polygonal line. By hypothesis, there is a set $\tilde{I}$ of pairs of indices $( i,j ); 1\le i < j \le n$ satisfying the lemma conditions for the open, simple polygonal $\tilde{X}$. Let $Y$ be the convex hull of $\tilde{X}$, and suppose $Y$ has vertices $P_{i_1}, \dots ,P_{i_t}$ in counterclockwise order, with $1\le i_i < \dots < i_t \le n$.

Let $\ell = \overline{P_{n-1}P_n}$. Since no three vertices of $X$ are collinear, if there is $(i,j) \in \tilde{I}$ such that $\ell$ intersect $\overline{P_i P_j}$ in at least two points, then $\ell=\overline{P_i P_j}$ and taking $I=\tilde{I}$, the result follows by induction. Suppose now that, for every $(i,j)\in I$, $\ell$ intersects $\overline{P_i P_j}$ in at at most one point. We have two cases to consider.

\textit{Case 1:} $\overline{P_{n-1}P_n}$ don't have points in the interior of $Y$.

In this case, by cyclically relabeling $P_{i_1}$, \dots, $P_{i_t}$ if necessary, there is $1\le a <b\le  t$ such that $\overline{P_{i_a} P_n}$ and $\overline{P_{i_b} P_n}$ are edges of the convex hull of $X$. Therefore, the set $I = \tilde{I} \cup \{(i_a, n), (i_{a}+1,n), \dots, (i_b, n)\}$ satisfies the lemma conditions for $X$ and the result follows by induction.

\textit{Case 2:} $\overline{P_{n-1}P_n}$ have at least one point in the interior of $Y$.

In this case, $\ell$ intersect at least one triangle of the triangulation of $Y$. If $\ell$ don't intersect any edge of the triangulation, then $\ell$ is inside some triangle with vertices $P_{n-1}, P_a, P_n$. Therefore, the set $I = \tilde{I} \cup \{ (a, n), (b, n) \}$ satisfies the lemma conditions to $X$ and the result follows by induction.

If $\ell$ intersect at least one edge of the triangulation, let $\ell_1, \dots, \ell_s$ be such edges and $X_i =\ell_i \cap \ell$ ($i=1,\dots, s$). Without loss of generality, suppose $P_{n-1}, X_1, \dots, X_s, P_n$ are in this order on $\ell$. Let also $T_1, \dots, T_s$ be triangles of the triangulation of $Y$, with $T_1$ having vertex at $P_{n-1}$ and edge $\ell_1$ and, for $i=2, \dots, s$, $T_i$ having edges $\ell_{i-1}, \ell_i$. We have two subcases to consider:

\textit{Case 2.1:} If $P_n$ is in the interior of $Y$, let $T$ be triangle with edge $\ell_s$ that contain $P_n$ and suppose the vertices of $T$ are $P_u, P_v, P_w$, with $\ell_s = \overline{P_u P_v}$. Since $\ell_1 ,\dots, \ell_s$ intersect $\overline{P_{n-1}P_n}$, they are not edges of $X$. So, removing $\ell_1, \dots, \ell_s$ from the triangulation of $Y$ and adding $\overline{P_{n-1}P_n}, \overline{P_u P_n}, \overline{P_v P_n}, \overline{P_w P_n}$, we decompose $Y$ into several triangles and two polygons that can be triangulated by hypothesis. Joining this triangulation with the remaining triangles of the triangulation of $Y$, the result follows by induction.

\begin{figure}[h]
	\centering
	\includegraphics[width=15cm]{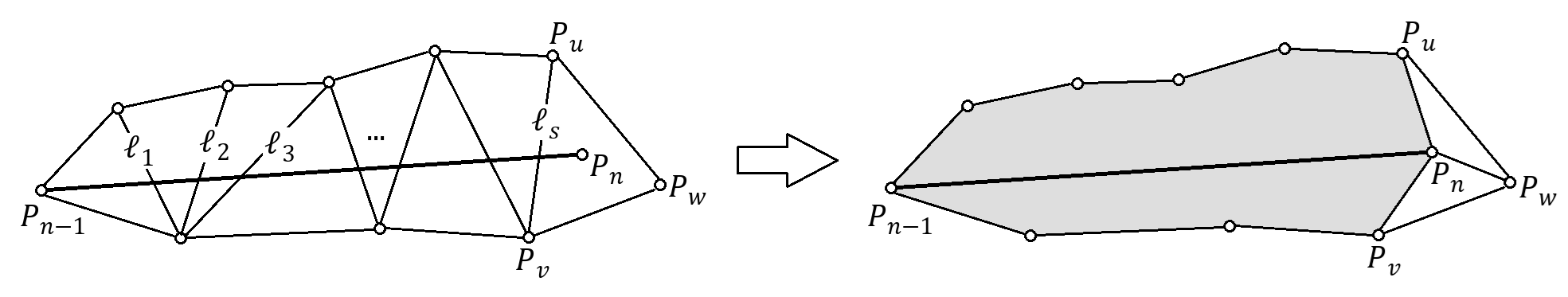}
	\label{tri1}
	\caption{Proof of Case 2.1 of Lemma \ref{triangulacao}.}
\end{figure}

\textit{Case 2.2:} If $P_n$ is not in the interior of $Y$, there is $1\le a <b\le  t$ such that $\overline{P_{i_a} P_n}$ and $\overline{P_{i_b} P_n}$ are edges of the convex hull of $X$ (cyclically relabeling $P_{i_1}$, \dots, $P_{i_t}$ if necessary). Since $\ell_1 ,\dots, \ell_s$ intersect $\overline{P_{n-1}P_n}$, they are not edges of $X$. So, removing $\ell_1, \dots, \ell_s$ from the triangulation of $Y$ and adding $\overline{P_{n-1}P_n}, \overline{P_{i_a} P_n}, \dots, \overline{P_{i_b} P_n}$, we decompose the convex hull of $X$ into several triangles and two polygons that can be triangulated by hypothesis. Joining this triangulation with the remaining triangles of the triangulation of $Y$, the result follows by induction.

\begin{figure}[h]
	\centering
	\includegraphics[width=15cm]{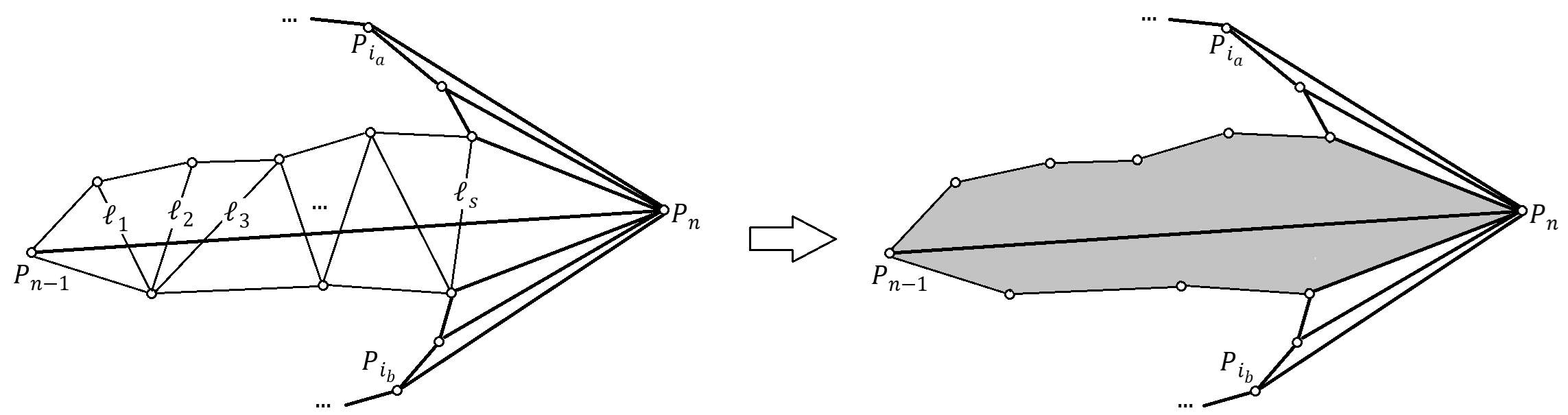}
	\label{tri1}
	\caption{Proof of Case 2.2 of Lemma \ref{triangulacao}.}
\end{figure}
\end{proof}

\begin{lemma}\label{reduction-triangulation}
Let $n\in \mathbb{N}_{\ge 3}$ and $(X,0) = (\gamma_1 , \dots , \gamma_n)$ a LNE, non degenerated polygonal surface and suppose there is $\alpha \in \mathbb{Q}_{\ge 1}$ such that all surface edges germs of $(X,0)$ are $\alpha$-H\"older triangles. Suppose $0 <\angle \gamma_{i-1} \gamma_i \gamma_{i+1} < \pi$, for each $i\in \{1,\dots,n\}$.
\begin{itemize}
\item If $X$ is a open polygonal surface, then $(X,0)$ is ambient bi-Lipschitz equivalent to $(\overline{\gamma_1 \gamma_n},0)$.
\item If $X$ is a closed polygonal surface, then $(X,0)$ is ambient bi-Lipschitz equivalent to a LNE, closed $3$-gonal surface germ.
\end{itemize}
\end{lemma}

\begin{proof}
We will prove the result by induction on $n$ (in what follows, consider $\gamma_0=\gamma_n$ and $\gamma_{1} = \gamma_{n+1}$, if $(X,0)$ is a closed polygonal surface). The base case $n=3$ was already been studied in Proposition \ref{2linear-knead-to-1linear}. For the inductive step, consider a set $I$ of index pairs $( i,j ); 1\le i < j \le n$ such that
\begin{enumerate}
\item $(1,2),\dots,(n-1,n) \in I$, if $(X,0)$ is a open polygonal surface germ, and $(1,2),\dots,(n-1,n), (1,n) \in I$, if $X$ is a closed polygonal surface germ;
\item for every $t>0$ small enough, $\cup_{\{i,j\} \in I} \overline{\gamma_i(t)\gamma_j(t)}$ is a triangulation of the convex hull of the set $\{ \gamma_1(t), \dots, \gamma_n(t) \}$.
\end{enumerate}
Since $(X,0)$ is a non-degenerate polygonal surface germ and the topological type of $X(t)$ is invariant for small $t$ (Theorem \ref{conical}), $I$ exists, by Lemma \ref{triangulacao}.

\begin{figure}[h]
\centering
\includegraphics[width=13cm]{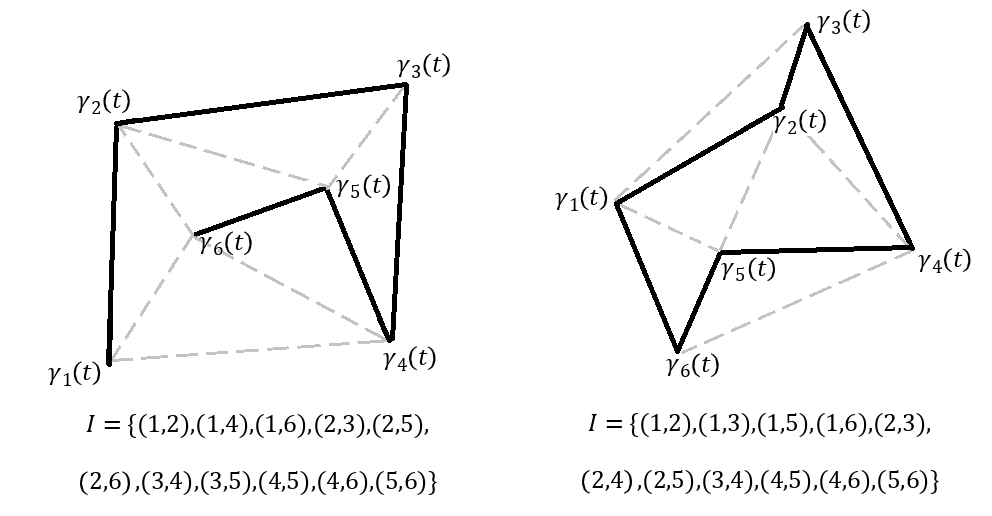}
\label{fig19}
\caption{Examples of polygonal surface link triangulations and their respective index pair sets $I$.}
\end{figure}

Consider the edges $\overline{\gamma_i \gamma_{i+1}}(t)$ in the link $X(t)$ as highlighted and suppose that the convex hull of $\{ \gamma_1(t), \dots, \gamma_n(t) \}$ is a $k$-gon, where $a$ of its $k$ edges are highlighted. For example, in the triangulation of the left diagram in Figure 33, the highlighted edges are $\overline{\gamma_1(t)\gamma_2(t)}$, $\dots$, $\overline{\gamma_5(t)\gamma_6( t)}$ and
$$I=\{ (1,2),(1,4),(1,6),(2,3),(2,5),(2,6),(3,4),(3 ,5),(4,5),(4,6),(5,6)\}.$$
In the triangulation of the right diagram in Figure 33, the highlighted edges are $\overline{\gamma_1(t)\gamma_2(t)}$, $\dots$, $\overline{\gamma_5(t)\gamma_6(t) }$, $\overline{\gamma_6(t)\gamma_1(t)}$ and
$$I=\{ (1,2),(1,3),(1,5),(1,6),(2,3),(2,4),(2,5),(3 ,4),(4,5),(4,6),(5,6)\}.$$
Let $T$ be the number of triangles inside this convex hull determined by the triangulation. Calculating the sum of the internal angles of all these triangles, we have $T . \pi$ in total. On the other hand, when counting this total looking at the angles in the $k$-gon of the convex hull and the angles around each interior point, we have a total of $(k-2).\pi + (n-k).2\pi =(2n-k-2) .\pi$. Equalizing these quantities, we have that the total number of triangles is $T=2n-k-2$.

\begin{claim}
If $n\ge 4$, then among those $2n-k-2$ triangles, there is one of then with exactly 2 highlighted edges.
\end{claim}

\begin{proof}
Suppose the opposite, that is, each triangle has at most one highlighted edge (the case where a triangle can have 3 highlighted edges is possible only when $n=3$), and let us first analyze the case $X(t) \simeq [0,1]$. Counting the number of highlighted edges in each triangle, we have at most $T=2n-k-2$ highlighted edges. On the other hand, as there are $a$ highlighted edges on the boundary of the convex hull, there are $n-1-a$ highlighted edges on the interior of the convex hull. Therefore, as each edge in the interior was counted twice and each edge in the boundary was counted once, the total of the highlighted edges counted in each triangle is equal to $2(n-1-a)+a=2n-a-2$. Since $a<k$, we have $2n-a-2>2n-k-2$, contradiction.

In the case $X(t) \simeq \mathbb{S}^1$, the maximum number of highlighted edges is still $2n-k-2$, but counting in a similar way the highlighted edges in the boundary ($a$ edges) and the interior of the convex hull ($n-a$ edges), this total is equal to $2(n-a)+a=2n-a$. Since $a \le k$, we have $2n-a > 2n-k-2$, again a contradiction.
\end{proof}

Now consider such a triangle and suppose its highlighted edges are $\overline{\gamma_{i-1}(t)\gamma_{i}(t)}, \overline{\gamma_{i}(t)\gamma_{i+1}(t)}$. If $2\le i \le n-2$, by Lemma \ref{reduction-lemma1} there is $\varepsilon>0$ such that, if $\gamma_{\varepsilon}(t) \in \overline{\gamma_{i}(t) \gamma_{i+1}(t) }$ is such that $\| \gamma_{\varepsilon}(t) - \gamma_{i+1}(t)\| = \varepsilon . t^{\alpha}$ and $\gamma_{\varepsilon} = \gamma_{\varepsilon}(t)$, then $(X,0)$ is ambient bi-Lipschitz equivalent to the germ of the LNE (open or closed) polygonal surface $(\gamma_1,\dots, \gamma_{i},\gamma_{\varepsilon},\gamma_{i+2},\dots, \gamma_n)$. Since $\angle \gamma_{i+1} \gamma_{i-1} \gamma_{\varepsilon} > 0$, there exists a $\theta$-kneading envelope $V_{\theta}(Y)$ of $ Y=\overline{\gamma_{i-1}\gamma_i}\cup \overline{\gamma_i \gamma_{\varepsilon}}$, such that $V_{\theta}(Y)$ does not intersect any point of $( X \setminus Y,0)$, as the triangle $\gamma_{i-1} \gamma_{\varepsilon} \gamma_{i+1} (t)$ has no points of $X(t)$ in its interior.

\begin{figure}[h]
	\centering
	\includegraphics[width=15cm]{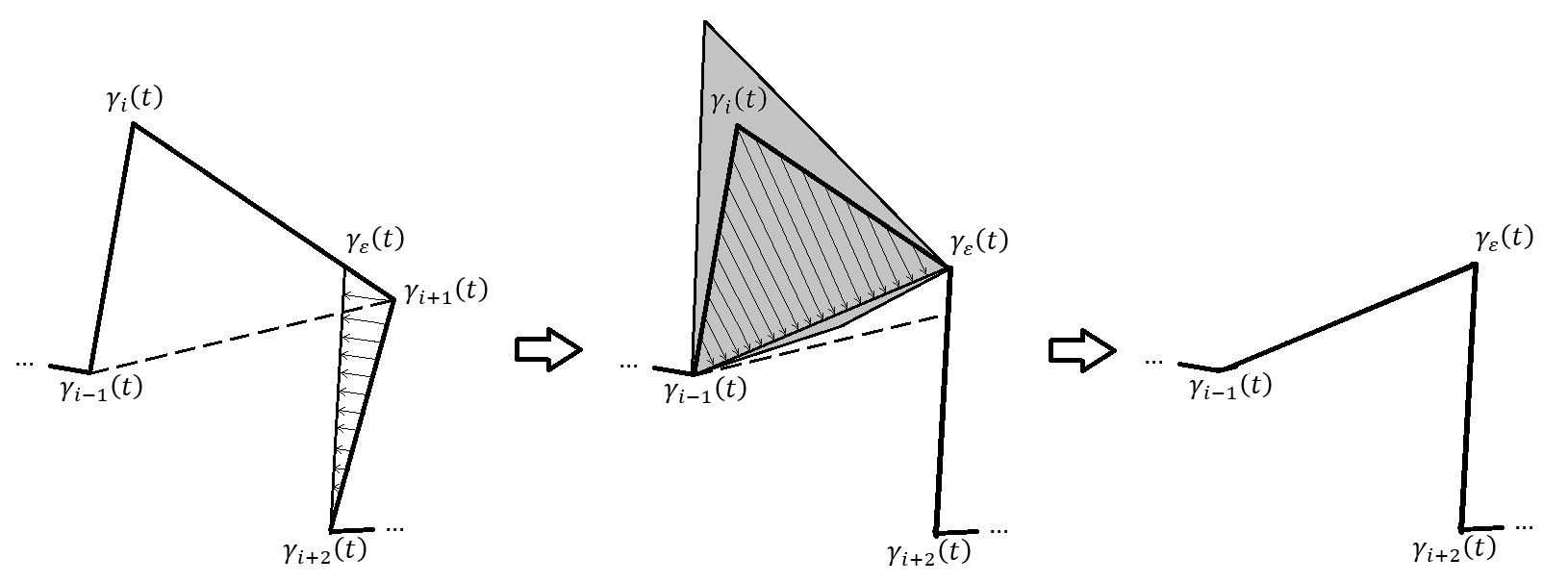}
	\label{fig26}
	\caption{Proof of Lemma \ref{reduction-triangulation}.}
\end{figure}

Therefore, $(Y,0)$ is kneadable in $V_{\theta}(Y)$ and $(X,0)$ is ambient bi-Lipschitz equivalent to the germ of the polygonal surface $(\gamma_1, \dots , \gamma_{i-1}, \gamma_{\varepsilon}, \gamma_{i+2}, \dots, \gamma_n)$, and the result follows by induction. This argument also works for every $i$ in the case $X(t) \simeq \mathbb{S}^1$, by renaming $\gamma_1,\dots,\gamma_n$ as $\gamma_{k+1},\dots,\gamma_{k+n}$, respectively, for some $k$ (where the indices are taken modulo $n$), and in the case $X(t) \simeq [0,1]$, this argument also works for $i=n-1,n$, by renaming $\gamma_1,\dots,\gamma_n$ as $\gamma_{n},\dots,\gamma_{1}$, respectively.
\end{proof}

\begin{proposition}\label{edge-reduction}
Let $a>0$, $n \in \mathbb{N}_{\ge 3}$ and let $(\gamma_1, \dots, \gamma_n)=(X,0) \subset (C_{a}^{3},0)$ be a LNE, polygonal surface germ, degenerated or not.
\begin{enumerate}
\item If $(X,0)$ is a open $(n-1)$-gonal, then $(X,0)$ is ambient bi-Lipschitz equivalent to $(\gamma_1 , \gamma_n)$.
\item If $(X,0)$ is a closed $n$-gonal, then $(X,0)$ is ambient bi-Lipschitz equivalent to a germ of a closed $3$-gonal LNE surface germ.
\end{enumerate}
\end{proposition}

\begin{proof}
The proof will be again by induction on $n$. In what follows, consider $\gamma_0=\gamma_n$, and $\gamma_{1} = \gamma_{n+1}$ if $(X,0)$ is a closed polygonal surface germ. The initial case $n=3$ has already been studied in Proposition \ref{2linear-knead-to-1linear}. For the inductive step, let's separate in several cases.

\textit{Case 1:} $\alpha_i = \alpha_j$, for all $i \ne j$. Let $\alpha$ be the common value of $\alpha_i$ and let us separate this case into two subcases.

\textit{Case 1.1:} if $(X,0)$ is a non-degenerate polygonal surface germ, the result follows by Lemma \ref{reduction-triangulation}.

\textit{Case 1.2:} if $(X,0)$ is a degenerate polygonal surface germ, then if $\angle \gamma_{i-1} \gamma_i \gamma_{i+1} = \pi$, for some $i$, then $(X,0)$ is ambient bi-Lipschitz equivalent to $(\gamma_1 , \dots, \gamma_{i-1}, \gamma_{i+1}, \dots, \gamma_n)$, by Lemma \ref{reduction-collinear}, and the result follows by induction. If $0 <\angle \gamma_{i-1} \gamma_i \gamma_{i+1} < \pi$, for every $i$, then by Lemma \ref{convert-non-degenerate} there is a non-degenerate polygonal surface $\tilde{X}$ with $n$ vertices such that $(X,0)$ is ambient bi-Lipschitz equivalent to $(\tilde{X},0)$, which shares the arcs $(\gamma_1,0), (\gamma_n,0)$ with $(X,0)$, reducing the proof to Case 1.1.

\textit{Case 2:} there are $i,j$ such that $\alpha_i \ne \alpha_j$. We will divide this case in more subcases.

\textit{Case 2.1:} $X(t) \simeq [0,1]$ and $\alpha_1 > \alpha_2$. By Lemma \ref{reduction-lemma2}, item $(1)$, $(X,0)$ is ambient bi-Lipschitz equivalent to the germ of the polygonal surface $(\gamma_1, \gamma_3, \dots, \gamma_n)$ , and the result follows by induction.

\textit{Case 2.2:} $X(t) \simeq [0,1]$ and $\alpha_{n-1} > \alpha_{n-2}$. This case is analogous to Case 2.1, by just renaming $\gamma_1,\dots,\gamma_n$ to $\gamma_n,\dots,\gamma_1$, respectively.

\textit{Case 2.3:} $X(t) \simeq [0,1]$ and $\alpha_{n-1} \le \alpha_{n-2}$, $\alpha_1 \le \alpha_2$. Let $\alpha = \max\{\alpha_1,\dots,\alpha_{n-1}\}$ and consider indices $1\le k_1 \le k_2 \le n-1$ such that $\alpha_{k_1} = \dots = \alpha_{k_2} = \alpha$, $\alpha_{k_1-1}<\alpha$, if $k_1 >1$, and $\alpha_{k_2+1}<\alpha$, if $ k_2 <n-1$. Let $C>1$ such that $X(t)$ is $C$-LNE, for every small $t>0$ (Theorem \ref{Edson-Rodrigo}). As $\alpha = tord(\gamma_{k_1},\gamma_{k_1+1})=\dots=tord(\gamma_{k_2 }, \gamma_{k_2+1})$, there is $a_{\alpha} >0$ such that $d_{X(t)}(\gamma_{k_1}(t), \gamma_{k_2+1}(t)) = a_{\alpha}.t^{\alpha} + o( t^{\alpha})$. Now we split this subcase into three subcases.

\textit{Case 2.3.1:} $k_1 = 1$. In this case, since $\alpha = \alpha_1 \le \alpha_2$ and there is $\alpha_i \ne \alpha_j$, we have $2\le k_2 \le n-2$. For $t>0$, let $\gamma(t) \in \overline{\gamma_{k_2+1}(t) \gamma_{k_2 +2}(t)}$ such that $\| \gamma_{k_2+1}(t) - \gamma(t) \|=a_{\alpha}.t^{\alpha}$. Notice that $\gamma(t)$ is well defined for all $t>0$ small enough, since $\alpha_{k_2 +1} < \alpha_{k_2}$. Let also $\gamma = \gamma(t)$ and let $(Y,0)$ be the open polygonal surface germ $(\gamma_1, \dots, \gamma_{k_2+1}, \gamma)$. By Case 1, $(Y,0)$ is ambient bi-Lipschitz equivalent to $\overline{\gamma_1, \gamma}$. Furthermore, there is $K>a_{\alpha}$ such that, if $D_K(t) = \{ p \in C_a^3(t) \mid \|p- \gamma_{k_2+1}(t)\| \le K.t^{\alpha}\}$ and $D_k=\bigcup_{t>0}D_K(t) \cup \{0\}$, then $(D_K,0)$ contains the germs of the $\theta$-kneading envelopes used to construct the ambient bi-Lipschitz map from $(Y,0)$ to $(\overline{\gamma_1 \gamma},0)$.

\begin{figure}[h]
\centering
\includegraphics[width=10cm]{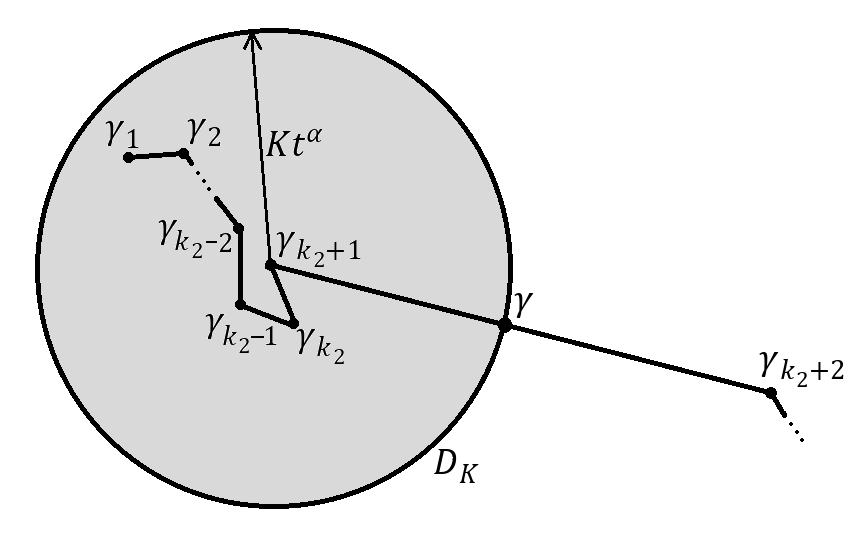}
\label{fig27}
\caption{Resolution of Case 2.3.1 of Proposition \ref{edge-reduction}.}
\end{figure}

Suppose that there is $p \in (X\setminus Y)(t)$ such that $p \in D_K(t)$. Then, $\|p-\gamma_{k_2+1}(t)\| \le K.t^{\alpha}$, but from $tord(\gamma_{k_2+1} , \gamma_{k_2+2})< \alpha$ we have $$d_{X(t)}(p,\gamma_{k_2+1}(t)) \ge \| \gamma_{k_2+1}(t) - \gamma_{k_2+2}(t)\| > C.K.t^{\alpha}.$$ 
This is a contradiction, and thus the $\theta$-kneading envelopes used to transform $(Y,0)$ into $(\overline{\gamma_1 \gamma},0)$ does not intersect $(X \setminus Y,0)$. Therefore $(X,0)$ is ambient bi-Lipschitz equivalent to the germ of the open polygonal surface germ $(\gamma_1, \gamma, \gamma_{k_2+2}, \dots, \gamma_n)$, and the result follows by induction.

\textit{Case 2.3.2:} $k_2 = n-1$. This case is analogous to Case 2.3.1, by just renaming $\gamma_1,\dots,\gamma_n$ to $\gamma_n,\dots,\gamma_1$, respectively.

\textit{Case 2.3.3:} $1 < k_1 \le k_2 < n-1$. Assume without loss of generality $\alpha_{k_1-1} \le \alpha_{k_2+1}$ (otherwise rename $\gamma_1,\dots,\gamma_n$ to $\gamma_n,\dots,\gamma_1$, respectively). If $k_1 = k_2$, then by Lemma \ref{reduction-lemma2}, item $(2)$, $(X,0)$ is ambient bi-Lipschitz equivalent to the open polygonal surface germ $(\gamma_1,\dots, \gamma_{k_1}, \gamma_{k_1 +2},\dots,\gamma_n)$ and the result follows by induction. If $k_1 < k_2$, for each $t>0$, let $\gamma(t) \in \overline{\gamma_{k_2+1}(t) \gamma_{k_2 +2}(t)}$ and $\tilde{\gamma}(t) \in \overline{\gamma_{k_1-1}(t) \gamma_{k_1}(t)}$ such that
$$\| \gamma_{k_2+1}(t) - \gamma(t) \|=\| \gamma_{k_1}(t) - \tilde{\gamma}(t) \|=a_{\alpha}t^{\alpha}.$$
Notice that $\gamma(t), \tilde{\gamma}(t)$ are well defined for all $t>0$ small enough, since $\alpha_{k_2 +1}, \alpha_{k_1-1 } < \alpha$. Let also $\gamma = \gamma(t)$, $\tilde{\gamma} =\tilde{\gamma}(t)$ and let $Z$ be the open polygonal surface $(\tilde{\gamma},\gamma_{k_1}, \dots, \gamma_{k_2+1}, \gamma)$. By Case 1, $(Z,0)$ is ambient bi-Lipschitz equivalent to $\overline{\tilde{\gamma} \gamma}$. Furthermore, similarly to Case 2.3.1, there exists $K>a_{\alpha}$ such that $(D_K,0)$ contains the germs of the $\theta$-kneading envelopes used to construct the ambient bi-Lipschitz map from $(Z,0)$ to $( \overline{\tilde{\gamma}\gamma},0)$.

\begin{figure}[h]
\centering
\includegraphics[width=13cm]{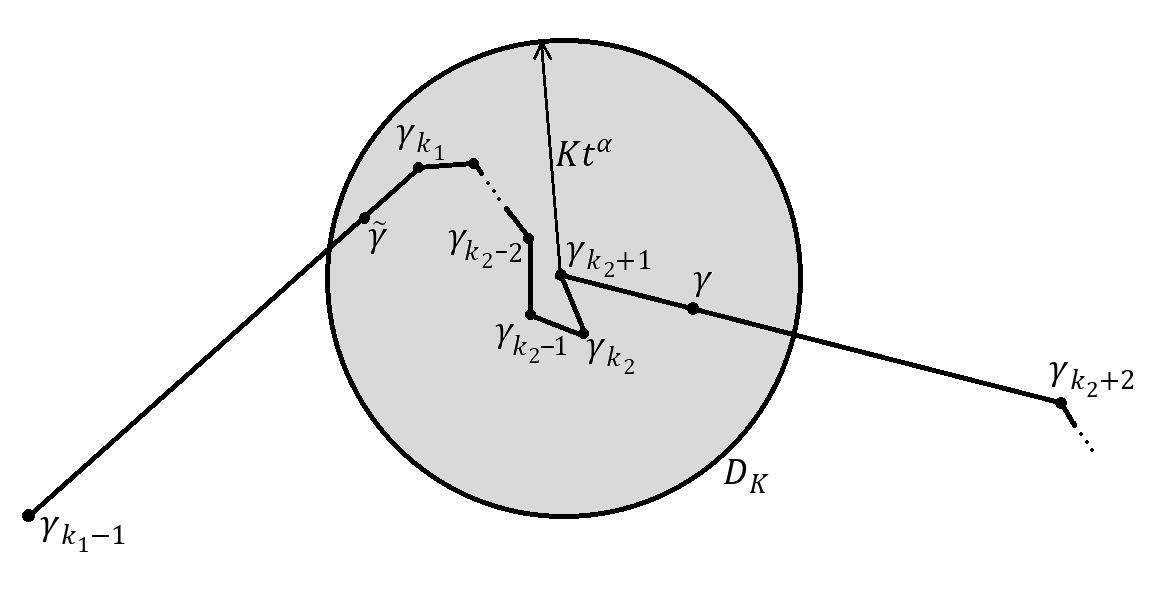}
\label{fig28}
\caption{Resolution of Case 2.3.3 of Proposition \ref{edge-reduction}.}
\end{figure}

Suppose there is $p \in (X \setminus Z)(t)$ such that $p \in D_K(t)$. Then, $\|p-\gamma_{k_2+1}(t)\| \le K.t^{\alpha}$. On the other hand, for small enough $t$,
$$\| \gamma_{k_2+1}(t) - \gamma_{k_2+2}(t)\|, \|\gamma_{k_1-1}(t) - \gamma_{k_1}(t) \| > C.Kt^{\alpha}.$$

Since $tord(\gamma_{k_2+1} , \gamma_{k_2+2}), tord(\gamma_{k_1 -1} , \gamma_{k_1})< \alpha$ and the path with minimum length in $X(t)$ connecting $p$ to $\gamma_{k_2+1}(t)$ contains the segment $\overline{\gamma_{k_1-1}\gamma_{k_1}} (t)$ or the segment $\overline{\gamma_{k_2+1}\gamma_{k_2+2}}(t)$, it follows that $d_{X(t)}(p,\gamma_{k_2+1}(t)) > C.K.t^{\alpha} \ge C. \|p-\gamma_{k_2+1}(t)\| $, a contradiction. Hence, the $\theta$-kneading envelopes used to construct the ambient bi-Lipschitz map from $(Z,0)$ to $(\overline{\tilde{\gamma} \gamma},0)$ does not intersect $(X \setminus Z,0)$. Therefore $(X,0)$ is ambient bi-Lipschitz equivalent to the germ of the open polygonal surface $(\gamma_1, \dots, \tilde{\gamma}, \gamma, \gamma_{k_2+2}, \dots, \gamma_n)$, and the result follows by induction.

\textit{Case 2.4:} $X(t) \simeq \mathbb{S}^1$. Let $\alpha = \max\{\alpha_1,\dots,\alpha_{n}\}$ and $\tilde{\alpha} = \min\{\alpha_1,\dots,\alpha_{n}\}$. For $i=1,\dots, n$, there is $a_i>0$ such that $\|\gamma_i(t) - \gamma_{i+1}(t)\| = a_i.t^{\alpha_i}+o(t^{\alpha_i})$. If $j \in \{1,\dots,n\}$ is such that $\tilde{\alpha} = \alpha_j$, since
$$\|\gamma_j(t) - \gamma_{j+1}(t)\| \ge \sum_{i \ne j} \|\gamma_i(t) - \gamma_{i+1}(t)\| \Rightarrow a_j.t^{\tilde{\alpha}} +o(t^{\tilde{\alpha}}) \ge \sum_{i \ne j} (a_i.t^{\alpha_i} +o( t^{\alpha_i})),$$
there are $i \ne j$ such that $\alpha_i = \tilde{\alpha}$. If $i = j+1$, $\angle \gamma_j \gamma_{j+1} \gamma_{j+2} >0$ implies that $tord(\gamma_j, \gamma_{j+2})= \tilde{\alpha}$, and thus there is $b>0$ such that $\|\gamma_j(t) - \gamma_{j+2}(t)\| = b.t^{\tilde{\alpha}} + o(t^{\tilde{\alpha}})$. Since we have
$$\|\gamma_j(t) - \gamma_{j+2}(t)\| \ge \sum_{i \ne j,j+1} \|\gamma_i(t) - \gamma_{i+1}(t)\| \Rightarrow b.t^{\tilde{\alpha}} +o(t^{\tilde{\alpha}}) \ge \sum_{i \ne j,j+1} (a_i.t^{\alpha_i} + o(t^{\alpha_i})),$$
it follows that there exists $k \ne \{j,j+1\}$ such that $\alpha_k = \tilde{\alpha}$. If $i=j-1$, analogously there is $k \ne \{j,j-1\}$ such that $\alpha_k = \tilde{\alpha}$. In any case, we can consider indices $1 < k_1 \le k_2 < n-1$ such that $\alpha_{k_1} = \dots = \alpha_{k_2} = \alpha$, $\alpha_{k_1-1}<\alpha$ and $\alpha_{k_2+1}<\alpha$, labeling $\gamma_1,\dots,\gamma_n$ as $\gamma_{1+l},\dots, \gamma_{n+l}$, if needed (indices modulo $n$). The rest of the proof is completely analogous to what was done in Case 2.3.3.
\end{proof}

\begin{theorem}\label{main}
	Let $(X,0) \subset (\mathbb{R}^3, 0)$ be a Lipschitz normally embedded semialgebraic surface germ.
	\begin{enumerate}
		\item If the link of $X$ is homeomorphic to $[0,1]$, then $(X,0)$ is ambient bi-Lipschitz equivalent to the germ of a standard $\alpha$-H\"older triangle embedded in $ \mathbb{R}^3$, with principal vertex at the origin, for some $\alpha \in \mathbb{Q}_{\ge 1}$.
		\item If the link of $X$ is homeomorphic to $\mathbb{S}^1$, then $(X,0)$ is ambient bi-Lipschitz equivalent to the standard $\beta$-horn germ $(H_\beta,0)$, for some $\beta \in \mathbb{Q}_{\ge 1}$.
	\end{enumerate}
\end{theorem}

\begin{proof}
	By Proposition \ref{suficiencia-em-cone}, it is enough to prove ambient bi-Lipschitz equivalence for $(X,0)\subset (C_a^3,0)$, $a>0$. From Proposition \ref{linear-decomposition}, $(X,0)$ is ambient bi-Lipschitz equivalent to a LNE polygonal surface germ in $(C_a^{3},0)$.
	\begin{itemize}
		\item If $X(t) \simeq [0,1]$, by Proposition \ref{edge-reduction}, $(X,0)$ is ambient bi-Lipschitz equivalent to the germ of an open $1$-gonal, and by Proposition \ref{linear-triangle-is-holder}, $(X,0)$ is ambient bi-Lipschitz equivalent to the germ of a standard $\alpha$-H\"older triangle embedded in $\mathbb{R}^3$, with principal vertex at the origin, for some $\alpha \in \mathbb{Q}_{\ge 1}$.
		\item If $X(t) \simeq \mathbb{S}^1$, by Proposition \ref{edge-reduction}, $(X,0)$ is ambient bi-Lipschitz equivalent to the germ of a closed $3$-gonal surface, and by Proposition \ref{linear-triangle-is-horn}, $(X,0)$ is ambient bi-Lipschitz equivalent to the germ of the standard $\beta$-horn $(H_\beta,0)$, for some $\beta \in \mathbb{Q}_{\ge 1}$.
	\end{itemize}
\end{proof}

\section{Germs of Surfaces With Disconnected Link}

In this section we are going to prove the main result of the paper.

\begin{theorem}\label{non-connected} Two LNE surface germs in $\R^3$ with isolated singularity are ambient bi-Lipschitz equivalent if and only if they are outer bi-Lipschitz equivalent and ambient topologically equivalent.
\end{theorem}

In the previous sections we proved this result for surfaces with connected link. Now we are going to prove the same result for the surfaces with disconnected link. We need first some additional definitions and propositions.

\begin{definition}\label{separating-cones}\normalfont
	Let $(X,0) \subset (\R^3,0)$ be a germ of a semialgebraic surface, with isolated singularity at the origin. Let $\{C_i\}_{i\in I}$ be a finite family of simple, closed curves in $\mathbb{S}^2 \subset \R^3$ and let $\tilde C_i$ be the cone, at $0$, of $C_i$, for each $i \in I$. We say that $\{ \tilde C_i\}_{i\in I}$ is a \emph{separating family of cones} for $(X,0)$ if it satisfies the following conditions, for each $t>0$ small enough:
	
	\begin{enumerate}
		\item for any two connected components $(X_1,0)$ and $(X_2,0)$ of $(X \setminus \{0\},0)$, the sets $(X_1)_t$ and $(X_1)_t$ belongs to different connected components of $\mathbb{S}^2_{t} \setminus \left(\cup_{i\in I} \tilde C_i\right)$.
		
		\item The family $\{(\tilde C_i,0 )\}_{i\in I}$ is ambient topologically equivalent, in $(\mathbb{R}^3,0)$, to the family $\{X_i\}_{i\in I}$ of the connected components of $(X \setminus \{0\},0)$.
	\end{enumerate}
\end{definition}

\begin{definition}\label{canonical-tree}\normalfont
	Let $(X,0) \subset (\R^3,0)$ be a LNE semialgebraic surface germ, with isolated singularity at the origin. 	For $t>0$ small enough, $X_t \subset \mathbb{S}^2_{t}$ is a finite family of simple, closed curves. The
	\normalfont \emph{canonical tree of $(X,0)$} is the graph $\tilde G(X)$ defined as follows: 
	\begin{itemize}
		\item the vertices of $\tilde G(X)$ are the connected components of $S_{t} \setminus X$;
		\item two vertices are connected by an edge if, and only if, its corresponding connected components have a common boundary curve in $X_t$.
	\end{itemize}
\end{definition}

\begin{remark}\label{canonical-tree-separating} The canonical tree of $(X,0)$ does not depend on $t$, due to Theorem \ref{conical}. Since $\mathbb{S}^2_{t}$ is connected and a family of $n$ simple, closed, non intersecting curves in $\mathbb{S}^2_{t}$ determines $n+1$ regions, the graph $\tilde G(X)$ is in fact a tree. Moreover, if $\{\tilde C_i\}_{i\in I}$ is any separating family of cones for $(X,0)$, then the canonical tree of $(\cup_{i \in I}\tilde C_i,0)$ is isomorphic to the canonical tree of $(X,0)$.
\end{remark}

\begin{proposition}\label{separating-cones-prop} Let $(X,0) \subset (\R^3,0)$ be a LNE semialgebraic surface germ, with isolated singularity at the origin. Then, there is a semialgebraic separating family of cones for $(X,0)$, and each cone of this family is also LNE.
\end{proposition}

\begin{proof}
	Let $\{(X_i,0)\}_{i\in I}$ be the family of connected components of $(X \setminus \{0\},0)$ and consider $C(X,0)$ the tangent cone of $X$ at $0$. Let also $L_0(X)=C(X,0)\cap \mathbb{S}^2$ be the tangent link of $X$. Since $(X,0)$ is LNE, $L_0(X)$ is the union of connected components $L_i \in \mathbb{S}^2$, such that $C(L_i,0)$ is the tangent cone of $(X_i,0)$, for each $i \in I$. Let $U_{\varepsilon}(L_i) = (\cup_{x\in L_i}B_\varepsilon(x))\cap \mathbb{S}^2$ be the $\varepsilon$-neighborhood of $L_i$ in $\mathbb{S}^2$. The sets $L_i$ may be one-dimensional, if $(X_i,0)$ is a 1-horn, or zero-dimensional, if $(X_i,0)$ is a $\beta$-horn, for $\beta>1$. If $L_j$ is zero-dimensional, then the boundary of its $\varepsilon$-neighborhood has just one connected component; otherwise, for $\varepsilon$ small enough, it have two connected components. Take $\varepsilon>0$ small such that
	\begin{itemize}
		\item $\varepsilon < \frac{1}{2} \cdot min\{d(L_j, L_i) \; ; \; i,j \in I; \; i\ne j\}$;
		\item For each $i \in I$ and $0<\varepsilon^{\prime}<\varepsilon$, $U_{\varepsilon}(L_i)$ and $U_{\varepsilon^{\prime}}(L_i)$ are homeomorphic and admits $L_i$ as a deformation retract in $\mathbb{S}^2$ (In particular, each boundary curve of $U_{\varepsilon}(L_i)$ is isotopic in $\mathbb{S}^2$ to $L_i$).
	\end{itemize}
	
	For each $i \in I$, the connected component $(X_i,0)$ belongs to $C(U_{\varepsilon}(L_i),0)$. Moreover, since $(X,0)$ is LNE and $\varepsilon < \frac{1}{2} \cdot d(L_i, L_j)$, the tangent cones of different components do not intersect and the cones over the $\varepsilon$-neighborhoods $U_{\varepsilon}(L_i)$ also do not intersect. Let also $\{M_{\ell}\}_{\ell \in L}$ be the set of all boundary arcs of the ${U_{\varepsilon}(L_i)}_{i\in I}$. Choose a point $x_0 \in \mathbb{S}^2 \setminus ((\cup_{\ell \in L}M_\ell)\cup (\cup_{i \in I}L_i))$. For each $i \in I$, we say that a boundary  curve $\tilde M_{\ell}$ of the $\varepsilon$-neighborhood of $L_i$ is external with respect to $x_0$ if there exists a path in $\mathbb{S}^2$ connecting $x_0$ with $M_{\ell}$, that does not intersect $L_i$. It is clear that, for each $x_0$ and $i \in I$, there is exactly one boundary curve $\tilde M_\ell$ of $U_{\varepsilon}(L_i)$ that is external with respect to $x_0$.
	
	Let $\{C_i\}_{i \in I}$ be the collection of all the curves from $\{M_{\ell}\}_{\ell\in L}$ which are external with respect to $x_0$ and let $\tilde C_i$ be the cone, at $0$, of $C_i$, for each $i \in I$. Now it is straightforward to see that ${\tilde C_i}$ satisfies all the conditions of Definition \ref{separating-cones}. 
\end{proof}

\begin{corollary}\label{canonical-tree-topol-inv} Let $(X,0), (Y,0) \subset (\R^3,0)$ be LNE semialgebraic surface germs, each one with isolated singularity at the origin. Then, $(X,0)$ is ambient topologically equivalent to $(Y,0)$ if, and only if, $\tilde G(X)$ is isomorphic to $\tilde G(Y)$.
\end{corollary}

\begin{proof}
	If $\{\tilde C_i\}_{i\in I}$ and $\{\tilde C_i^{\prime}\}_{i\in I}$ are separating families of cones for $(X,0)$ and $(Y,0)$, respectively, it is enough to show that $\cup_{i\in I} \tilde C_i$ and $\cup_{i\in I} \tilde C_i^{\prime}$ are ambient topologically equivalent if, and only if, $\tilde G(\cup_{i\in I} \tilde C_i)$ and $\tilde G(\cup_{i\in I} \tilde C_i^{\prime})$ are isomorphic. For each $i\in I$ and $t>0$ small enough, let $C_i = \tilde C_i \cap \mathbb{S}^2_{t}$ and $C_i^{\prime} = \tilde C_i^{\prime} \cap \mathbb{S}^2_{t}$. Then $\cup_{i\in I} \tilde C_i$ and $\cup_{i\in I} \tilde C_i^{\prime}$ are ambient topologically equivalent if, and only if, $\cup_{i\in I} C_i$ and $\cup_{i\in I} C_i^{\prime}$ are ambient topologically equivalent in $\mathbb{S}^2$, because $\tilde C_i$ and $\tilde C_i^{\prime}$ are cones. By Definition \ref{canonical-tree}, the result follows.
\end{proof}

\begin{proposition}\label{amb-bi-lip-local-link}
	Let $C \subset \mathbb{S}^2$ be a LNE, semialgebraic curve, homeomorphic to $\mathbb{S}^1$, and let $\tilde C$ be the cone, at $0$, of $C$. Let $(X_1,0), (X_2,0) \in \mathbb{R}^3$ be two LNE surface germs with link homeomorphic to either $[0,1]$ or $\mathbb{S}^1$ such that:
	\begin{enumerate}
		\item $(X_1,0)$ is outer bi-Lipschitz equivalent to $(X_2,0)$;
		\item $tord(X_1,\tilde C)=tord(X_2,\tilde C)= 1$;
		\item $(X_1\setminus \{0\},0)$ and $(X_2\setminus \{0\},0)$ are in the same connected component $(T,0)$ of $(\R^3 \setminus \tilde C,0)$.
	\end{enumerate}
	Then, there is a bi-Lipschitz map $\varphi: (\R^3,0) \to (\R^3,0)$ such that $\varphi|_{(\R^3 \setminus T)} = id_{(\R^3 \setminus T)}$ and $\varphi(X_1,0)=(X_2,0)$.
\end{proposition}

\begin{proof}
	Since $\tilde C$ is a cone, by Theorem \ref{main}, for each $a>0$, there is a bi-Lipschitz map $\varphi_0: (\R^3,0) \to (\R^3,0)$ such that $\varphi_0(\overline{T})=(C_a^3,0)$ and $\varphi(\tilde C,0) = (\partial C_a^3,0)$.  By conditions (2) and (3), there is $a^{\prime}<a$ such that $\varphi_0(X_1,0),\varphi_0(X_2,0) \subset (C_{a^{\prime}}^3,0)$. Choose $a$ large enough such that $a^{\prime}>1$. By condition (1) and Theorem \ref{main}, there is $\theta \in \Q_{\ge 1}$ and bi-Lipschitz maps $\varphi_i : (C_{a^{\prime}}^3,0) \to (C_{a^{\prime}}^3,0)$ ($i=1,2$)
	such that $\varphi_i|_{\partial C_{a^{\prime}}^3}=id_{\partial C_{a^{\prime}}^3}$ and $\varphi_i\circ \varphi_0 (X_i,0)$ is a $\theta$-H\"older triangle, if the link of $X_i$ is homeomorphic to $[0,1]$, or $\varphi_i\circ \varphi_0 (X_i,0)$ is a $\theta$-horn, if the link of $X_i$ is homeomorphic to $\mathbb{S}^1$.
	
	Define $\tilde \varphi: (C_a^3,0) \to (C_a^3,0)$ and $\varphi: (\R^3,0) \to (\R^3,0)$ as:
	$$\tilde \varphi(p) =
	\begin{cases}
		p, & p \notin (C_{a^{\prime}}^3,0) \\
		\varphi_2^{-1} \circ \varphi_1 (p), & p \in (C_{a^{\prime}}^3,0) 
	\end{cases},$$
	$$\varphi(p) =
	\begin{cases}
		p, & p \notin (\overline{T},0) \\
		\varphi_0^{-1} \circ \tilde \varphi \circ \varphi_0 (p), & p \in 	(\overline{T},0) 
	\end{cases}.$$
	
	By Proposition \ref{extensão-bi-lip-invariante-bola}, $\varphi$ is a bi-Lipschitz map. By construction, $\varphi|_{(\R^3 \setminus T)} = id_{(\R^3 \setminus T)}$ and $\varphi(X_1,0)=(X_2,0)$.
\end{proof}

\begin{proposition}\label{topol->lipeo-in-S2} Let $n \in \mathbb{N}_{\ge 1}$ and $\{C_i\}_{1\le i \le n}$, $\{C_i^{\prime}\}_{1\le i \le n}$ be two families of LNE, semialgebraic curves, each one homeomorphic to $\mathbb{S}^1$ in $\mathbb{S}^2$, such that $C_i \cap C_j=C_i^{\prime}\cap C_j^{\prime}=\emptyset$, for each $i\ne j$. If there is a homeomorphism $\varphi: \mathbb{S}^2 \to \mathbb{S}^2$ such that $\varphi(C_i)=C_i^{\prime}$, for $1\le i \le n$, then there is a bi-Lipschitz map $\tilde \varphi: \mathbb{S}^2 \to \mathbb{S}^2$ such that $\tilde \varphi(C_i)=C_i^{\prime}$, for $1\le i \le n$.
\end{proposition}

\begin{proof}
	For $i=1,\dots,n$, let $\tilde C_i$, $\tilde C_i^{\prime}$ be the tangent cones of $C_i$, $C_i^{\prime}$, respectively, and define $X=\cup_{i=1}^{n}\tilde C_i$, $X^{\prime}=\cup_{i=1}^{n}\tilde C_i^{\prime}$.
	We will prove by induction on $n$ that there is a bi-Lipschitz map $\Phi: \R^3 \to \R^3$ such that $\Phi(X)=X^{\prime}$. This, together with Corollary 0.2 of \cite{Valette-Link}, will give the desired bi-Lipschitz map $\tilde \varphi$ on the link $\mathbb{S}^2$.
	
	The base case $n=1$ is a consequence of Theorem \ref{main}, item (2), when $\beta=1$. For the inductive step, consider a vertex $v$ of the canonical tree $\tilde G(X)$ with degree 1, $u$ the vertex of $\tilde G(X)$ adjacent to $v$ and let $v'$, $u^{\prime}$ the vertices of $\tilde G(X^{\prime})$ isomorphic to $v$, $u$, respectively. Let $T_u$, $T_v$ be the connected components of $\R^3 \setminus X$ associated with $u$ and $v$, respectively, and let $T^{\prime}_u$, $T^{\prime}_v$ be the connected component of $\R^3 \setminus X^{\prime}$ associated with $u^{\prime}$ and $v^{\prime}$, respectively. Since $v$, $v^{\prime}$ have degree 1, suppose $\tilde C_n = \partial T_v$ and $\tilde C_n^{\prime} = \partial T^{\prime}_v$. Let also $T=T_u \cup \tilde C_n \cup T_v$ and $T'=T'_u \cup \tilde C'_n \cup T'_v$.
	
	By hypothesis, there is a bi-Lipschitz map $\Phi_0: \R^3 \to \R^3$ such that $\Phi_0(X\setminus(\tilde C_n \setminus\{0\})) = X^{\prime}\setminus(\tilde C_n^{\prime} \setminus\{0\})$. By Corollary \ref{canonical-tree-topol-inv}, $\tilde G(X)$ and $\tilde G(X^{\prime})$ are isomorphic and then $\Phi_0(X\setminus (T\setminus \{0\}))=X'\setminus (T'\setminus \{0\})$, $\Phi_0(T)=T'$ (in particular, $\Phi_0(\tilde C_n), \tilde C_n' \subset T'$). Now let $\tilde C, \tilde C'\subset T'$ be cones, at $0$, of curves homeomorphic to $\mathbb{S}^1$ such that $\tilde C, \tilde C'$ contains $\Phi_0(\tilde C_n), \tilde C_n'$ in its interior with respect to $T'$, respectively. Such curves exists by Proposition \ref{separating-cones-prop}.
	
	By Proposition \ref{amb-bi-lip-local-link}, for each $\varepsilon>0$ small enough, there are bi-Lipschitz maps $\Phi_1, \Phi'_1: \R^3 \to \R^3$ such that $\Phi_1 \circ \Phi_0(\tilde C_n)$ and $\Phi'_1(\tilde C_n')$ are cones of circles of radius $\varepsilon$ in $T'\cap \mathbb{S}^2$ and $\Phi_1|_{(\R^3 \setminus T')}=\Phi'_1|_{(\R^3 \setminus T')}=id_{(\R^3 \setminus T')}$. Choose $\varepsilon>0$ small enough such that there is a semialgebraic curve $D\subset T'\cap \mathbb{S}^2$ homeomorphic to $\mathbb{S}^1$ containing $\Phi_1 \circ \Phi_0(\tilde C_n) \cap \mathbb{S}^2$ and $\Phi'_1(\tilde C'_n) \cap \mathbb{S}^2$ in its interior (relative to $T'$). One way to construct such curve is to connect the centers of such circles of radius $\varepsilon$ by arcs of maximum circles in $T'$, then consider the boundary of points in $T'$ with distance at most $3\varepsilon/2$ of the union of those arcs. 
	
	By Proposition \ref{amb-bi-lip-local-link} there is a bi-Lipschitz map $\Phi_2: \R^3 \to \R^3$ such that $\Phi_2(\Phi_1 \circ \Phi_0(\tilde C_n))=\Phi'_1(\tilde C'_n)$ and $\Phi_2|_{(\R^3 \setminus T')}=id_{(\R^3 \setminus T')}$. Therefore, $\Phi=(\Phi_1')^{-1}\circ\Phi_2\circ\Phi_1\circ\Phi_0$ satisfies $\Phi(\tilde C_n)=\tilde C'_n$ by construction and $\Phi(\tilde C_i)=\tilde C'_i$ for $i<n$, by hypothesis, and thus the result follows by induction.
	\begin{figure}[h]
		\centering
		\includegraphics[width=15cm]{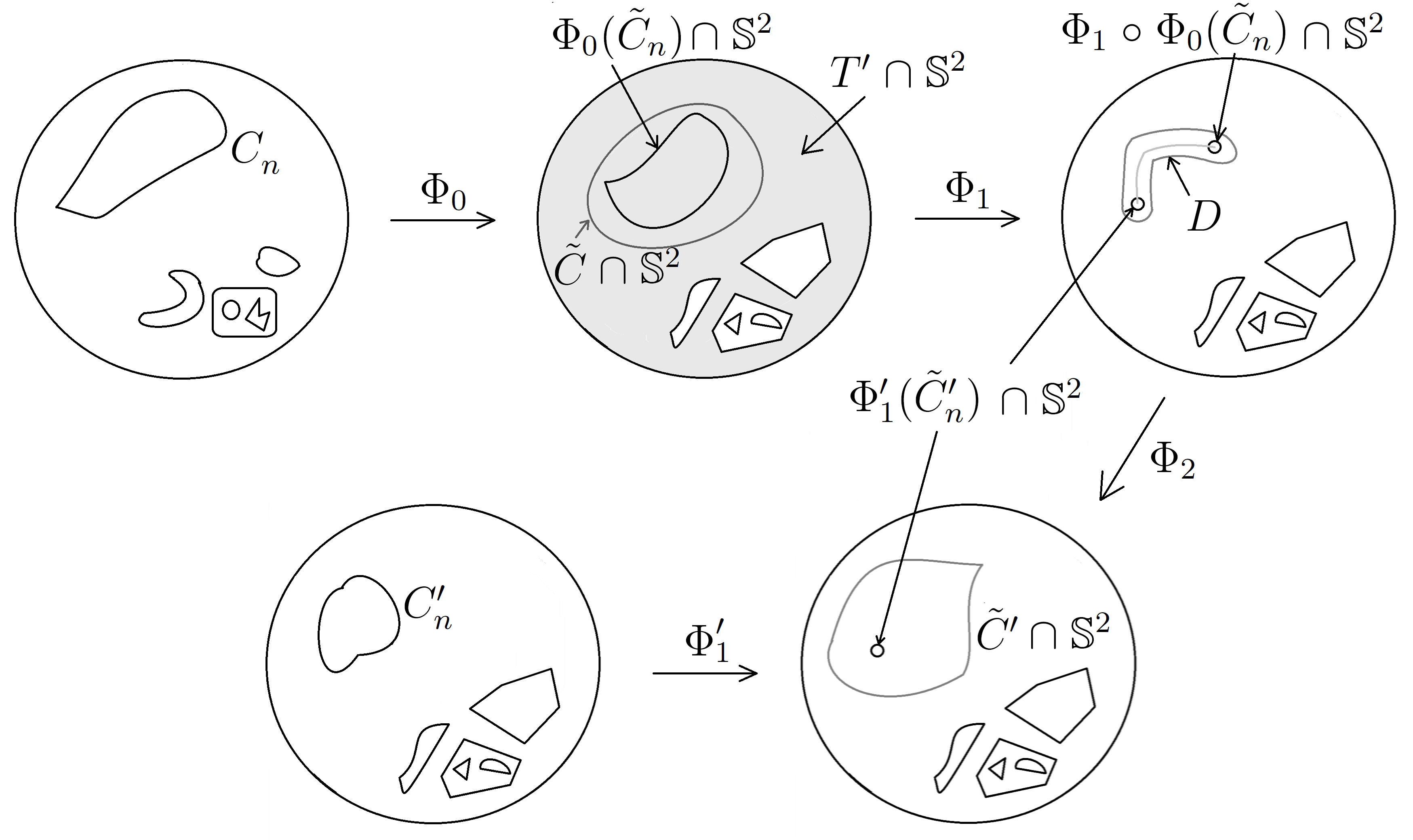}
		\label{fig-fig}
		\caption{Proof of Proposition  \ref{topol->lipeo-in-S2}.}
	\end{figure}
\end{proof}

\begin{definition}\label{extended-canonical}\normalfont 
	Let $(X,0) \subset (\R^3,0)$ be a LNE, semialgebraic surface germ, with isolated singularity at the origin. Consider $V(\tilde G(X))$ the set of all vertices of $\tilde G(X)$. The \emph{extended canonical tree of $X$} is the graph $\tilde G(X)$ together with a function $\rho : V(\tilde G(X)) \to \mathbb{Q}_{\ge 1}$ defined as follows: for each $t>0$ small enough, if $v$ the vertex associated with the connected component $X_v \cap \mathbb{S}_t^2$ of $\mathbb{S}^2_{t} \setminus X$, and if $\beta \in \mathbb{Q}_{\ge 1}$ is the only rational number such that the limit $\lim\limits_{\varepsilon \to 0^+}\frac{Area(X_v(\varepsilon))}{t^{\beta}} \in \mathbb{R}_+^*$, then $\rho(v)=\beta$. 
\end{definition}

\begin{remark}\label{extended-canonical-remark}
	The extended canonical tree is the combinatorial object containing all information needed for the classification of LNE surface germs in $\R^3$ with isolated singularity at the origin, up to ambient bi-Lipschitz equivalence. The tree contains the ambient topology information (Corollary \ref{canonical-tree-topol-inv}) and the function $\rho$ contains the exponent of the corresponding horn, thus describing the outer Lipschitz geometry. Therefore, two such germs are ambient topologically equivalent and outer bi-Lipschitz equivalent if, and only if, they have isomorphic extended canonical trees. Finally, for a proof of existence and uniqueness of $\beta$, see \cite{Bir-Bras}.
\end{remark}
	
\begin{theorem}\label{classification} Two LNE surface germs $(X,0)$ and $(Y,0)$ with isolated singularity at the origin are ambient bi-Lipschitz equivalent in $(\R^3,0)$ if, and only if, the corresponding extended canonical trees are isomorphic. In other words, Theorem \ref{non-connected} holds.
\end{theorem}

\begin{proof}
	The direct implication was discussed in Remark \ref{extended-canonical-remark}, so let us prove the converse. Let $(X,0)=(X_1,0)\cup \ldots \cup (X_n,0)$, $(Y,0)=(Y_1,0)\cup \ldots \cup (Y_n,0)$, with $(X_i,0)$, $(Y_i,0)$ being $\beta_i$-horns, for $i=1,\ldots,n$, and $\{0\}=(X_i \cap X_j,0) = (Y_i \cap Y_j,0)$, for $i\ne j$. Let also $\{\tilde C_i\}_{1\le i \le n}$ and $\{\tilde C_i'\}_{1\le i \le n}$ be separating families of cones for $(X,0)$ and $(Y,0)$, respectively, with $(X_i,0)$, $(\tilde C_i,0)$, $(\tilde C_i',0)$ and $(Y_i,0)$ being ambient topologically equivalent, for each $i\in \{1,\dots,n\}$. By Theorem \ref{main}, each $(X_i,0), (Y_i,0)$ with $\beta_i=1$ is ambient bi-Lipschitz equivalent to their tangent cone, and therefore we can suppose, after a ambient bi-Lipschitz map, that $(X_i,0)=(\tilde C_i,0)$, $(Y_i,0) = (\tilde C_i',0)$ for all such $i \in \{1,\dots,n\}$.
	
	By Proposition \ref{topol->lipeo-in-S2}, there is a bi-Lipschitz map $\Phi: \R^3 \to \R^3$ such that $\Phi(\tilde C_i)=\tilde C_i'$, for every $i$, if we set $\Phi(0)=0$ and $\Phi(tx)=t\tilde \varphi(x)$, for $x \in \mathbb{S}^2$ and $t>0$. If $k\in \{1,\dots,n\}$ is such that $\beta_k>1$, the LNE condition on $(X,0)$ and $(Y,0)$ implies that each $\Phi(X_k)$, $Y_k$ is in the same connected component $T_k$ of $\R^3 \setminus (\cup_{i=1}^{n} \tilde C'_i)$ whose boundary is $\tilde C_k$ and $T_k\cap \mathbb{S}^2$ is homeomorphic to an open disk. By Proposition \ref{amb-bi-lip-local-link}, there is a ambient bi-Lipschitz map $\varphi_k: (\R^3,0) \to (\R^3,0)$ such that $\varphi|_{(\R^3\setminus T_k)}=id_{(\R^3\setminus T_k)}$ and $\varphi_k(\Phi(X_k,0))=(Y_k,0)$.
	
	Finally, if we consider the map $\varphi: (\R^3,0) \to (\R^3,0)$ as
	$$ \varphi(p) =
	\begin{cases}
		\varphi_k \circ \Phi(p), & p \in (T_k,0) \; (\beta_k >1) \\
		\Phi(p), & p \in \left(\R^3 \setminus \left(\bigcup_{\beta_k>1}T_k\right),0\right) 
	\end{cases},$$
	then $\varphi$ is a well-defined bi-Lipschitz map, by Proposition \ref{colagem-bi-lip-inner}, such that $\varphi(X_i,0)=\Phi(X_i,0)=(Y_i,0)$, for all $i \in \{1,\dots,n\}$ such that $\beta_i=1$, and $\varphi(X_k,0)=\varphi_k(\Phi(X_k,0))=(Y_k,0)$, for all $k \in \{1,\dots,n\}$ such that $\beta_k>1$. Therefore, $(X,0)$ and $(Y,0)$ are ambient bi-Lipschitz equivalent.
\end{proof}

\begin{remark}\label{important-example}\normalfont
	For germs of LNE surfaces with non-isolated singularity in $\R^3$, Theorem \ref{non-connected} is not true in general. For a counterexample, consider, for each $0<t<1$,
	$$X_1(t)=\overline{A(t)B(t)} \cup \overline{B(t)C(t)} \cup \overline{C(t)A(t)} \cup \overline{A(t)D_1(t)}$$
	$$X_2(t)=\overline{A(t)B(t)} \cup \overline{B(t)C(t)} \cup \overline{C(t)A(t)} \cup \overline{A(t)D_2(t)}$$
	$$X_1 = (\cup_{t>0} X_1(t))\cup \{0\} \; ; \; X_2 = (\cup_{t>0} X_2(t))\cup \{0\}$$
	
	where:
	
	$$A(t) = (0,0,t) \; ; \; B(t) =(t^2,t^2,t) \; ; \; C(t)=(t^2,-t^2,t)$$
	$$D_1(t) = (t^3,0,t)\; ; \;D_2(t) = (-t^3,0,t)$$
	
	\begin{figure}[h]
		\centering
		\includegraphics[width=8cm]{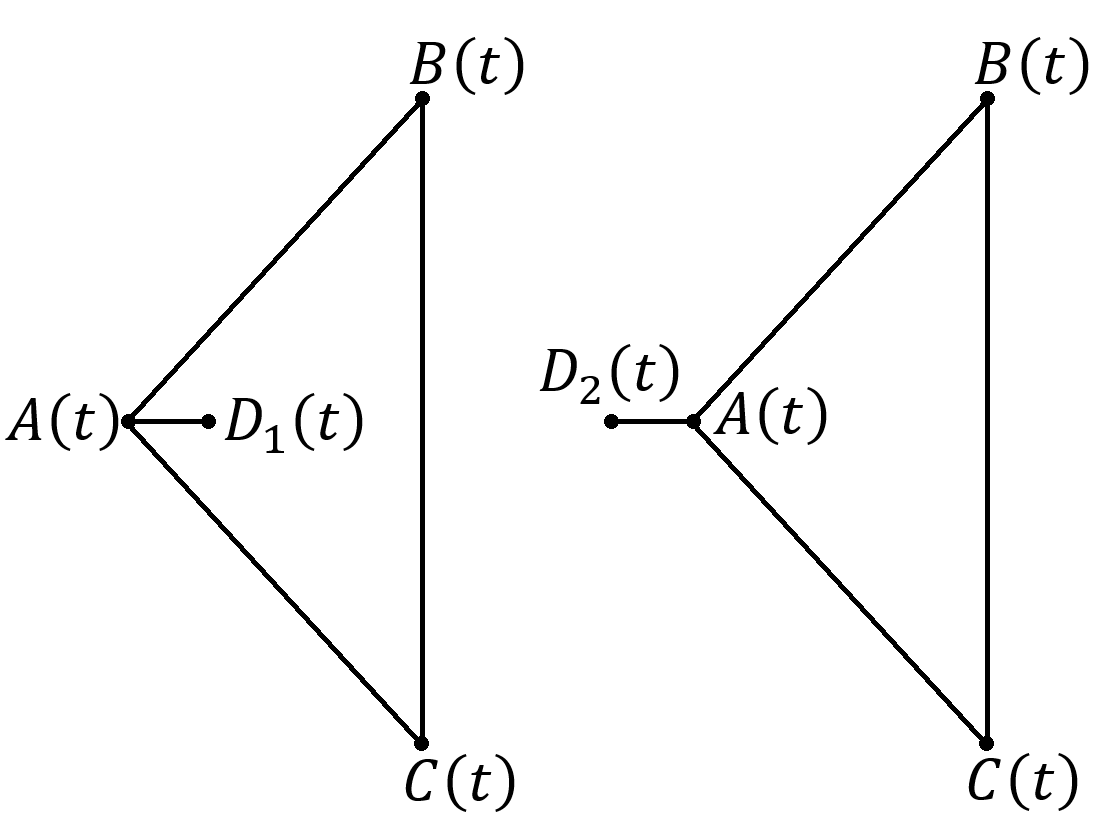}
		\label{fig-fig}
		\caption{The links of $X_1$ and $X_2$ in Remark \ref{important-example}.}
	\end{figure}
\end{remark} 

The following proposition is devoted to prove that $(X_1,0)$ and $(X_2,0)$ are indeed counterexamples to Theorem \ref{non-connected}.

\begin{proposition}\label{important-example-proposition}
	The surfaces germs $(X_1,0)$ and $(X_2,0)$ are ambient topological equivalent, outer bi-Lipschitz equivalent, but not ambient bi-Lipschitz equivalent.
\end{proposition}

\begin{proof}
	Since the corresponding H\"older complexes of $(X_1,0)$ and $(X_2,0)$ are combinatorially equivalent, $(X_1,0)$ and $(X_2,0)$ are inner bi-Lipschitz equivalent. For each $0<t< 1$, $X_1(t)$ and $X_2(t)$ are finite unions of segments, with the angles between any two of them in the interval $[\frac{\pi}{6},\frac{5\pi}{6}]$. Then, $X_1(t)$ and $X_2(t)$ are $C$-LNE, for some $C>0$. By Theorem \ref{Edson-Rodrigo}, $(X_1,0)$ and $(X_2,0)$ are LNE, and thus they are also outer Lipschitz equivalent. 
	
	For each $\varepsilon>0$ small enough and $i=1,2$, consider the set $Y_i(\varepsilon)$ as the region of $\mathbb{S}^2_{\varepsilon}$ bounded by $X_i \cap \mathbb{S}^2_{\varepsilon}$, not containing the interior of the "tail" $\cup_{t>0}\overline{A(t)D_i(t)} \cap \mathbb{S}^2_{\varepsilon}$, and let $Y_i = \{0\}\cup\left(\cup_{0<\varepsilon < 1}Y_i(\varepsilon)\right)$. It is clear that $(X_1,0)$ and $(X_2,0)$ are topologically equivalent, and for every homeomorphism $h:(X_1,0) \to (X_2,0)$, we have $h(Y_1,0)=h(Y_2,0)$. 
	
	Suppose that $(X_1,0)$ and $(X_2,0)$ are ambient Lipschitz equivalent. Then $(Y_1,0)$ and $(Y_2,0)$ are ambient Lipschitz equivalent. In particular, the volume growth numbers of $Y_1$ and $Y_2$ at the origin must be equal (see \cite{Bir-Bras}). On the other hand, since ${\rm area}(Y_1\cap \mathbb{S}^2_{t}) \approx t$ and ${\rm area}(Y_2\cap \mathbb{S}^2_{t}) \approx t^2$, the volume growth number of $Y_1$ is equal to $2\cdot 1+1=3$ and the volume growth number of $Y_2$ is equal to $2\cdot 1+1=5$, contradiction. Therefore, $(X_1,0)$ and $(X_2,0)$ are not ambient Lipschitz equivalent.
\end{proof}

%algumas bibliografias precisam ser ajeitadas, vou ver isso no final

\end{document}